\documentclass[12pt,reqno]{amsart}
\usepackage{amsmath,amsthm,amstext}
\usepackage{amsfonts,amssymb}
\usepackage[mathscr]{eucal}
\usepackage[]{hyperref}
\usepackage{setspace,url}
\usepackage{tikz}
\usepackage{epstopdf}

\newtheorem{theorem}{Theorem}
\newtheorem{lemma}{Lemma}[section]
\newtheorem{corollary}{Corollary}
\newtheorem{definition}{Definition}
\newtheorem{proposition}{Proposition}[section]
\newtheorem{remark}{Remark}[section]

\newtheorem{numerical-lemma}{Numerical Lemma}[section]

\numberwithin{equation}{section}

\newcommand{\sech}{\ensuremath{\text{sech}}}

\newcommand{\SP}{\bf SP}

\newcommand{\qmax}{q_{\textrm{max}}}

\newcommand{\bif}{\textrm{bif}}
{\begin{trivlist} \item[]{\textbf{Proof} }}%
{\hspace*{\fill}$\rule{.3\baselineskip}{.35\baselineskip}$\end{trivlist}}

\begin{document}

\title{\bf Standing waves on a flower graph}

\author[A. Kairzhan]{Adilbek Kairzhan}
\address[A. Kairzhan]{Department of Mathematics, McMaster University, Hamilton ON, Canada, L8S 4K1}
\email{kairzhaa@math.mcmaster.ca}

\author[R. Marangell]{Robert Marangell}
\address[R. Marangell]{School of Mathematics and Statistics F07, University of Sydney, NSW 2006, Australia}
\email{robert.marangell@sydney.edu.au}

\author[D.E. Pelinovsky]{Dmitry E. Pelinovsky}
\address[D.E. Pelinovsky]{Department of Mathematics, McMaster University, Hamilton ON, Canada, L8S 4K1}
\email{dmpeli@math.mcmaster.ca}

\author[K.L. Xiao]{Ke Liang Xiao}
\address[K.L. Xiao]{Department of Mathematics, McMaster University, Hamilton ON, Canada, L8S 4K1}
\email{xiaok1@mcmaster.ca}

\maketitle

\begin{abstract}
A flower graph consists of a half line and $N$ symmetric loops connected at a single vertex with $N \geq 2$
(it is called the tadpole graph if $N = 1$). We consider positive
single-lobe states on the flower graph in the framework of the cubic nonlinear Schr\"{o}dinger
equation. The main novelty of our paper is a rigorous application of the period function for
second-order differential equations towards understanding the symmetries and bifurcations
of standing waves on metric graphs. We show that the positive single-lobe symmetric state (which is the ground state
of energy for small fixed mass) undergoes exactly one bifurcation for larger mass, at which point $(N-1)$
branches of other positive single-lobe states appear: each branch has $K$ larger components
and $(N-K)$ smaller components, where $1 \leq K \leq N-1$. We show that
only the branch with $K = 1$ represents a local minimizer of energy for large fixed mass, however, the
ground state of energy is not attained for large fixed mass if $N \geq 2$.
Analytical results obtained from the period function are illustrated numerically.
\end{abstract}

\section{Introduction}

A {\bf \em flower graph} is a metric graph which consists of a half-line and $N$ symmetric loops connected at a single common vertex.
We denote such a graph by $\Gamma_N$. Without loss of generality, we normalize the length of symmetric loops
to $2 \pi$ and parameterize the loops by $[-\pi,\pi]$. The half-line coincides with $[0,\infty)$.
We count $N+1$ edges and $2$ vertices (one at infinity), so that the Betti number of $\Gamma_N$ is equal to $N$.
Figure \ref{fig-0} gives schematic examples of the flower graph for two and three loops.

\begin{figure}[htbp] %  figure placement: here, top, bottom, or page
   \centering
      \includegraphics[width=3in, height = 2in]{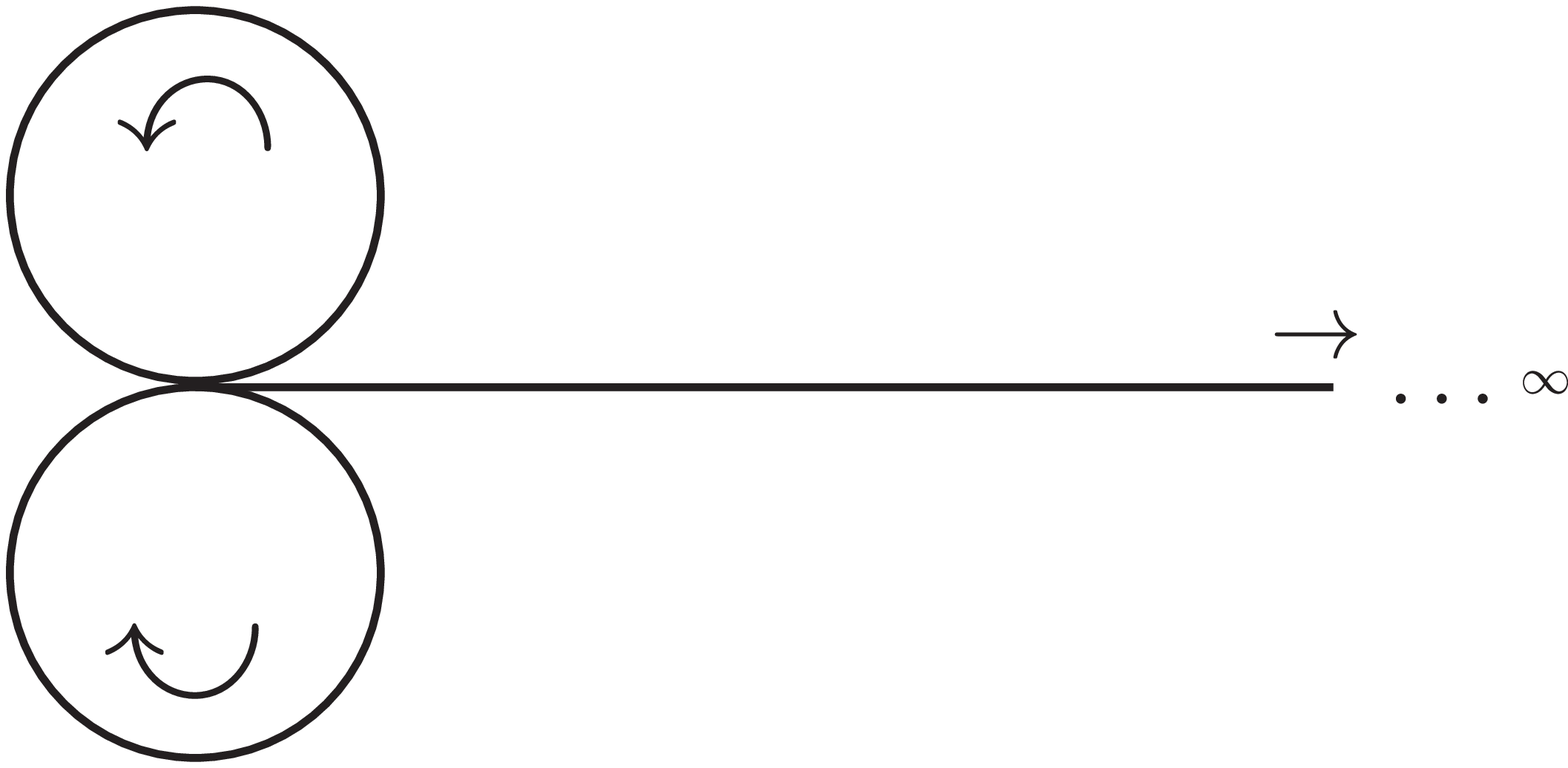}
   \includegraphics[width=3in, height = 2in]{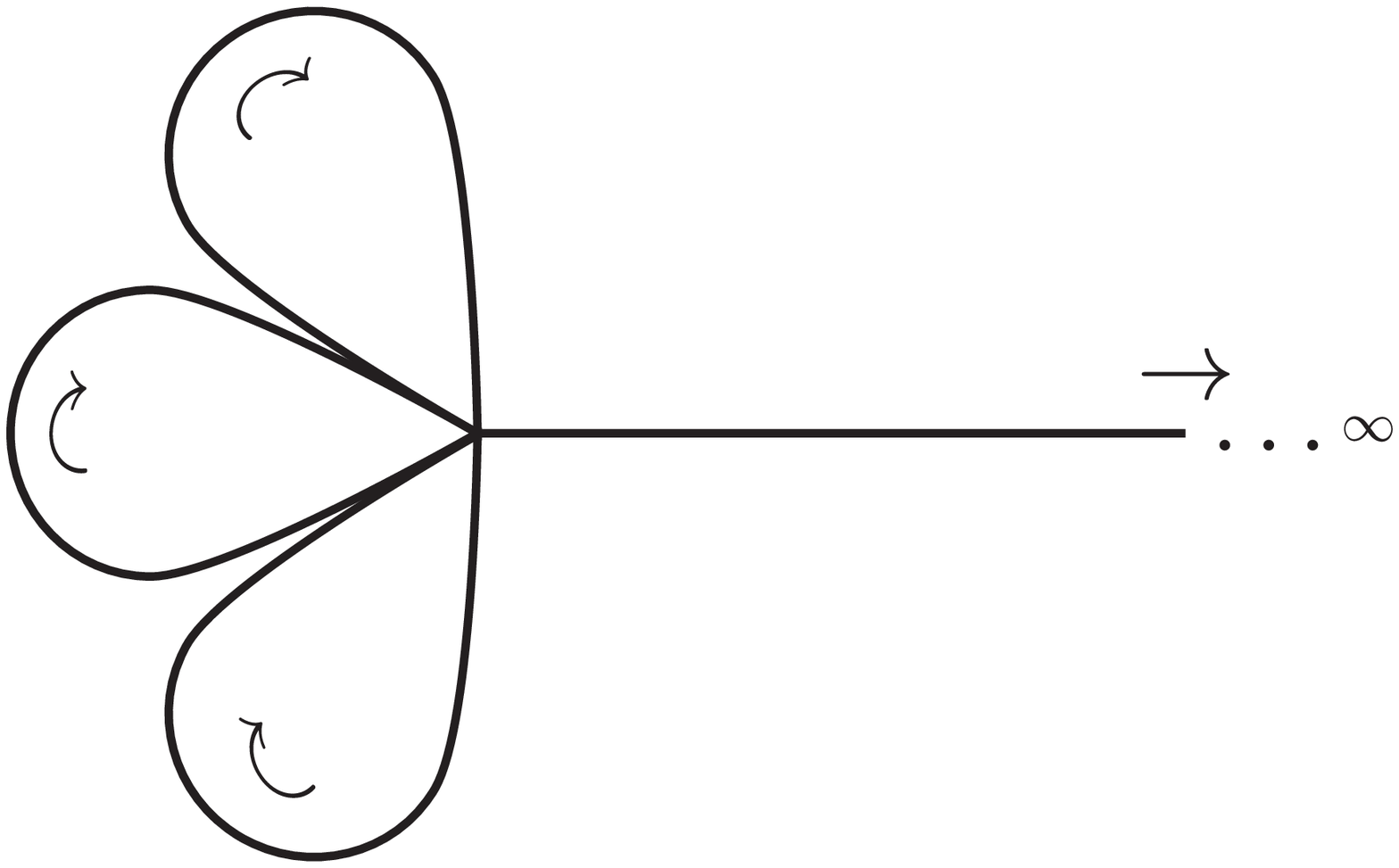}
   \caption{A schematic example of the flower graph $\Gamma_N$ with $N = 2$ (left) and $N = 3$ (right).}
   \label{fig-0}
\end{figure}

Standing waves in the nonlinear Schr\"{o}dinger (NLS) equation on metric graphs have
attracted much attention in recent years \cite{Noja}. The NLS equation with a
power nonlinearity is usually posed in the normalized form
\begin{equation}
\label{nls-time}
i \Psi_t + \Delta \Psi + (p+1) |\Psi|^{2p} \Psi = 0,
\end{equation}
where the Laplacian $\Delta$ is defined componentwise on the metric graph subject to
proper boundary conditions (see, e.g., monographs \cite{BK,Exner}).

Let the wave function $\Psi = (\psi_1,\psi_2,\dots,\psi_N, \psi_0)$ on the flower graph $\Gamma_N$
be represented by the functions $\{ \psi_j \}_{j=1}^N : [-\pi,\pi] \mapsto \mathbb{C}$ on the $N$ symmetric loops
and by $\psi_0 : [0,\infty) \mapsto \mathbb{C}$ on the half-line.
We define the space of square-integrable functions $L^2(\Gamma_N)$ componentwise as
$$
L^2(\Gamma_N) = \underbrace{L^2(-\pi,\pi) \times \dots  \times L^2(-\pi,\pi)}_\text{\rm N times} \times L^2(0,\infty)
$$
The NLS equation is locally well-posed in the energy space $H^1_C(\Gamma_N) := H^1(\Gamma_N) \cap C^0(\Gamma_N)$,
where the Sobolev space $H^1(\Gamma_N)$ is also defined componentwise as
$$
H^1(\Gamma_N) =  \underbrace{H^1(-\pi,\pi) \times \dots  \times H^1(-\pi,\pi)}_\text{\rm N times} \times H^1(0,\infty),
$$
and $C^0(\Gamma_N)$ denotes the space of continuous functions on edges of $\Gamma_N$ and across
the vertex point in $\Gamma_N$. The local solution to the NLS equation (\ref{nls-time}) conserves
the energy
\begin{equation}
\label{energy}
E(\Psi) = \| \nabla \Psi \|^2_{L^2(\Gamma_N)} - \| \Psi \|^{2p+2}_{L^{2p+2}(\Gamma_N)}
\end{equation}
and the mass
\begin{equation}
\label{mass}
Q(\Psi) = \| \Psi \|^2_{L^2(\Gamma_N)}.
\end{equation}

A {\bf \em standing wave} of the NLS equation (\ref{nls-time})
is given by the solution of the form $\Psi(t,x) = \Phi(x) e^{-i \omega t}$,
where $\Phi \in H^1_C(\Gamma_N)$ is a real-valued solution of the stationary NLS equation
\begin{equation}
\label{nls-stat-intro}
\omega \Phi =  - \Delta \Phi - (p+1) |\Phi|^{2p} \Phi,
\end{equation}
and $\omega < 0$ is a frequency parameter. Among all standing wave solutions,
we are particularly interested in the {\bf \em positive single-lobe states}, examples of which are shown on Figure \ref{fig-single-lobe}.

\begin{definition}
\label{def-single-lobe}
The standing wave $\Phi \in H^1_C(\Gamma_N)$ is said to be a positive single-lobe state
if $\Phi(x) > 0$ for every $x \in \Gamma_N$ and on each bounded edge of $\Gamma_N$,
either the maximum of $\Phi$ is achieved at a single internal point
and the minima of $\Phi$ occur at the vertices or the minimum of $\Phi$
is achieved at a single internal point and the maxima of $\Phi$ occur at the vertices.
\end{definition}

\begin{figure}[htbp] %  figure placement: here, top, bottom, or page
   \centering
   \includegraphics[width=2.8in, height = 2in]{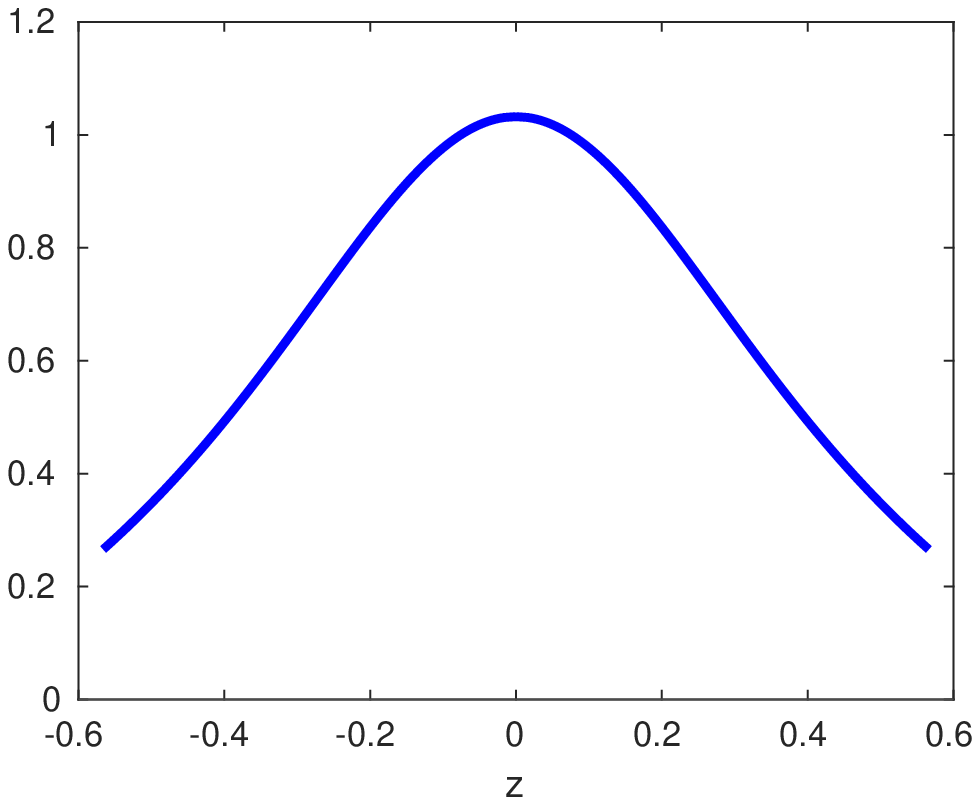}
 \includegraphics[width=2.8in, height = 2in]{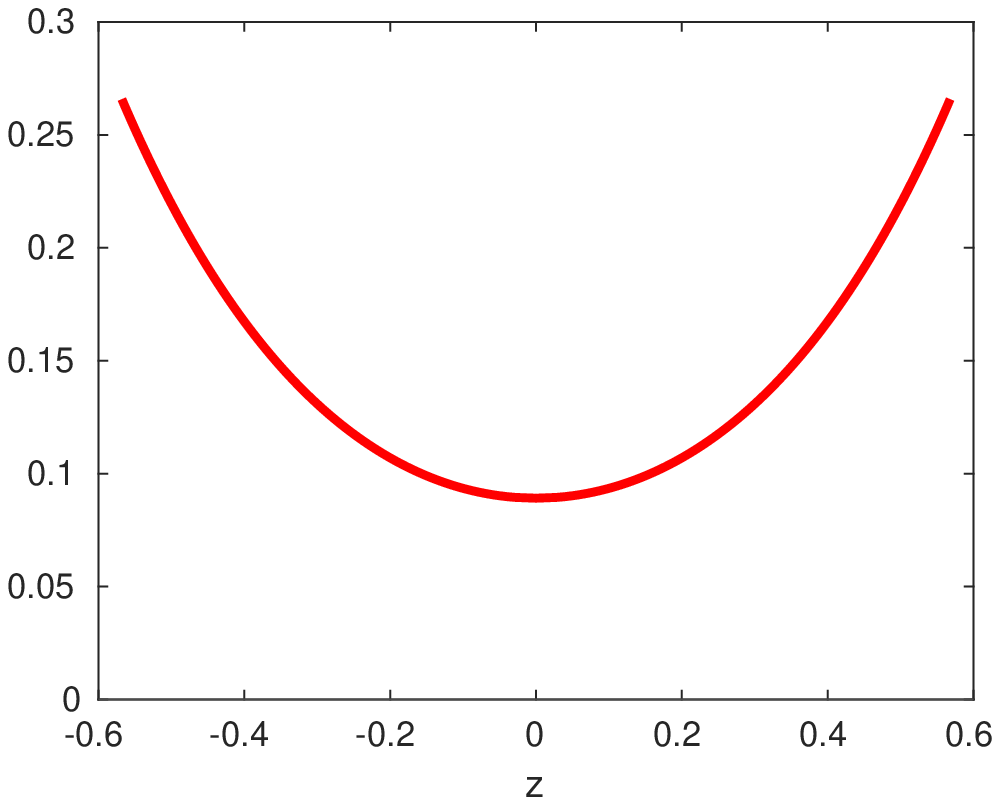}
   \caption{Examples of a positive single-lobe state on a bounded edge.
   Left: the maximum is achieved at the internal point, and the minima is achieved at the vertices.
   Right: the minimum is achieved at the internal point, and the maxima is achieved at the vertices.}
\label{fig-single-lobe}
\end{figure}

If $N = 1$, the graph $\Gamma_1$ is usually called the tadpole graph. Construction of standing waves
of the cubic NLS equation $(p = 1)$ on the tadpole graph $\Gamma_1$
was obtained with the use of elliptic functions in \cite{CFN}. Bifurcations and stability of standing waves
for small negative $\omega$ were analyzed for any $p > 0$ in \cite{NPS} by using Sturm's theory and asymptotic methods.

For the subcritical powers with $p \in (0,2)$ and for the tadpole graph $N = 1$,
it was shown in \cite{AdamiJFA} based on the variational method and symmetric energy-decreasing
rearrangements that the ground state of energy $E(\Psi)$ subject to the fixed mass $\mu := Q(\Psi)$
is attained for every $\mu > 0$ at the positive single-lobe state $\Phi$, which is
symmetric on the loop $[-\pi,\pi]$ and monotonically decreasing on $[0,\pi]$ and $[0,\infty)$. The ground state
$\Phi \in H^1_C(\Gamma_N)$ is the global minimizer of the variational problem
\begin{equation}
\label{min}
\mathcal{E}_{\mu} = \inf_{\Psi \in H^1_C(\Gamma_N)} \left\{ E(\Psi) : \quad Q(\Psi) = \mu \right\}.
\end{equation}
In the case $N = 1$, $\mathcal{E}_{\mu} = E(\Phi)$ is attained on the ground state $\Phi \in H^1_C(\Gamma_N)$ for $p \in (0,2)$.
Generally, $\mathcal{E}_{\mu}$ may not be attained on unbounded metric graphs \cite{AdamiCV}. For instance,
a sufficient condition on $\mu$ was found in Theorem 5.1 of \cite{AdamiJFA}
which ensures that $\mathcal{E}_{\mu}$ is not attained on a graph with a compact core and exactly one half-line
for $p \in (0,2)$. This result is applicable to the flower graph $\Gamma_N$ in the limit of large $N$.

For the critical power $p = 2$, it was shown in Theorem 3.3 in \cite{AST17} that the ground state
on the metric graph with exactly one half-line is attained if and only if $\mu \in (\mu_{\mathbb{R}^+},\mu_{\mathbb{R}}]$, where
$\mu_{\mathbb{R}^+}$ is the mass of the half-soliton on the half-line $\mathbb{R}^+$
and $\mu_{\mathbb{R}}$ is the mass of the full-soliton on the full line $\mathbb{R}$,
both values are independent of $\omega$ for $p = 2$. It is shown in the recent work \cite{MNP}
for the tadpole graph $\Gamma_1$ that the ground state is again given by the positive single-lobe state $\Phi$,
which is symmetric on the loop $[-\pi,\pi]$ and monotonically decreasing on $[0,\pi]$ and $[0,\infty)$.

Another relevant result is Theorem 3.3 in \cite{AST}, where the existence of local energy minimizers
was proven in the limit of large mass $\mu$ for $p \in (0,2)$ under the additional condition that the energy minimizer is localized
on one bounded edge of an unbounded graph and attains a maximum on this edge. This result applies to $\Gamma_N$ for every $N \geq 1$.
Alternative characterization of the standing waves in the limit of large mass $\mu$
was obtained in the cubic case $(p = 1)$ by using the elliptic functions \cite{BMP} where
the state of minimal energy at a fixed large mass $\mu$ was identified among the local minimizers.

The purpose of this work is to study the interplay between the existence of standing waves
of the NLS equation (\ref{nls-time}) and the symmetry of the metric graph in the particular case
of the flower graph $\Gamma_N$. We develop a novel analytical method to treat the existence of positive single-lobe
states from properties of the {\bf \em period function} for second-order differential equations.
Such properties are typically used for analysis of existence of periodic solutions to nonlinear evolution equations
\cite{Vill1,Vill2} as well as their spectral stability \cite{GP17}. The main novelty of our paper
is to show how applications of this method allow us to obtain precise analytical results on
the existence of positive single-lobe states. For clarity, we consider
the cubic case $(p=1)$ only. However, since we are not using elliptic functions,
the results here can be applied for any subcritical power with $p \in (0,2)$.

Let us now present the main results and the organization of this paper. Since we work with $p = 1$ and with real-valued $\Phi \in H^1_C(\Gamma_N)$, we rewrite the stationary NLS equation (\ref{nls-stat-intro}) in the explicit form:
\begin{equation}
\label{nls-stat}
\omega \Phi =  - \Delta \Phi - 2 \Phi^3.
\end{equation}
The standing wave $\Phi = (\phi_1,\phi_2,\dots,\phi_N, \phi_0)$
is a strong solution to the stationary NLS equation (\ref{nls-stat})
subject to the natural Neumann--Kirchhoff boundary conditions given by
\begin{equation}
\label{KBC}
\left\{ \begin{array}{l}
\phi_1(\pm \pi) = \phi_2(\pm \pi) = \dots = \phi_N(\pm \pi) = \phi_0(0), \\
\sum_{j=1}^N \left[ \phi_j'(\pi) - \phi_j'(-\pi) \right] = \phi_0'(0),
\end{array} \right.
\end{equation}
where the derivatives are defined as the one-sided limits of quotients.
We say that $\Phi \in H^2_{\rm NK}(\Gamma_N)$ if $\Phi \in H^2(\Gamma_N)$
satisfies the Neumann--Kirchhoff boundary conditions (\ref{KBC}), where
the Sobolev space $H^2(\Gamma_N)$ is also defined componentwise as
$$
H^2(\Gamma_N) = \underbrace{H^2(-\pi,\pi) \times \dots  \times H^2(-\pi,\pi)}_\text{\rm N times} \times H^2(0,\infty).
$$
$H^2_{\rm NK}(\Gamma_N)$ is the domain of the Laplacian operator
$\Delta : H^2_{\rm NK}(\Gamma_N) \subset L^2(\Gamma_N) \to L^2(\Gamma_N)$, where $\Delta$
is defined componentwise in $L^2(\Gamma_N)$.  By Theorem 1.4.4 in \cite{BK},
the Laplacian operator is self-adjoint in $L^2(\Gamma_N)$.
One can verify via integration by parts that for every $\Phi \in H^2_{\rm NK}(\Gamma_N)$ we have
$$
\langle (-\Delta) \Phi, \Phi \rangle_{L^2(\Gamma_N)} = \| \nabla \Phi \|_{L^2(\Gamma_N)}^2 \geq 0.
$$
Hence $\sigma(-\Delta) \subseteq [0,\infty)$ and $\omega$ in the stationary NLS equation (\ref{nls-stat})
is restricted to be negative. It is shown in Appendix \ref{appendix-spectrum} that $\sigma(-\Delta) = [0,\infty)$ includes
the continuous spectrum and a set of positive embedded eigenvalues.

Thanks to the symmetry of the flower graph $\Gamma_N$, we are first interested in
the existence of symmetric state, according to the following definition.

\begin{definition}
\label{def-symmetric}
We say that the standing wave is symmetric if $\Phi \in H^2_{\rm NK}(\Gamma_N)$
satisfies the symmetry condition
\begin{equation}
\label{sym-state}
\phi_1(x) = \phi_2(x) = \dots = \phi_N(x) \textrm{ for } x \in [-\pi, \pi].
\end{equation}
\end{definition}

The first main result states that there exists the unique positive single-lobe symmetric state
with the monotonically decreasing tail in the stationary NLS equation (\ref{nls-stat}) for every $\omega < 0$.
The proof of this result is given in Section \ref{sec-global-existence}.

\begin{theorem}
\label{global-existence}
For every $\omega<0$, there exists only one positive single-lobe symmetric state $\Phi \in H^2_{\rm NK}(\Gamma_N)$
which satisfies the stationary NLS equation (\ref{nls-stat}), is symmetric on each loop parameterized
by $[-\pi,\pi]$, and is monotonically decreasing on $[0,\pi]$ and $[0,\infty)$
The map $(-\infty,0) \ni \omega \mapsto \Phi(\cdot,\omega) \in H^2_{\rm NK}(\Gamma_N)$ is $C^1$
and the mass $\mu(\omega) := Q(\Phi(\cdot,\omega))$ is a $C^1$ monotonically decreasing function
satisfying the limits
$\mu(\omega) \to 0$ as $\omega \to 0$ and $\mu(\omega) \to \infty$
as $\omega \to -\infty$.
\end{theorem}

\begin{remark}
There exist other positive symmetric states satisfying the stationary NLS equation (\ref{nls-stat})
with more than one maximum on the $N$ loops or with a non-monotonically decreasing tail on $[0,\infty)$.
However, these other positive symmetric states are not local energy minimizers, and do not exist
for small negative $\omega$, hence we ignore them here.
\end{remark}

In what follows, we will often omit the dependence of $\Phi(\cdot,\omega)$ on $\omega$ obtained
in Theorem \ref{global-existence}. Given the positive single-lobe symmetric state $\Phi \in H^2_{\rm NK}(\Gamma_N)$
to the stationary NLS equation (\ref{nls-stat}), we can define the self-adjoint linear operator
$\mathcal{L} :  H^2_{\rm NK}(\Gamma_N) \subset L^2(\Gamma_N) \to L^2(\Gamma_N)$ given by
\begin{equation}
\label{Jacobian}
\mathcal{L} = -\Delta - \omega - 6 \Phi^2.
\end{equation}
Since $\phi_0(x) \to 0$ as $x \to \infty$ on the half-line, an application of Weyl's Theorem yields that
the continuous spectrum of $\mathcal{L}$ is given by
\begin{equation}
\label{abs-cont-part}
\sigma_{\rm a.c.}(\mathcal{L}) = \sigma(-\Delta - \omega) = [|\omega|,\infty).
\end{equation}
This implies that there are only finitely many eigenvalues of $\mathcal{L}$ of finite multiplicities located below $|\omega|$.
Let $n(\mathcal{L})$ be the {\bf \em Morse index} (the number of negative eigenvalues of $\mathcal{L}$ counted with their multiplicities)
and $z(\mathcal{L})$ be the {\bf \em nullity index} of the kernel of $\mathcal{L}$
(the multiplicity of the zero eigenvalue of $\mathcal{L}$). Since
\begin{equation}
\label{quad-form-phi}
\langle \mathcal{L} \Phi, \Phi \rangle_{L^2(\Gamma_N)} = - 4 \| \Phi \|_{L^4(\Gamma_N)}^{4} < 0,
\end{equation}
there is always a negative eigenvalue of $\mathcal{L}$ so that $n(\mathcal{L}) \geq 1$.
When the nullity index is nonzero, we define bifurcations of the symmetric state,
according to the following definition.

\begin{definition}
\label{def-bifurcation}
We say that the positive single-lobe symmetric state $\Phi$ at the given $\omega < 0$ undergoes a bifurcation if $z(\mathcal{L}) \geq 1$.
\end{definition}

The second main result states that the positive single-lobe symmetric state of Theorem \ref{global-existence}
undergoes exactly one bifurcation in the parameter continuation in $\omega$.
The proof of this result is given in Section \ref{sec-global-stability}.

\begin{theorem}
\label{global-stability}
Assume $N \geq 2$, and consider the positive single-lobe
symmetric state $\Phi \in H^2_{\rm NK}(\Gamma_N)$ of Theorem \ref{global-existence}.
There exists $\omega_* \in (-\infty,0)$ such that $z(\mathcal{L}) = N-1$ for $\omega = \omega_*$ and
$z(\mathcal{L}) = 0$ for $\omega \neq \omega_*$. Moreover, $n(\mathcal{L}) = N$ for $\omega \in (-\infty,\omega_*)$
and $n(\mathcal{L}) = 1$ for $\omega \in [\omega_*,0)$.
\end{theorem}

Figure \ref{fig-branches} shows the bifurcation diagram on the parameter plane $(\omega,\mu)$
in the case $N = 2$ (left) and $N = 3$ (right). 
The blue line on Fig. \ref{fig-branches} shows the symmetric state $\Phi$. 
At the bifurcation point $\omega_*$ of Theorem \ref{global-stability},
$(N-1)$ branches of positive asymmetric single-lobe states appear. 
These asymmetric states are defined as follows.

\begin{definition}
\label{def-asymmetric}
Fix $1 \leq K \leq N-1$. We say that the positive single-lobe state $\Phi \in H^1_C(\Gamma_N)$ is
asymmetric and $K$-split if, up to permutation between the components in the $N$ loops,
components of $\Phi$ satisfy the condition:
\begin{equation}
\label{K-symmetry}
\phi_1(x) = \dots = \phi_K(x), \quad \phi_{K+1}(x) = \dots = \phi_N(x), \quad \textrm{ for } x \in [-\pi, \pi].
\end{equation}
For convenience, we denote the positive single-lobe state satisfying (\ref{K-symmetry}) by $\Phi^{(K)}$
and assume that the $K$ components have larger amplitudes ($L^{\infty}$ norms on the corresponding edges),
whereas the $(N-K)$ components have smaller amplitudes.
\end{definition}

\begin{figure}[htbp] %  figure placement: here, top, bottom, or page
   \centering
   \includegraphics[width=3in, height = 2in]{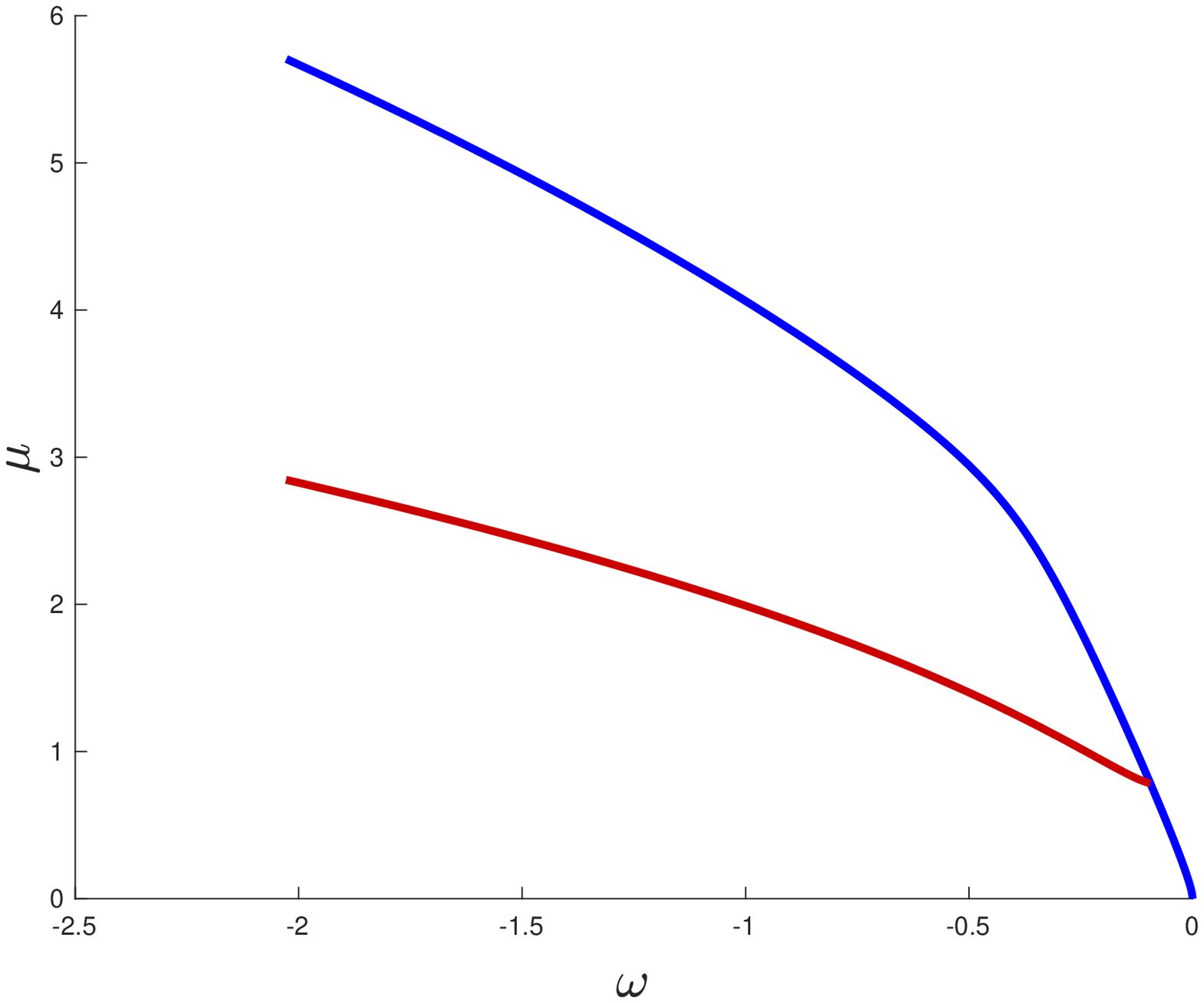}
   \includegraphics[width=3in, height = 2in]{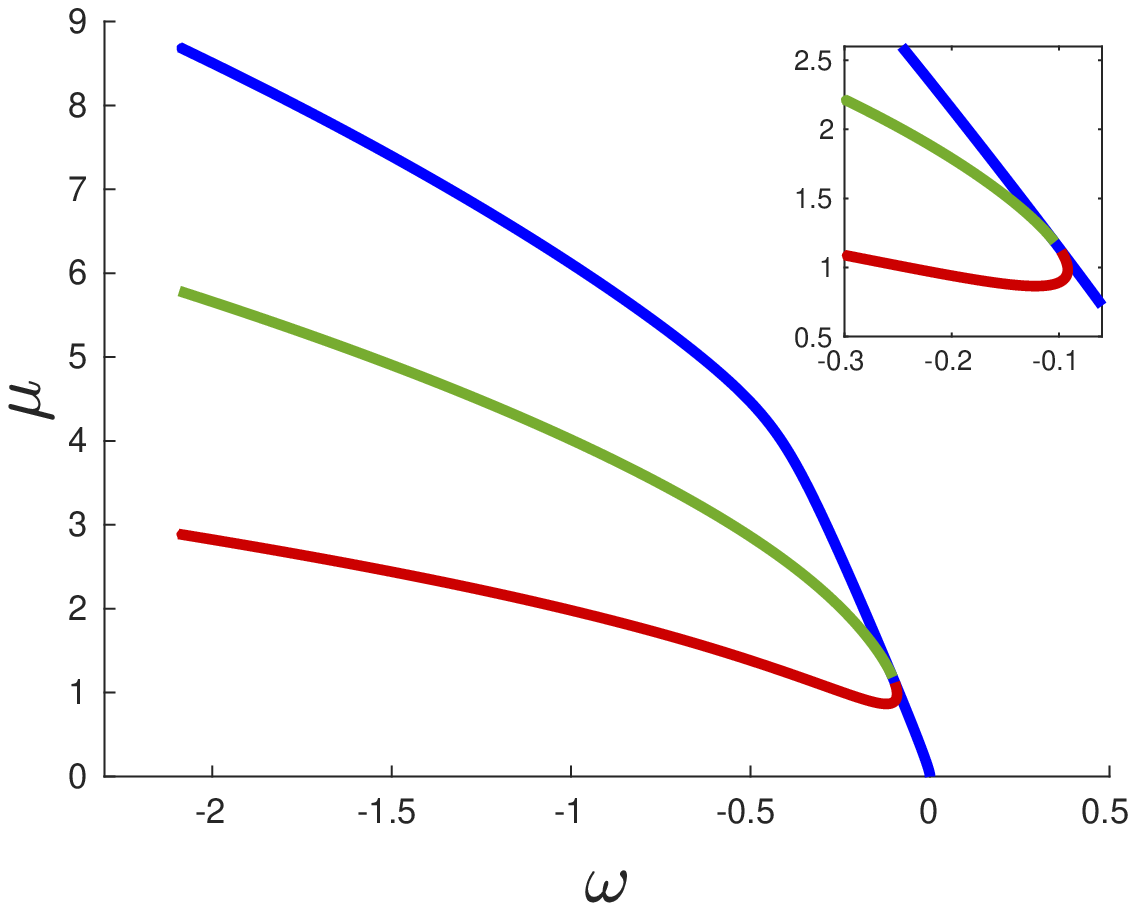}
   \caption{The bifurcation diagram of positive single-lobe states on the parameter plane $(\omega,\mu)$
   for $N = 2$ (left) and $N=3$ (right). The blue line shows the positive single-lobe symmetric state $\Phi$.
   The red line is the single-lobe state $\Phi^{(1)}$ with one component having larger amplitude than the other components.
   The green line (for $N = 3$) is the single-lobe state $\Phi^{(2)}$ with two components having larger amplitudes than the third one. }
   \label{fig-branches}
\end{figure}

Recall again that the blue line on Fig. \ref{fig-branches} depicts the symmetric state $\Phi$, which undergoes a bifurcation at $\omega = \omega_*$. 
The asymmetric $K$-split state $\Phi^{(K)}$ appears at the bifurcation point.
It follows from the insert of Fig. \ref{fig-branches} (right) for $N = 3$ that the branch of $\Phi^{(2)}$ given by the green line
is only located for $\omega < \omega_*$, whereas the branch of $\Phi^{(1)}$ given by the red line exists for
$\omega > \omega_*$ near the bifurcation point at $\omega_*$ and has a fold point at $\omega_1 \in (\omega_*,0)$.
The branch turns at the fold point and extends for every $\omega < \omega_1$. Hence, two points on the same
branch are located for a fixed value of $\omega$ in $(\omega_*,\omega_1)$.
Details of the numerical approximation
which produce the bifurcation diagram on Figure \ref{fig-branches} are described in Section \ref{sec-numerics}.

Although the behavior of $(N-1)$ branches can be complicated near the bifurcation point $\omega_*$, it becomes simple for large negative
values of $\omega$. Our third main result states a rather simple characterization of the positive single-lobe asymmetric states
for large negative $\omega$. The proof of this result is given in Section \ref{sec-global-bifurcations}.

\begin{theorem}
\label{global-bifurcations}
There exists $\omega_{\infty} \in (-\infty,\omega_*)$ such that
for every $\omega \in (-\infty,\omega_{\infty})$
there are exactly $N$ (up to permutations between the components in the $N$ loops)
positive single-lobe states $\Phi^{(K)} \in H^2_{\rm NK}(\Gamma_N)$ with $1 \leq K \leq N$,
which satisfy the stationary NLS equation (\ref{nls-stat}), are symmetric on each loop parameterized
by $[-\pi,\pi]$, and are monotonically decreasing on the half-line $[0,\infty)$. Moreover,
the first $K$ components in (\ref{K-symmetry}) are monotonically decreasing on $[0,\pi]$ and
the other $N-K$ components in (\ref{K-symmetry}) are monotonically increasing on $[0,\pi]$.
For every $K$, the map $(-\infty,\omega_{\infty}) \ni \omega \mapsto \Phi^{(K)}(\cdot,\omega) \in H^2_{\rm NK}(\Gamma_N)$
is $C^1$ and the mass $\mu^{(K)}(\omega) := Q(\Phi^{(K)}(\cdot,\omega))$ is a $C^1$ monotonically
decreasing function satisfying the limits
$\mu^{(K)}(\omega) \to \infty$ as $\omega \to -\infty$. Moreover,
\begin{equation}
\label{ordering-mass}
\mu^{(1)}(\omega) < \mu^{(2)}(\omega) < \dots < \mu^{(N-1)}(\omega) < \mu^{(N)}(\omega), \quad \omega \in (-\infty,\omega_{\infty}),
\end{equation}
where $\Phi^{(N)} = \Phi$ with $\mu^{(N)}(\omega) = \mu(\omega)$ are given by the symmetric state in Theorem \ref{global-existence}.
\end{theorem}

It follows from the characterization of local minimizers of energy in the limit of large mass in \cite{AST}
that the Morse index of $\Phi^{(K=1)}$ is $1$, whereas Theorem \ref{global-bifurcations}
defines a monotonically decreasing map $\omega \mapsto \mu^{(K=1)}(\omega)$ for large negative $\omega$. By Theorems
\ref{global-existence} and \ref{global-stability}, the Morse index of $\Phi \equiv \Phi^{(N)}$
is $1$ for small negative $\omega$ and the map $\omega \mapsto \mu^{(K=N)}(\omega)$ is monotonically decreasing for every $\omega$.
By the standard theory of orbital stability of standing waves, the following corollary is deduced
from these results.

\begin{corollary}
\label{corollary-global}
Assume $N \geq 2$. There exist $\mu_*$ and $\mu_{\infty}$ satisfying $0 < \mu_* \leq \mu_{\infty} < \infty$ such that
the positive single-lobe symmetric state $\Phi^{(K=N)} = \Phi$ of Theorem \ref{global-existence} is a
local minimizer of energy $E(\Psi)$ subject to the fixed mass $Q(\Psi) = \mu$ for $\mu \in (0,\mu_*)$,
whereas the positive single-lobe state $\Phi^{(K=1)}$ of Theorem \ref{global-bifurcations}
is a local minimizer of energy $E(\Psi)$ subject to the fixed mass $Q(\Psi) = \mu$ for $\mu \in (\mu_{\infty},\infty)$.
Moreover, $\mu_* = \mu(\omega_*) = Q(\Phi(\cdot,\omega_*))$, where $\omega_*$ is defined in
Theorem \ref{global-stability}.
\end{corollary}

\begin{remark}
\label{remark-flower}
One can show by the methods used
in \cite{BMP} and \cite{NPS} that the symmetric state $\Phi$ of Theorem \ref{global-existence}
is the ground state of the constrained minimization problem (\ref{min}) for small $\mu$,
whereas the asymmetric state $\Phi^{(K=1)}$ of Theorem \ref{global-bifurcations} is not
the ground state for large $\mu$ if $N \geq 2$, because
the infimum of the constrained minimization problem (\ref{min}) is not attained.
These results are given in Appendices \ref{appendix-small} and \ref{appendix-large} for completeness.
\end{remark}

\begin{remark}
\label{remark-tadpole}
For the tadpole graph ($N=1)$, no $\omega_*$ or $\mu_*$ exist and the symmetric state of Theorem
\ref{global-existence} is a local constrained minimizer of energy for every $\omega \in (-\infty,0)$.
Moreover, by Corollary 3.4 and the construction on Figure 4 in \cite{AdamiJFA}, it is the ground state of energy for every mass
$\mu \in (0,\infty)$.
\end{remark}

It follows from Proposition 3.3 in \cite{AdamiCV} that positive states are the only candidates for minimizers of the energy $E(\Psi)$
subject to the fixed mass $Q(\Psi) = \mu$. By Theorem 2.2 in \cite{AdamiCV}, $\mathcal{E}_{\mu}$ satisfies the bounds
\begin{equation}
\label{bounds-on-E}
-\frac{1}{3} \mu^3 \leq \mathcal{E}_{\mu} \leq -\frac{1}{12} \mu^3,
\end{equation}
where the lower bound is the energy of a half-soliton on a half-line with the same mass $\mu$
and the upper bound is the energy of a full soliton on a full line with the same mass $\mu$. By Theorem
3.3 and Corollary 3.4 in \cite{AdamiJFA}, the infimum is attained if there exists $\Psi_* \in H^2_{\rm NK}(\Gamma_N)$
such that $E(\Psi_*) \leq -\frac{1}{12} \mu^3$.

Figure \ref{fig-N-1} shows the branch of the positive single-lobe state $\Phi$ in the case $N = 1$ on the $(\omega,\mu)$ plane
(left) and on the $(\mu,\eta)$ plane (right), where $\eta := E(\Phi)$. The shaded area on Figure \ref{fig-N-1} (right) is defined between
the lower and upper bounds in (\ref{bounds-on-E}). The branches are computed numerically by
using the numerical methods based on the period function, see Section \ref{sec-numerics}.

In agreement with Remark \ref{remark-tadpole},
the positive single-lobe state for $N = 1$ is the ground state of the constrained minimization problem (\ref{min}) in the sense that
the solution branch on the $(\mu,\eta)$ plane is located in the shaded area for every $\mu > 0$. It approaches
the lower bound as $\mu \to 0$ when $\Phi$ is close to the half-soliton on the half-line and it approaches the upper bound
as $\mu \to \infty$ when $\Phi$ is close to the full soliton on the full line (see Appendices \ref{appendix-small} and \ref{appendix-large}).

\begin{figure}[htbp] %  figure placement: here, top, bottom, or page
   \centering
   \includegraphics[width=3in, height = 2in]{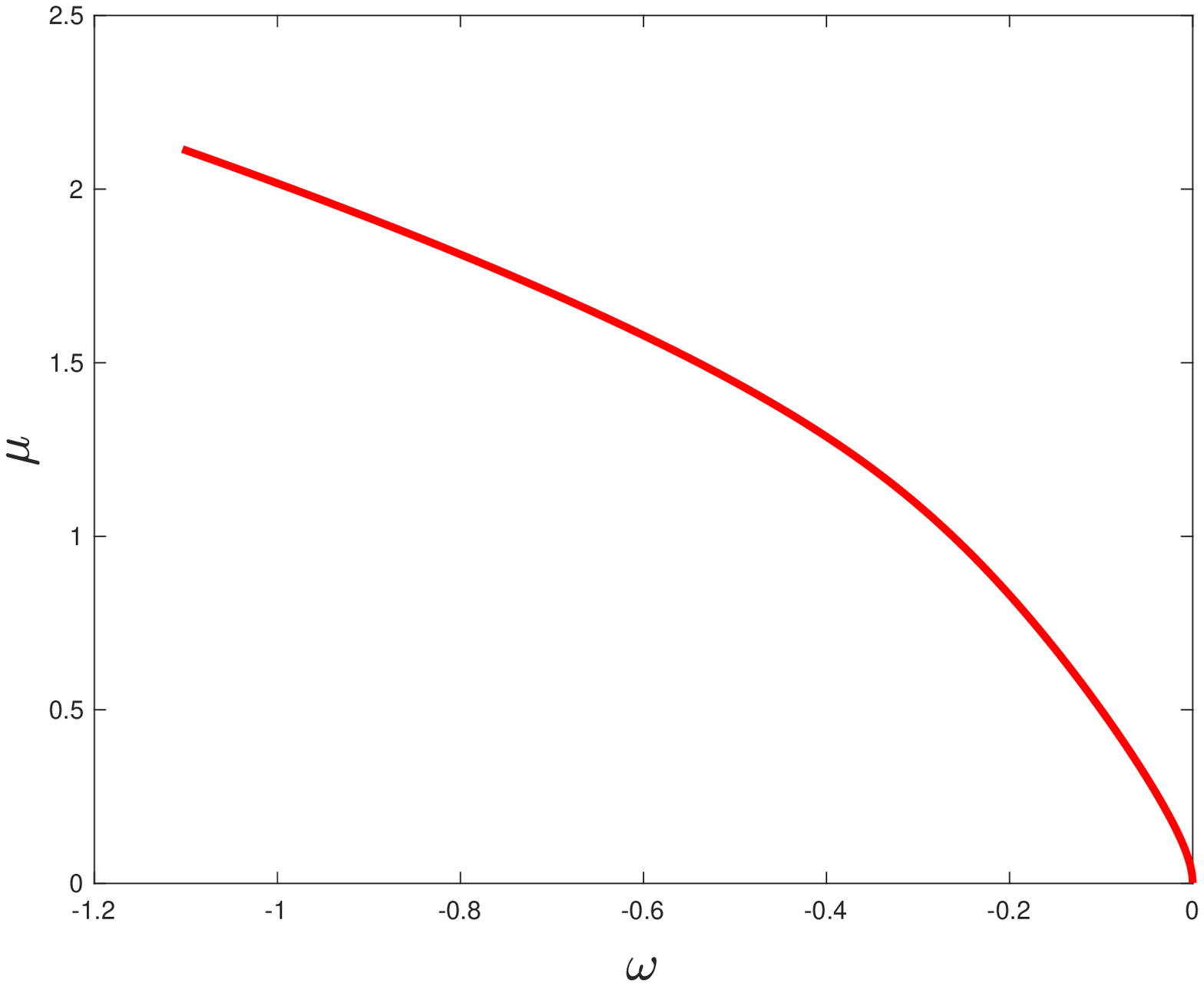}
   \includegraphics[width=3in, height = 2in]{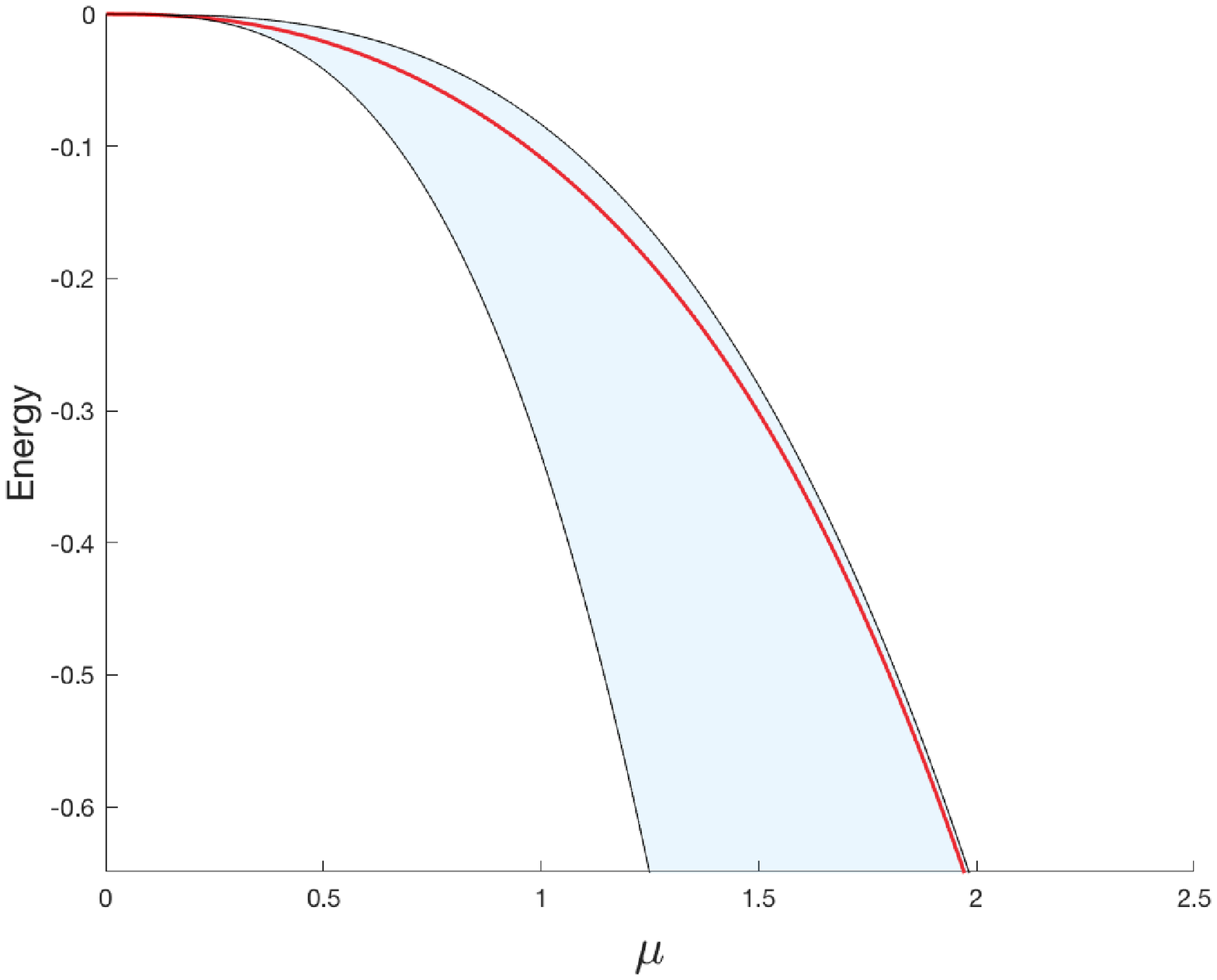}
   \caption{The branch of the positive single-lobe state $\Phi$ in the case $N = 1$ on the plane $(\omega,\mu)$
(left) and on the mass--energy plane (right).}
\label{fig-N-1}
\end{figure}

Figure \ref{fig-N-2} shows numerically computed branches of the positive single-lobe states on the $(\mu,\eta)$ plane
for $N = 2$ (left) and $N = 3$ (right). Compared to the case $N = 1$ on Figure \ref{fig-N-1} (right)
and in agreement with Remark \ref{remark-flower},
the branch for the positive single-lobe symmetric state $\Phi$ is located inside the shaded region
only for small mass $\mu$ and it goes beyond the shaded region, where the bifurcation of Theorem
\ref{global-stability} occurs. All new branches of positive single-lobe asymmetric states in Theorem \ref{global-bifurcations}
bifurcating from the branch for $\Phi$ stay away from the shaded region, hence these states are not the ground state
of the constrained minimization problem (\ref{min})
for any $\mu > 0$. Nevertheless, we note that the branch for $\Phi^{(K=N)} = \Phi$ is close to the lower bound as $\mu \to 0$ and the branch for $\Phi^{(K=1)}$
approaches the upper bound as $\mu \to \infty$ from the unshaded region.

\begin{figure}[htbp] %  figure placement: here, top, bottom, or page
	\centering
	\includegraphics[width=3in, height = 2in]{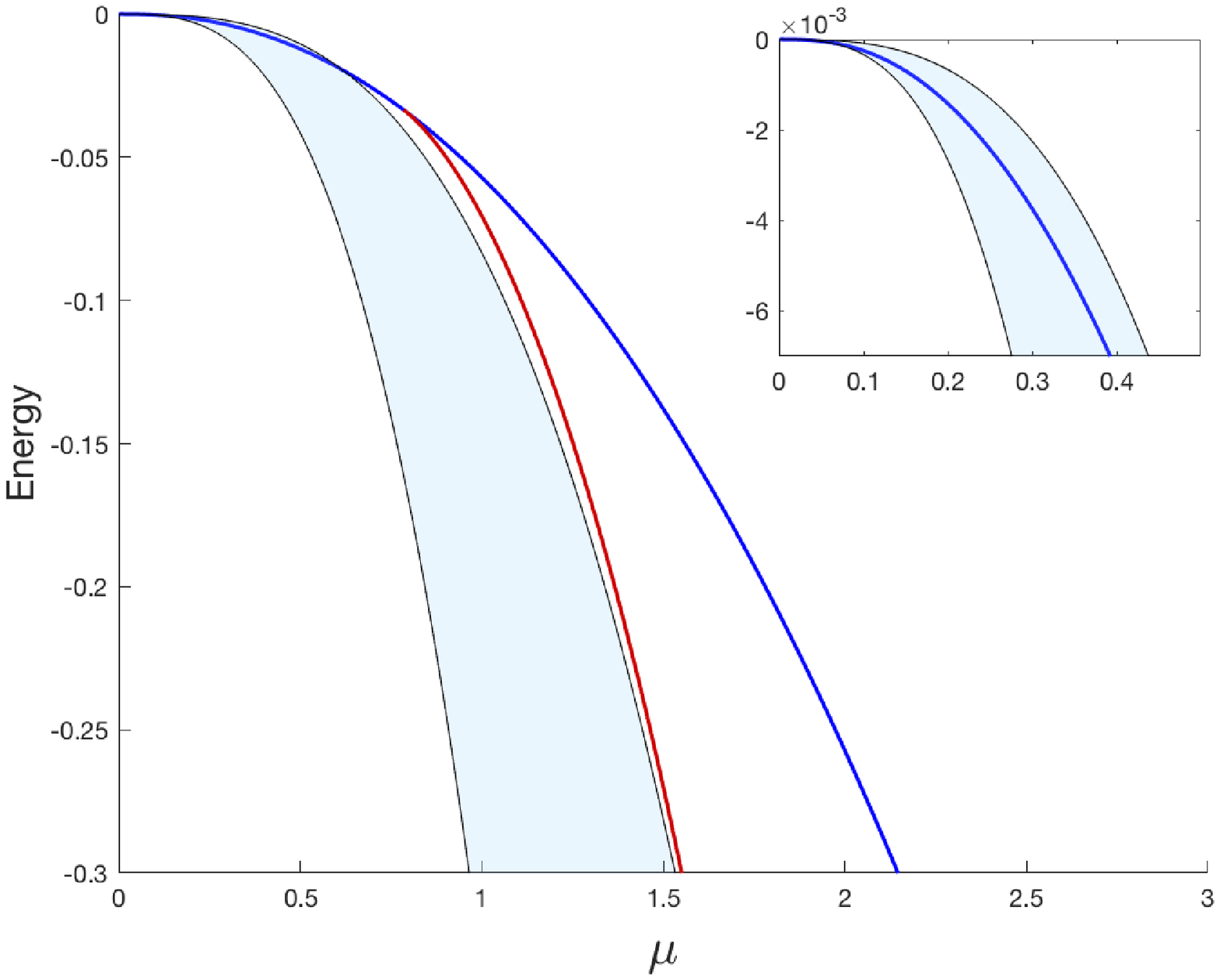}
	\includegraphics[width=3in, height = 2in]{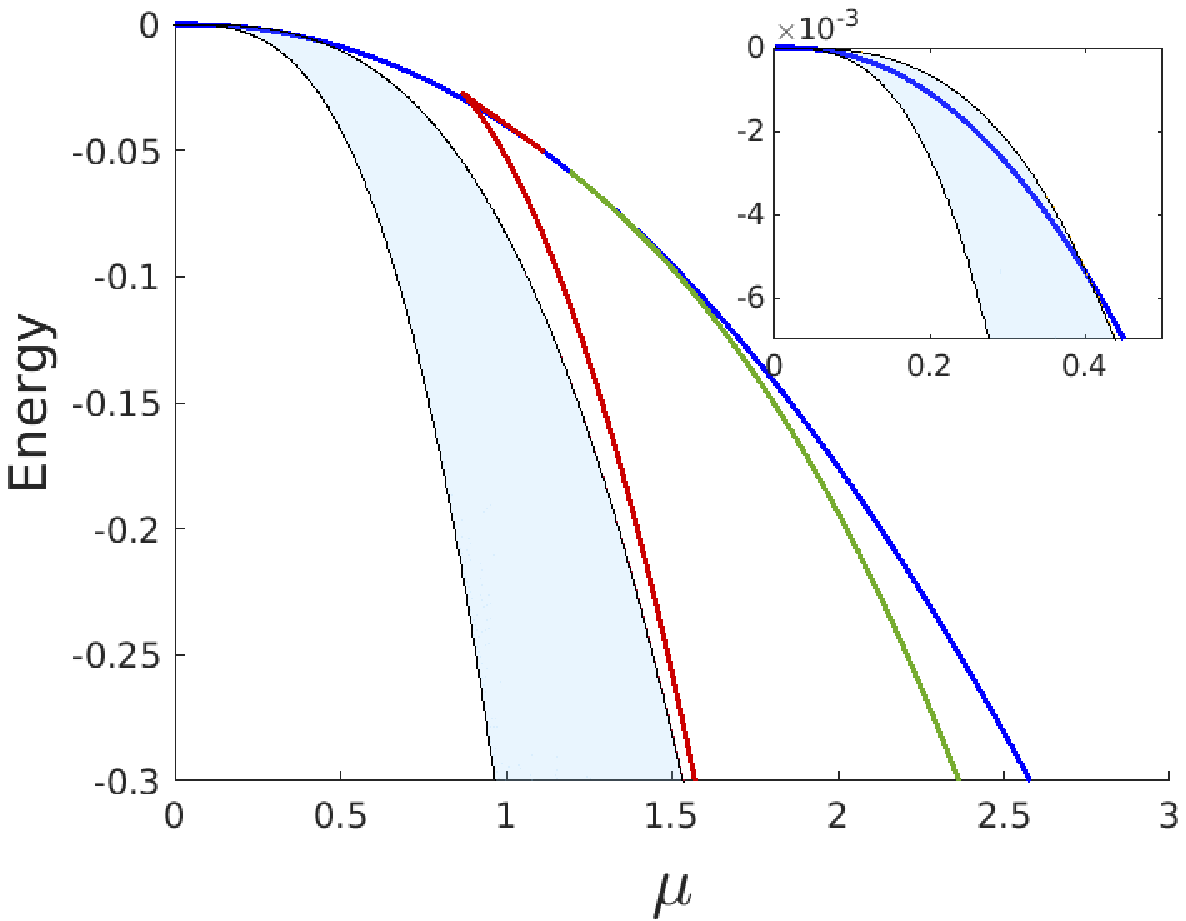}
	\caption{Bifurcation diagram of positive single-lobe states on the mass--energy plane
		for $N = 2$ (left) and $N=3$ (right).}
	\label{fig-N-2}
\end{figure}

\section{Existence of the positive single-lobe symmetric state}
\label{sec-global-existence}

Here we reformulate the stationary NLS equation (\ref{nls-stat}) 
equipped with the Neumann--Kirchhoff conditions (\ref{KBC}) in the form
for which we can use the dynamical system theory for orbits on the plane, e.g.
the period function. Then, we obtain estimates on the period function and on the mass of the symmetric state,
from which we prove Theorem \ref{global-existence}.

\subsection{Reformulation of the existence problem}

We use the following scaling transformation for $\omega : = -\epsilon^2 < 0$ with $\epsilon > 0$:
\begin{equation}
\label{scaling-transform}
\phi_0(x) = \epsilon u_0(\epsilon x), \quad
\phi_j(x) = \epsilon u_j(\epsilon x), \quad j \in \{1,2,\dots,N\}.
\end{equation}
In new variables, the stationary NLS equation (\ref{nls-stat}) transforms to the following system
of differential equations:
\begin{equation}
\label{NLS-scaled}
\left\{ \begin{array}{l} -u_j''(z) + u_j(z) - 2 u_j(z)^3 = 0, \quad z \in (-\pi \epsilon,\pi \epsilon), \quad j \in \{1,2,\dots,N\}, \\
-u_0''(z) + u_0(z) - 2 u_0(z)^3 = 0, \quad z > 0, \end{array} \right.
\end{equation}
where $z = \epsilon x$. The only dependence
of system (\ref{NLS-scaled}) on $\epsilon$ is due to the length of the interval $[-\pi \epsilon,\pi \epsilon]$.
The boundary conditions (\ref{KBC}) transform to the equivalent boundary conditions:
\begin{equation}
\left\{ \begin{array}{l}
u_1(\pm \pi \epsilon) = u_2(\pm \pi \epsilon) = \dots = u_N(\pm \pi \epsilon) = u_0(0), \\
\sum_{j=1}^N u_j'(\pi \epsilon) - u_j'(-\pi \epsilon) = u_0'(0),
\end{array} \right.
\label{bvp-scaled}
\end{equation}
The only positive decaying solution to equation 
$-u_0''(z) + u_0(z) - 2 u_0(z)^3 = 0$ on the half-line is expressed by the shifted NLS soliton:
\begin{equation}
\label{soliton}
u_0(z) = {\rm sech}(z + a), \quad z > 0,
\end{equation}
where $a \in \mathbb{R}$ is an arbitrary translation parameter.
If $a > 0$, $u_0$ is monotonically decreasing on $[0,\infty)$ and if $a < 0$,
$u_0$ is non-monotone on $[0,\infty)$. In order to prove Theorem \ref{global-existence},
we only consider the positive states with the monotonically decreasing $u_0$,
hence we select $a > 0$.

Each second-order differential equation in the system (\ref{NLS-scaled})  is integrable
with the first-order invariant:
\begin{equation}
\label{invariant}
E(u,v) = v^2 - A(u), \quad v := \frac{du}{dz}, \quad A(u) := u^2 (1 - u^{2}),
% consider changing E to I everywhere
\end{equation}
where the value of $E(u,v) = E$ is independent of $z$. Figure \ref{fig-phase-portrait} shows the phase portrait given by the level curves of the function $E(u,v)$ on the $(u,v)$-plane.

\begin{figure}[htbp] 
	\centering
	\includegraphics[width=4in, height = 3in]{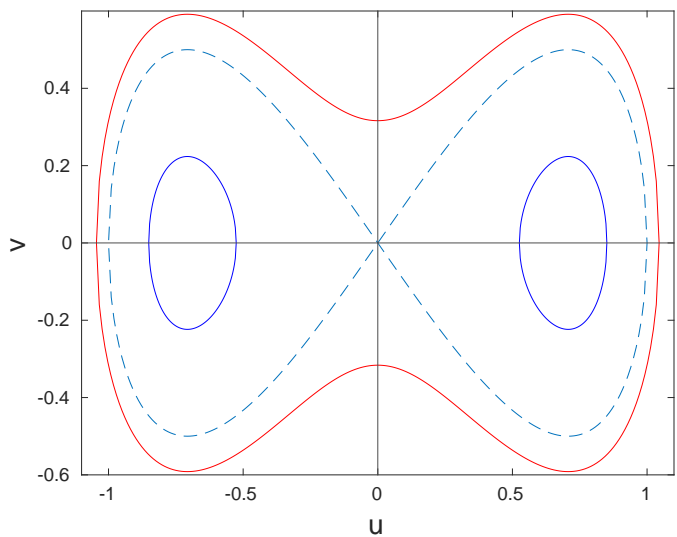}
	\caption{Phase portrait on the $(u,v)$-plane given by the level curves of the function $E(u,v)$.}
	\label{fig-phase-portrait}
\end{figure}

Since $A(u) = u^2 (1-u^2)$, there exists only one positive root of $A'(u)$ denoted as $p_*$ such that $A'(p_*) = 0$, in fact, $p_* = \frac{1}{\sqrt{2}}$.
Two homoclinic orbits exist for $E = 0$, one corresponds to positive $u$ and the other one corresponds to negative $u$. 
Periodic orbits exist inside each of the two
homoclinic loops and correspond to $E \in (E_*,0)$, where $E_* = - A(p_*) = - \frac{1}{4}$, and they correspond to either strictly positive $u$ or strictly negative $u$. Periodic orbits outside the two homoclinic loops exist for $E \in (0,\infty)$ and they correspond to sign-indefinite $u$. 
Note that 
\begin{equation}
\label{E-positivity}
E + A(p_*) > 0, \qquad E \in (E_*,\infty).
\end{equation}
The homoclinic orbit with the decaying solution (\ref{soliton}) corresponds to $E = 0$ and either $v = \sqrt{A(u)}$ if $z + a < 0$
or $v = -\sqrt{A(u)}$ if $z + a > 0$. Since $u_0$ is monotonically decreasing on $[0,\infty)$ and $a > 0$, we have $v = -\sqrt{A(u)}$ for all $z > 0$.

Let us define $p_0 := {\rm sech}(a)$, that is, the value of $u_0(z)$ at $z = 0$.
Then, $-\sqrt{A(p_0)}$ is the value of $u_0'(z)$ at $z = 0$. Note that $p_0 \in (0,1)$
is a free parameter obtained from $a \in (0,\infty)$ such that $p_0(a) \to 1$ when $a \to 0$
and $p_0(a) \to 0$ when $a \to \infty$.

Under the scaling transformation (\ref{scaling-transform}),
the symmetry condition (\ref{sym-state}) yields
\begin{equation}
\label{sym-state-scaled}
u_1(z) = u_2(z) = \dots = u_N(z), \qquad z \in [-\pi \epsilon,\pi \epsilon],
\end{equation}
hence the positive symmetric state of Definition \ref{def-symmetric} is found from the following boundary-value problem:
\begin{equation}
\left\{ \begin{array}{l} -u_1''(z) + u_1(z) - 2 u_1(z)^3 = 0, \quad z \in (-\pi \epsilon,\pi \epsilon), \\
u_1(\pi \epsilon) = u_1(-\pi \epsilon) = p_0, \\
u_1'(-\pi \epsilon) = - u_1'(\pi \epsilon) = \frac{1}{2N} \sqrt{A(p_0)},
\end{array} \right.
\label{bvp-i}
\end{equation}
where $p_0 \in (0,1)$ is a free parameter of the problem.
The positive single-lobe states of Definition \ref{def-single-lobe} correspond to a part of the level curve $E(u,v) = E$ 
which intersects $p_0$ only twice at the ends of the interval $[-\pi \epsilon,\pi \epsilon]$.

Figure \ref{fig-plane} shows a geometric construction of solutions to the boundary-value problem (\ref{bvp-i}) on the plane $(u,v)$.
The dashed line represents the homoclinic orbit at $E = 0$ with the solid part depicting the shifted NLS soliton (\ref{soliton}) for $a = 0.7$ (left) and $a = 1$ (right). The dashed-dotted
vertical line depicts the value of $p_0 = u_0(0) = {\rm sech}(a)$. The red solid line plots $q_0 = \frac{1}{2N} \sqrt{A(p_0)}$ versus $p_0 \in (0,1)$. The level curve 
$E(u,v) = E(p_0,q_0)$ at $p_0 = {\rm sech}(a)$
and $q_0 = \frac{1}{2N} \sqrt{A(p_0)}$ is shown by the dashed line, whereas the solid part depicts a suitable solution to the boundary-value problem (\ref{bvp-i}).

\begin{figure}[htbp] %  figure placement: here, top, bottom, or page
   \centering
  \includegraphics[width=2.8in, height = 2in]{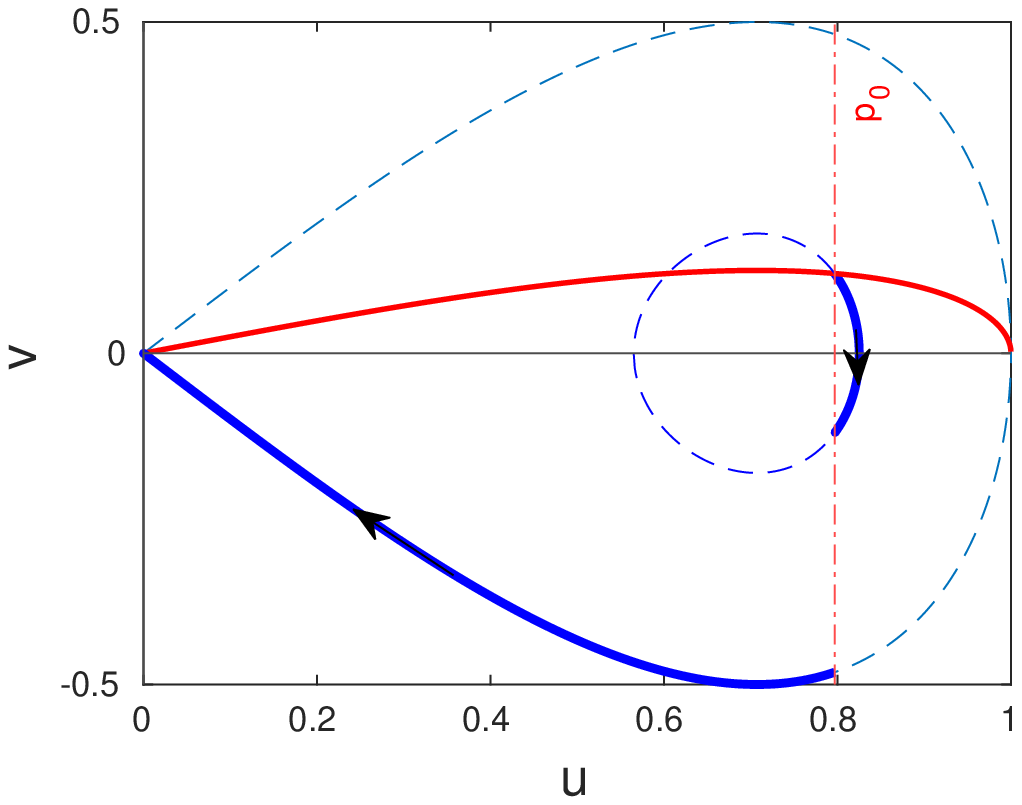}
  \includegraphics[width=2.8in, height = 2in]{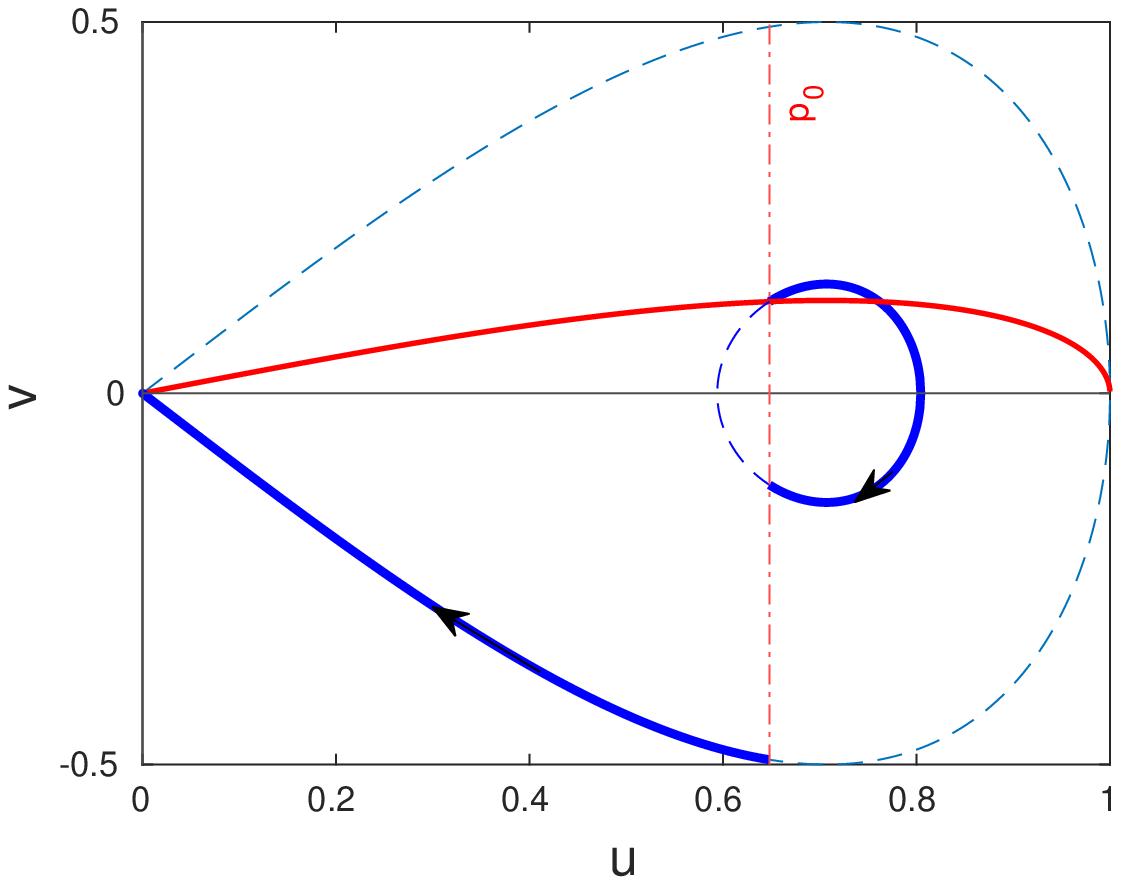}
   \caption{Geometric construction of the positive single-lobe symmetric state on the plane for $N=2$
   with $a = 0.7$ (left) and $a = 1$ (right).}
   \label{fig-plane}
\end{figure}

We shall make this geometric picture rigorous by using analytical tools of
the period function (see, e.g., \cite{Vill1}). We define two period functions for a given $(p_0,q_0)$:
\begin{equation}
\label{period}
T_+(p_0,q_0) := \int_{p_0}^{p_+} \frac{du}{\sqrt{E + A(u)}}, \quad
T_-(p_0,q_0) := \int_{p_-}^{p_0} \frac{du}{\sqrt{E + A(u)}},
\end{equation}
where the fixed value $E$ and the turning points $p_+$ and $p_-$ are defined
from $(p_0,q_0)$ by
\begin{equation}
\label{p+p-}
E = q_0^2 - A(p_0) = -A(p_+) = -A(p_-).
\end{equation}
For each level curve of $E(u,v) = E$ inside the homoclinic loop on Figure \ref{fig-plane}, we can order the turning points as follows: 
\begin{equation}
\label{turning-points}
0 < p_- < p_* < p_+ < 1.
\end{equation}
It is obvious that $u_1$ is a positive single-lobe solution 
of the boundary-value problem (\ref{bvp-i}) if and only if $p_0$ 
is a root of the nonlinear equation:
\begin{equation}
\label{root-i}
\mathcal{T}(p_0) = \pi \epsilon, \quad \mbox{\rm where   } \mathcal{T}(p_0) := T_+\left( p_0,\frac{1}{2N} \sqrt{A(p_0)}\right).
\end{equation}
Since $\mathcal{T}(p_0)$ is uniquely defined by $p_0 \in (0,1)$, the nonlinear equation (\ref{root-i})
defines a unique mapping $(0,1) \ni p_0 \mapsto \epsilon(p_0) := \frac{1}{\pi} \mathcal{T}(p_0) \in (0,\infty)$. Monotonicity of this mapping is shown next.

\subsection{Monotonicity of the period function}

It follows that $u = p_*$ is a double root of $A(u) - A(p_*)$ since $A'(p_*) = 0$ and $A''(p_*) \neq 0$,
where we can use the explicit computations of $A'(u) = 2u (1 - 2 u^{2})$ and $A''(u) = 2 (1 - 6 u^{2})$.

Recall that if $W(u,v)$ is a $C^1$ function in an open region of $\mathbb{R}^2$, then the differential of $W$ is defined by 
$$
dW(u,v) = \frac{\partial W}{\partial u} du + \frac{\partial W}{\partial v} dv
$$
and the line integral of $d W(u,v)$ along any $C^1$ contour $\gamma$ connecting $(u_0,v_0)$ and $(u_1,v_1)$ does not depend on $\gamma$ and is evaluated as 
$$
\int_{\gamma} d W(u,v) = W(u_1,v_1) - W(u_0,v_0).
$$
At the level curve of $E(u,v) = v^2 - A(u) = E$, we can write
\begin{eqnarray*}
d \left[ \frac{2 (A(u) - A(p_*)) v}{A'(u)} \right] =
\left[ 2 - \frac{2 (A(u) - A(p_*)) A''(u)}{[A'(u)]^2} \right] v du +  \frac{2 (A(u) - A(p_*))}{A'(u)} dv.
\end{eqnarray*}
where the quotients are not singular for every $u > 0$. 
Since $2 v dv = A'(u) du$ on the level curve $E(u,v) = E$, 
the previous representation allows us to express
\begin{equation}
\frac{A(u) - A(p_*)}{v} du = -\left[ 2 - \frac{2 (A(u) - A(p_*)) A''(u)}{[A'(u)]^2} \right] v du +
d \left[ \frac{2 (A(u) - A(p_*)) v}{A'(u)} \right].
\label{integration}
\end{equation}
The following lemma justifies monotonicity of the mapping $(0,1) \ni p_0 \mapsto \mathcal{T}(p_0) \in (0,\infty)$.

\begin{lemma}
\label{global-decreasing}
The function $p_0 \mapsto \mathcal{T}(p_0)$ is $C^1$ and monotonically decreasing for every $p_0 \in (0,1)$.
\end{lemma}

\begin{proof}
Since $q_0 = \frac{1}{2N} \sqrt{A(p_0)}$ in (\ref{root-i})
for a given $p_0 \in (0,1)$, the value of $\mathcal{T}(p_0)$
is obtained from the level curve $E(u,v) = E_0(p_0)$, where
\begin{equation}
\label{energy-level}
E_0(p_0) := q_0^2 - A(p_0) = -\left( 1-\frac{1}{4N^2} \right) A(p_0).
\end{equation}
For every $p_0 \in (0,1)$, we use the formula (\ref{integration}) to get
\begin{eqnarray*}
\left[ E_0(p_0) + A(p_*) \right] \mathcal{T}(p_0) & = & \int_{p_0}^{p_+} \left[ v - \frac{A(u)-A(p_*)}{v} \right] du \\
& = & \int_{p_0}^{p_+} \left[3 - \frac{2 (A(u) - A(p_*)) A''(u)}{[A'(u)]^2} \right] v du +
\frac{2(A(p_0)-A(p_*))q_0}{A'(p_0)},
\end{eqnarray*}
where we have used that $v = 0$ at $u = p_+$ and $v = q_0$ at $u = p_0$.
Because the integrands are free of singularities and $E_0(p_0) + A(p_*) > 0$ due to (\ref{E-positivity}),
the mapping $(0,1) \ni p_0 \mapsto \mathcal{T}(p_0) \in (0,\infty)$ is $C^1$. We
only need to prove that $\mathcal{T}'(p_0) < 0$ for every $p_0 \in (0,1)$.

Differentiating the previous expression with respect to $p_0$ yields
\begin{eqnarray*}
\left[ E_0(p_0) + A(p_*) \right] \mathcal{T}'(p_0) =
-\frac{A(p_*)}{4N^2q_0} - \frac{A'(p_0)}{2}\left( 1-\frac{1}{4N^2} \right)
 \int_{p_0}^{p_+} \left[1 - \frac{2 (A(u) - A(p_*)) A''(u)}{[A'(u)]^2} \right] \frac{du}{v},
\end{eqnarray*}
where we have used 
$$
E_0'(p_0) = -\left( 1-\frac{1}{4N^2} \right) A'(p_0), \quad q_0'(p_0) = \frac{1}{8N^2 q_0} A'(p_0), \quad \frac{\partial v}{\partial p_0}  = \frac{1}{2v} E_0'(p_0).
$$
The formula for $\mathcal{T}'(p_0)$ can be simplified using $A(u) = u^2 (1-u^2)$ to the form:
\begin{eqnarray}
\label{simplified-form}
\left[ E_0(p_0) + A(p_*) \right] \mathcal{T}'(p_0) = -\frac{A(p_*)}{4N^2q_0} - \frac{A'(p_0)}{8}\left( 1-\frac{1}{4N^2} \right)
 \int_{p_0}^{p_+} \frac{1 - 2u^2}{u^2v} du.
\end{eqnarray}
Since $A'(p_0) <0$ for any $p_0\in (p_*, 1)$ and $1 - 2 u^2 < 0$ for $u \in [p_0,p_+] \subset (p_*,1)$, we get that
$\mathcal{T}'(p_0) < 0$ for any $p_0 \in (p_*, 1)$.
Similarly, since $A'(p_*) = 0$, we also have $\mathcal{T}'(p_*) < 0$.

If $p_0 \in (0, p_*)$ we use $A'(u) = 2u(1-2u^2)$ and proceed with integration by parts to get
\begin{eqnarray}
\label{good-int-by-parts}
\int_{p_0}^{p_+} \frac{1 - 2u^2}{u^2v} du & = &
\int_{p_0}^{p_+} \frac{A'(u)}{2u^3 \sqrt{E_0(p_0) + A(u)}} du \\
\nonumber
&= & - \frac{\sqrt{E_0(p_0) + A(p_0)}}{p_0^3} + 3\int_{p_0}^{p_+} \frac{\sqrt{E_0(p_0) + A(u)}}{u^4} du \\
&= & -\frac{q_0}{p_0^3} + 3\int_{p_0}^{p_+} \frac{v}{u^4} du.
\nonumber
\end{eqnarray}
Substituting this into (\ref{simplified-form}) yields
\begin{eqnarray*}
\left[ E_0(p_0) + A(p_*) \right] \mathcal{T}'(p_0) & =&
 -\frac{A(p_*)}{4N^2q_0} - \frac{A'(p_0)}{8}\left( 1-\frac{1}{4N^2} \right)
 \int_{p_0}^{p_+} \frac{1 - 2u^2}{u^2v} du \\
& = & -\frac{A(p_*)}{4N^2q_0} + \frac{A'(p_0)}{8}\left( 1-\frac{1}{4N^2} \right) \frac{q_0}{p_0^3}
-\frac{3A'(p_0)}{8}\left( 1-\frac{1}{4N^2} \right)
\int_{p_0}^{p_+} \frac{v}{u^4} du.
\end{eqnarray*}
The last term is negative since $A'(p_0)>0$ for every $p_0 \in (0, p_*)$. To evaluate the first two terms
we use that $A(p_*) = \frac{1}{4}$, $A'(p_0) = 2p_0 (1 - 2 p_0^2)$, and $4 N^2 q_0^2 = p_0^2 - p_0^4$ so that we get
\begin{eqnarray*}
-\frac{A(p_*)}{4N^2q_0} + \frac{A'(p_0)}{8}\left( 1-\frac{1}{4N^2} \right) \frac{q_0}{p_0^3} & = &
 -\frac{1}{16N^2q_0} + \frac{(1-2p_0^2)q_0}{4p_0^2} -
\frac{A'(p_0)q_0}{32N^2 p_0^3} \\
& = & - \frac{(3-2p_0^2)p_0^2}{16N^2q_0} - \frac{A'(p_0)q_0}{32N^2 p_0^3},
\end{eqnarray*}
which is negative for every $p_0 \in (0, p_*)$.
As a result of the above calculations, for every $p_0 \in (0, p_*)$
we have $\mathcal{T}'(p_0) <0$.
\end{proof}

\subsection{Monotonicity of the mass of the symmetric state}

By construction of the symmetric state $\Phi$, we compute the mass $\mu(\omega) := Q(\Phi(\cdot,\omega))$
in the form
\begin{eqnarray*}
\mu = N \int_{-\pi}^{\pi} \phi_1^2 dx + \int_0^{\infty} \phi_0^2 dx.
\end{eqnarray*}
Due to the scaling transformation (\ref{scaling-transform}), the explicit solution
on the half-line (\ref{soliton}), and the first-order invariant (\ref{invariant}),
the mass integral can be rewritten as follows:
\begin{equation}
\label{mass-explicit}
\mu = 2N \epsilon \int_{p_0}^{p_+} \frac{u^2 du}{\sqrt{E + A(u)}} + \epsilon \left( 1-\sqrt{1-p_0^2} \right).
\end{equation}
where $p_0 \in (0,1)$ is the same parameter as in (\ref{root-i}), $E(u,v) = E_0(p_0)$ is fixed at the energy level
(\ref{energy-level}), $A(u) = u^2 - u^4$, and we have used $\tanh(a) = \sqrt{1 - p_0^2}$ that follows from
${\rm sech}(a) = p_0$ with $a > 0$.

Using $\mathcal{T}(p_0) = \pi \epsilon$ in (\ref{root-i}), we rewrite (\ref{mass-explicit}) as
\begin{equation}
\label{mass-pi}
\mathcal{M}(p_0) := \pi \mu = \mathcal{T} (p_0) \left[ 2N \int_{p_0}^{p_+} \frac{u^2 du}{\sqrt{E+A(u)}}
+ \left( 1-\sqrt{1-p_0^2} \right) \right].
\end{equation}

Recall that $\omega = -\epsilon^2$ and that the function $p_0 \mapsto \mathcal{T}(p_0)$ is $C^1$ and monotonically decreasing for every $p_0 \in (0,1)$ by Lemma \ref{global-decreasing}. 
The following lemma gives monotonicity
of the mapping $(0,1) \ni p_0 \mapsto \mathcal{M}(p_0) \in (0,\infty)$.

\begin{lemma}
\label{lemma-mass-monotonicity}
The function $p_0 \mapsto \mathcal{M}(p_0)$ is $C^1$ and monotonically decreasing for every $p_0 \in (0,1)$.
\end{lemma}

\begin{proof}
We denote $\mathcal{B}(p_0) := \int_{p_0}^{p_+} \frac{u^2 du}{\sqrt{E+A(u)}}$ and prove that
the mapping $(0,1) \ni p_0 \mapsto \mathcal{B}(p_0) \in (0,\infty)$ is $C^1$. At the level curve $E(u,v) = E_0(p_0)$, we can write
\begin{eqnarray*}
d \left[ \frac{2 (A(u) - A(p_*)) u^2 v}{A'(u)} \right] =
2 \left[ 1 + \frac{2 (1+2u^2) (A(u) - A(p_*))}{[A'(u)]^2} \right] u^2 v du +  \frac{2 (A(u) - A(p_*))}{A'(u)} u^2 dv.
\end{eqnarray*}
where the relations $A'(u) = 2u(1-2u^2)$ and $A''(u) = 2(1-6u^2)$ have been used.
Since $2v dv = A'(u) du$ along the level curve $E(u,v) = E_0(p_0)$, we obtain
\begin{eqnarray}
\nonumber
\frac{(A(u) - A(p_*))}{v} u^2 du & = &
d \left[ \frac{2 (A(u) - A(p_*)) u^2 v}{A'(u)} \right] \\
& \phantom{t} & - 2 \left[ 1 +
\frac{2(1+2u^2) (A(u) - A(p_*))}{[A'(u)]^2} \right] u^2 v du,
\label{integration-notes}
\end{eqnarray}
where the quotients are not singular for every $u > 0$. For every $p_0 \in (0,1)$
we use the formula (\ref{integration-notes}) to write
\begin{eqnarray*}
%\label{B}
\left[ E_0(p_0) + A(p_*) \right] \mathcal{B}(p_0) & = & \int_{p_0}^{p_+} \left[ vu^2 - \frac{(A(u)-A(p_*))u^2}{v} \right] du \\
& = & \int_{p_0}^{p_+} \left[3 +
\frac{4 (1+2u^2) (A(u) - A(p_*))}{[A'(u)]^2} \right] u^2 v du  +\frac{2(A(p_0)-A(p_*))p_0^2 q_0}{A'(p_0)},
\end{eqnarray*}
where we have used that $v = 0$ at $u = p_+$ and $v = q_0$ at $u = p_0$. Because the integrands are free of singularities and
$E_0(p_0) + A(p_*) > 0$ due to (\ref{E-positivity}), the mapping $(0,1) \ni p_0 \mapsto \mathcal{B}(p_0) \in (0,\infty)$ is $C^1$.
Hence, the mapping $(0,1) \ni p_0 \mapsto \mathcal{M}(p_0) \in (0,\infty)$ is $C^1$.
It remains to prove that $\mathcal{M}'(p_0)<0$ for every $p_0 \in (0,1)$.

We differentiate (\ref{mass-pi}) with respect to $p_0$:
\begin{equation}
\label{mass-pi-2}
\mathcal{M}'(p_0) = \mathcal{T}'(p_0) \left[ 2 N \mathcal{B}(p_0) +\left( 1-\sqrt{1-p_0^2} \right) \right]
+ 2N \mathcal{T}(p_0) \mathcal{B}'(p_0) + \mathcal{T}(p_0) \frac{p_0}{\sqrt{1-p_0^2}}.
\end{equation}
It follows from the proof of Lemma \ref{global-decreasing} that $\mathcal{T}'(p_0) < 0$, where
\begin{equation}
\label{period-d}
\left[ E_0(p_0) + A(p_*) \right] \mathcal{T}'(p_0) =
-\frac{A(p_*)}{4N^2q_0} - \frac{A'(p_0)}{8}\left( 1-\frac{1}{4N^2} \right)
 \int_{p_0}^{p_+} \frac{1 - 2u^2}{u^2v} du.
\end{equation}
Similarly, differentiating the expression for $\mathcal{B}(p_0)$ yields the following expression:
\begin{eqnarray}
\left[ E_0(p_0) + A(p_*) \right] \mathcal{B}'(p_0) =
-\frac{A(p_*)p_0^2}{4N^2q_0} + \frac{A'(p_0)}{8}\left( 1-\frac{1}{4N^2} \right)
 \int_{p_0}^{p_+} \frac{1 - 2u^2}{v} du.
\label{d-B}
\end{eqnarray}
The first term in the right-hand side of (\ref{mass-pi-2}) is always negative, whereas the third term is always positive.
The second term can be of either sign depending on the value of $p_0 \in (0,1)$. In order to prove that
$\mathcal{M}'(p_0) < 0$ for every $p_0 \in (0,1)$, we shall balance the positive terms in the right-hand side of (\ref{mass-pi-2})
with the negative terms.

We combine the second and third terms in the right-hand side of (\ref{mass-pi-2}) after multiplication
by $(E_0(p_0)+A(p_*))$ and obtain:
\begin{eqnarray*}
I & := & (E_0(p_0)+A(p_*)) \left[ 2N \mathcal{T}(p_0) \mathcal{B}'(p_0) + \mathcal{T}(p_0) \frac{p_0}{\sqrt{1-p_0^2}} \right] \\
& = & 2N \mathcal{T}(p_0) \left[ -\frac{A(p_*)p_0^2}{4N^2q_0} + \frac{A'(p_0)}{8}\left( 1-\frac{1}{4N^2} \right)
 \int_{p_0}^{p_+} \frac{1 - 2u^2}{v} du \right] \\
 && \qquad \qquad \qquad + \mathcal{T}(p_0) \frac{p_0}{\sqrt{1-p_0^2}} (E_0(p_0) + A(p_*)) \\
& = &  - \left(1 - \frac{1}{4N^2} \right) \mathcal{T}(p_0) \left[
 \frac{p_0 A(p_0)}{\sqrt{1-p_0^2}}  + \frac{N A'(p_0)}{4}
 \int_{p_0}^{p_+} \frac{2u^2 - 1}{v} du \right]
\end{eqnarray*}
where we have used the relations $E_0(p_0) = -\left(1 - \frac{1}{4N^2} \right) A(p_0)$ and $q_0 = \frac{1}{2N} \sqrt{A(p_0)}$.
The first term in $I$ is already negative, however, the second term in $I$ is sign-indefinite.

For $p_0 \in (p_*,1)$, the second term in $I$ is positive because $A'(p_0) < 0$  and $2u^2 - 1 > 0$ for $u \in [p_0,p_+]$.
Using the integration by parts, we write
$$
\int_{p_0}^{p_+} \frac{2u^2-1}{v} du = -\int_{p_0}^{p_+} \frac{A'(u) du}{2u \sqrt{E + A(u)}} =
\frac{q_0}{p_0} -
\int_{p_0}^{p_+} \frac{v}{u^2} du.
$$
Substituting this expression into the expression for $I$ yields
\begin{eqnarray*}
I &=& - \left(1 - \frac{1}{4N^2} \right) \mathcal{T}(p_0) \left[
 \frac{p_0 (1 + 2 p_0^2) \sqrt{1-p_0^2}}{4}  - \frac{N A'(p_0)}{4}
 \int_{p_0}^{p_+} \frac{v}{u^2} du \right]
\end{eqnarray*}
which is negative since $A'(p_0) \leq 0$ for $p_0 \in [p_*,1)$.
Hence, $\mathcal{M}'(p_0) < 0$ for $p_0 \in [p_*,1)$.

For $p_0 \in (0,p_*)$, we have $A'(p_0) > 0$ but $2u^2 - 1$ is sign-indefinite for $u \in [p_0,p_+]$.
We combine the second term in $I$ and the second term in the right-hand side
of (\ref{period-d}), which appears in the first term of the right-hand side of (\ref{mass-pi-2}) after multiplication
by $(E_0(p_0)+A(p_*))$. All other terms
in the right-hand side of (\ref{mass-pi-2}) are negative. Hence, we consider
\begin{eqnarray*}
II & :=& - \frac{N A'(p_0)}{4}\left( 1-\frac{1}{4N^2} \right)
\left[ \mathcal{T}(p_0) \int_{p_0}^{p_+} \frac{2u^2-1}{v} du +
\mathcal{B}(p_0) \int_{p_0}^{p_+} \frac{1-2u^2}{u^2 v} du \right] \\
& = & - \frac{N A'(p_0)}{4}\left( 1-\frac{1}{4N^2} \right)
\left[ \left( \int_{p_0}^{p_+} \frac{u^2 du}{v} \right)
\left( \int_{p_0}^{p_+} \frac{du}{u^2v} \right) -
\left( \int_{p_0}^{p_+} \frac{du}{v} \right)^2 \right],
\end{eqnarray*}
where $A'(p_0) > 0$ if $p_0 \in (0,p_0)$. Thanks to the Cauchy--Schwarz inequality
$$
\int_{p_0}^{p_+} \frac{du}{v} = \int_{p_0}^{p_+} \frac{u}{\sqrt{v}} \frac{du}{u \sqrt{v}} \leq
\left( \int_{p_0}^{p_+} \frac{u^2 du}{v} \right)^{1/2}
\left( \int_{p_0}^{p_+} \frac{du}{u^2v} \right)^{1/2},
$$
the expression in $II$ is negative. Hence, $\mathcal{M}'(p_0) < 0$ for $p_0 \in (0,p_*)$.
\end{proof}

\subsection{Proof of Theorem \ref{global-existence}}

By monotonicity of the period function $\mathcal{T}(p_0)$ in $p_0$ given by Lemma \ref{global-decreasing}
and by the nonlinear equation (\ref{root-i}),
we have a diffeomorphism $(0,1) \ni p_0 \mapsto \epsilon(p_0) = \frac{1}{\pi} \mathcal{T}(p_0) \in (0,\infty)$.
Let us show that $\epsilon(p_0) \to 0$ as $p_0 \to 1$ and $\epsilon(p_0) \to \infty$ as $p_0 \to 0$.
Then, since the function $\mathcal{T}(p_0)$ is monotonically decreasing, the range
of the mapping $p_0 \mapsto \epsilon(p_0)$ is indeed $(0,\infty)$.

Recall (\ref{p+p-}) with $E(u,v) = E_0(p_0)$ given by (\ref{energy-level}). 
Also recall the ordering given by (\ref{turning-points}). For a given 
$p_0 \in (0,1)$, the equation $E_0(p_0) = -A(p_+)$ determines $p_+$ 
from the nonlinear equation
\begin{equation}
\label{p+p0}
p_+^2 (1- p_+^2) = \left(1 - \frac{1}{4 N^2} \right) p_0^2 (1 - p_0^2).
\end{equation}
Since $p_+ \in (p_0,1)$, it follows from (\ref{p+p0}) 
that $p_+ \to 1$ as $p_0 \to 1$
so that $|p_+ - p_0| \to 0$ as $p_0 \to 1$. Since the weakly singular integrand below
is integrable, we have
\begin{equation}
\label{limit-integral}
\mathcal{T}(p_0) = \int_{p_0}^{p_+} \frac{du}{\sqrt{E + A(u)}} =
\int_{p_0}^{p_+} \frac{du}{\sqrt{A(u) - A(p_+)}} \to 0 \quad \mbox{\rm as} \;\; p_0 \to 1,
\end{equation}
hence $\epsilon(p_0) = \frac{1}{\pi} \mathcal{T}(p_0) \to 0$ as $p_0 \to 1$. For every $0<p_0<p_+<1$ we obtain
\begin{equation}
\label{divergent-integral}
\mathcal{T}(p_0) =
\int_{p_0}^{p_+} \frac{du}{\sqrt{A(u) - A(p_+)}} \geq
\int_{p_0}^{p_+} \frac{du}{u \sqrt{1-u^2}}.
\end{equation}
Since $p_+ \in (p_*,1)$, it follows from (\ref{p+p0}) 
that $p_+ \to 1$ as $p_0 \to 0$. Since 
$$
\int_0^1 \frac{du}{u \sqrt{1 - u^2}} = \infty,
$$
we have $\mathcal{T}(p_0) \to \infty$ as $p_0 \to 0$, hence $\epsilon(p_0) = \frac{1}{\pi} \mathcal{T}(p_0) \to \infty$ as $p_0 \to 0$.

Thus, for each $p_0 = {\rm sech}(a) \in (0,1)$ or equivalently, for each $a \in (0,\infty)$,
there exists exactly one root $\epsilon \in (0,\infty)$ of the nonlinear equation (\ref{root-i}).
By using $\omega = -\epsilon^2$, the scaling transformation (\ref{scaling-transform}),
the soliton (\ref{soliton}), and the symmetry (\ref{sym-state-scaled}),
we obtain a unique solution $\Phi \in H^2_{\rm NK}(\Gamma_N)$ satisfying the stationary NLS equation
(\ref{nls-stat}), which is symmetric on each loop parameterized
by $[-\pi,\pi]$ and is monotonically decreasing on $[0,\pi]$ and $[0,\infty)$.
Moreover, by Lemma \ref{global-decreasing} and by the construction,
the map $(-\infty,0) \ni \omega \mapsto \Phi(\cdot,\omega) \in H^2_{\rm NK}(\Gamma_N)$ is $C^1$.

Let us now define the mass $\mu(\omega) := Q(\Phi(\cdot,\omega))$ on the unique solution $\Phi \in H^2_{\rm NK}(\Gamma_N)$
for each $\omega \in (-\infty,0)$. By Lemma \ref{lemma-mass-monotonicity}, the mapping $(0,1) \ni p_0 \mapsto \mathcal{M}(p_0) \in (0,\infty)$
is $C^1$ and monotonically decreasing, where $\mathcal{M}(p_0) = \pi \mu(\omega)$.
Since the mapping $(0,\infty) \ni \epsilon \mapsto p_0(\epsilon)$ is $C^1$ and
monotonically decreasing, whereas $\omega = -\epsilon^2$, we obtain that the
mapping $(-\infty,0) \ni \omega \mapsto \mu(\omega) \in (0,\infty)$ is $C^1$ and monotonically decreasing,
which follows from the chain rule
\begin{equation}
\label{chain-rule}
\frac{d \mu}{d \omega} = \frac{d \mu}{dp_0}\frac{dp_0}{d\epsilon} \frac{d\epsilon}{d \omega}.
\end{equation}
It remains to prove that $\mu(\omega) \to 0$ as $\omega \to 0$ and $\mu(\omega) \to \infty$ as $\omega \to -\infty$.

Since $\epsilon \to 0$ as $p_0 \to 1$, it follows from (\ref{mass-explicit})
that $\mu \to 0$ as $p_0 \to 1$. Moreover, the first term in (\ref{mass-explicit}) is smaller than
the second term in (\ref{mass-explicit}) due to
\begin{equation}
\label{limit-integral-again}
\int_{p_0}^{p_+} \frac{u^2 du}{\sqrt{E + A(u)}} =
\int_{p_0}^{p_+} \frac{u^2 du}{\sqrt{A(u) - A(p_+)}} \to 0 \quad \mbox{\rm as} \;\; p_0 \to 1.
\end{equation}
Hence, it follows from (\ref{mass-explicit}) that 
\begin{equation}
\lim_{\epsilon \to 0} \frac{\mu}{\epsilon} = 1.
\end{equation}
On the other hand, since $\epsilon \to \infty$ as $p_0 \to 0$,
we obtain $\mu \to \infty$ as $p_0 \to 0$. Moreover, the second term in (\ref{mass-explicit}) is smaller than
the first term in (\ref{mass-explicit}) since $1 - \sqrt{1-p_0^2} \to 0$ as $p_0 \to 0$ whereas 
\begin{equation}
\label{limit-integral-again-again}
\int_{p_0}^{p_+} \frac{u^2 du}{\sqrt{E + A(u)}} \to \int_0^1 \frac{u du}{\sqrt{1-u^2}} = 1 \quad \mbox{\rm as} \;\; p_0 \to 0.
\end{equation}
Hence, it follows from (\ref{mass-explicit}) that 
\begin{equation}
\lim_{\epsilon \to \infty} \frac{\mu}{\epsilon} = 2N.
\end{equation}
Thus, the mass $\mu(\omega)$ in (\ref{mass-explicit})
satisfies $\mu(\omega) \to 0$ as $\omega \to 0$ and $\mu(\omega) \to \infty$
as $\omega \to -\infty$. The proof of Theorem \ref{global-existence} is complete.

\begin{remark}
\label{on-u-v}
For every $\epsilon>0$, the solution $u_1$ to the boundary-value problem (\ref{bvp-i})
which corresponds to Theorem \ref{global-existence} is given by a positive, even function on $[-\pi \epsilon, \pi \epsilon]$
such that $u'(z) < 0$ for every $z \in (0, \pi \epsilon]$.
\end{remark}

\begin{remark}
\label{T-plus-C-divergence}
In the proof of Theorem \ref{global-existence}, we show that $\mathcal{T}(p_0) = T_+(p_0, \frac{1}{2N}\sqrt{A(p_0)}) \to \infty$ as
$p_0 \to 0$ using the estimate (\ref{divergent-integral}) in the limit $p_+ \to 1$ as $p_0 \to 0$.
In a similar manner, we can prove that $T_+(p_0, C \sqrt{A(p_0)}) \to \infty$ as $p_0 \to 0$ for any positive constant $C$.
\end{remark}

\section{Bifurcations from the positive single-lobe symmetric state}
\label{sec-global-stability}

By Theorem \ref{global-existence}, for every $\omega<0$, there exists a unique positive
single-lobe symmetric state $\Phi \in H_{\rm NK}(\Gamma_N)$. For every such $\Phi$, we define the self-adjoint operator
$\mathcal{L}: H_{\rm NK}(\Gamma_N) \subset L^2(\Gamma_N) \to L^2(\Gamma_N)$ as in (\ref{Jacobian}).
Thanks to the exponential decay of $\phi_0(x) \to 0$ as $x \to \infty$,
by Weyl's theorem, the spectrum of $\mathcal{L}$ in $L^2(\Gamma_N)$ consists of finitely many isolated eigenvalues of finite
multiplicities below $|\omega|$, which is the infimum of the continuous spectrum of $\mathcal{L}$ in (\ref{abs-cont-part}).

Here we prove Theorem \ref{global-stability}. We shall first group the negative and zero eigenvalues of $\mathcal{L}$
into three sets. By using the Sturm comparison theorem
and the analytical properties of the period function $T_+(p_0,q_0)$, we control the first eigenvalues in each set.
In the end, we prove that there exists only one value of $\omega \in (-\infty,0)$, labeled as $\omega_*$, for which $z(\mathcal{L}) = N-1$,
whereas $z(\mathcal{L}) = 0$ for $\omega \neq \omega_*$. We also show that $n(\mathcal{L}) = 1$ for $\omega \in (\omega_*,0)$
and $n(\mathcal{L}) = N$ for $\omega \in (-\infty,\omega_*)$.

Note that we avoid the surgery techniques for the count of nodal domains \cite{BKKM}, which do not provide precise information
on the Morse index for graphs with positive Betti number. Instead, we explore Sturm's comparison theory on bounded intervals
and further analytical properties of the period function. In particular, we show that the bifurcation at $\omega_*$ is related
to the existence of a critical point of the period function $T_+(p_0,q_0)$ with respect to the parameter $q_0$ at the corresponding
level curve on the $(u,v)$-plane.

\subsection{Eigenvalues of $\mathcal{L}$}

Let us consider the spectral problem $\mathcal{L} \Upsilon = \epsilon^2 \lambda \Upsilon$,
where $\Upsilon \in H^2_{\rm NK}(\Gamma_N)$ is an eigenfunction of $\mathcal{L}$ corresponding to the eigenvalue
$\epsilon^2 \lambda$ and the parameter $\epsilon$ is used to express $\omega = -\epsilon^2$ and the positive single-lobe symmetric state $\Phi$ by using
the scaling transformation (\ref{scaling-transform}) with $(u_1, u_2, \dots, u_N, u_0)$.
By using a similar transformation with $(v_1, v_2, \dots, v_N, v_0)$ for the eigenfunction $\Upsilon$,
we rewrite the spectral problem $\mathcal{L} \Upsilon = \epsilon^2 \lambda \Upsilon$ as the
following boundary-value problem:
\begin{equation}
\label{L-spectral-sym}
\left\{ \begin{array}{l}
-v_j''(z) + v_j(z) - 6 u_j(z)^{2} v_j(z) = \lambda v_j(z),
\quad z \in (-\pi \epsilon,\pi \epsilon), \quad j \in \{1,2,\dots,N\}, \\
-v_0''(z) + v_0(z) - 6 u_0(z)^{2} v_0(z) = \lambda v_0(z), \quad z > 0, \\
v_1(\pm \pi \epsilon) = v_2(\pm \pi \epsilon) = \dots = v_N(\pm \pi \epsilon) = v_0(0), \\
\sum_{j=1}^N v_j'(\pi \epsilon) - v_j'(-\pi \epsilon) = v_0'(0).
 \end{array} \right.
\end{equation}
In what follows, $\epsilon > 0$ is a fixed parameter
and the statements hold for every $\epsilon > 0$.

Due to the symmetry (\ref{sym-state-scaled}) on the positive single-lobe symmetric state $\Phi$, we have the following trichotomy.

\begin{lemma}
\label{lem-vertex-cond}
The set of eigenvalues $\lambda$ of the boundary-value problem (\ref{L-spectral-sym}) with $\lambda \leq 0$ is a union of sets 
$\mathcal{S}_1$, $\mathcal{S}_2$, and $\mathcal{S}_3$, where
\begin{itemize}
\item $\mathcal{S}_1$ consists of simple eigenvalues with $v_0 \not \equiv 0$ and even $v_1 = \dots = v_N$ on $[-\pi \epsilon, \pi \epsilon]$;
\item $\mathcal{S}_2$ consists of eigenvalues of multiplicity $(N-1)$ with $v_0 \equiv 0$ and even $v_j$ on $[-\pi \epsilon,\pi \epsilon]$ for every $j$;
\item $\mathcal{S}_3$ consists of eigenvalues of multiplicity $N$ with $v_0 \equiv 0$ and odd $v_j$ on $[-\pi \epsilon,\pi \epsilon]$ for every $j$.
\end{itemize}
Moreover, $\mathcal{S}_1 \cap \mathcal{S}_2 = \emptyset$ and $\mathcal{S}_2 \cap \mathcal{S}_3 = \emptyset$.
\end{lemma}

\begin{proof}
If $v_0 \not \equiv 0$, there exists only one solution of the second-order equation for $v_0$ which decays
to $0$ as $z \to \infty$, as is shown, e.g., in \cite[Lemma 5.1]{KP2}.
Hence, if $v_0 \not \equiv 0$, the multiplicity of $\lambda$ in the set $\mathcal{S}_1$ is one. In fact, the solution $v_0$
(up to normalization) is available in the following analytic form:
\begin{equation}
\label{v-0}
v_0(z) = V_0(z; \lambda): = e^{-\sqrt{1-\lambda} z} \frac{3 - \lambda + 3 \sqrt{1-\lambda} \tanh(z + a) - 3 {\rm sech}^2(z+a)}{3 - \lambda + 3 \sqrt{1-\lambda}}.
\end{equation}
Since $a>0$ for the symmetric state in Theorem \ref{global-existence}, 
it follows from (\ref{v-0}) for every $\lambda \leq 0$ that $v_0(z) > 0$ for every $z \geq 0$.

Thanks to the symmetry condition (\ref{sym-state-scaled}) and the even parity of $u_j$ in the symmetric state of Theorem \ref{global-existence}, if $v_j$ satisfies the boundary conditions
$v_j(-\pi \epsilon) = v_j(\pi \epsilon) = v_0(0) \neq 0$, then $v_j$ is even and $v_1 = v_2 = \dots = v_N$. Hence, $v_j$ is a solution of the 
following boundary-value problem:
\begin{equation}
\label{sp2}
\SP_1: \quad \quad
\left\{ \begin{array}{l}
-v''(z) + v(z) - 6 u_1(z)^{2} v(z) = \lambda v(z),
\quad z \in (-\pi \epsilon,\pi \epsilon), \\
v(-\pi \epsilon) = v(\pi \epsilon) = V_0(0; \lambda),\\
2N v'(\pi \epsilon) = V_0'(0; \lambda),
 \end{array} \right.
\end{equation}
where the prime denotes the derivative in $z$.

If $v_0 \equiv 0$, then  $v_j$ is a solution of 
the following Sturm--Liouville boundary-value problem
\begin{equation}
\label{sp1}
\SP_2: \quad \quad
\left\{ \begin{array}{l} -v''(z) + v(z) - 6 u_1(z)^{2} v(z) = \lambda v(z), \quad z \in (-\pi \epsilon,\pi \epsilon), \\
v(-\pi \epsilon) = v(\pi \epsilon) = 0. \end{array} \right.
\end{equation}
If $v_1$ is a solution to $\SP_2$, then so are $v_2, \dots, v_N$. By the linear superposition principle and the even parity of $u_1$,
the solution $v$ to $\SP_2$ is generally a linear combination of the even and odd functions. 

If $v_1$ is even, then the derivative boundary condition in (\ref{L-spectral-sym})
yields a nontrivial constraint:
\begin{equation}
\label{derivative-constraint}
\sum_{j=1}^N v_j'(\pi \epsilon) = 0
\end{equation}
and since $v'(\pi \epsilon) \neq 0$ for a nonzero solution of the spectral problem (\ref{sp1}),
then there are only $N-1$ combinations of $v_1,v_2,\dots,v_N$ satisfying the constraint (\ref{derivative-constraint}).
Hence the eigenvalue $\lambda$ in the set $\mathcal{S}_2$ has multiplicity $(N-1)$.

If $v_1$ is odd, then the derivative boundary condition in (\ref{L-spectral-sym}) is trivially satisfied,
hence there are $N$ linearly independent functions $v_1,v_2,\dots,v_N$ and the eigenvalue $\lambda$
in the set $\mathcal{S}_3$ has multiplicity $N$.

The boundary-value problem $\SP_2$ is the Sturm--Liouville problem with the Dirichlet boundary conditions,
hence its eigenvalues are all simple. This implies $\mathcal{S}_2 \cap \mathcal{S}_3 = \emptyset$.

Each $v(z)$ satisfying $\SP_1$ is even on $(-\pi \epsilon, \pi \epsilon)$.
Since $V_0(0;\lambda) > 0$ for every $\lambda \leq 0$, this implies that $v(\pm \pi \epsilon) > 0$
so that $v(z)$ does not satisfy $\SP_2$ and vice versa. This implies that
$\mathcal{S}_1 \cap \mathcal{S}_2 = \emptyset$.
\end{proof}

Let us order the eigenvalues in the spectral problem (\ref{L-spectral-sym}) counting their multiplicities as follows:
\begin{equation}
\label{order}
\lambda_1 \leq \lambda_2 \leq \lambda_3 \leq \dots
\end{equation}
By Lemma \ref{lem-vertex-cond}, each eigenvalue of the spectral problem (\ref{L-spectral-sym}) corresponds
to either $v_0 \not \equiv 0$ or $v_0 \equiv 0$, and so, the set of eigenvalues (counting multiplicities) in
the spectral problem (\ref{L-spectral-sym}) is in one-to-one correspondence with the union of
sets of eigenvalues of the boundary-value problems $\SP_1$ and $\SP_2$. Next, we control the sign of the first eigenvalues of the boundary-value problems $\SP_1$ and $\SP_2$.

\subsection{Eigenvalues of the boundary-value problems $\SP_1$ and $\SP_2$.}

We start with the first eigenvalue $\lambda_1$ of the spectral problem (\ref{L-spectral-sym}).
By the Rayleigh-Ritz principle (see \cite[Lemma 5.12]{Teschl}), this eigenvalue can be characterized variationally as follows:
\begin{equation}
\label{min-L1}
\lambda_1 = \inf_{\tilde{\Upsilon} \in H^1_C(\tilde{\Gamma}_N)} \left\{ \langle \tilde{\mathcal{L}} \tilde{\Upsilon},
\tilde{\Upsilon} \rangle_{L^2(\tilde{\Gamma}_N)} : \quad \| \tilde{\Upsilon} \|_{L^2(\tilde{\Gamma}_N)} = 1 \right\},
\end{equation}
where $\tilde{\mathcal{L}}$ is the $\epsilon$-scaled version of the linearized operator $\mathcal{L}$
and $\tilde{\Upsilon} = (v_1,v_2,\dots,v_N,v_0)$ is the scaled eigenfunction on
the $\epsilon$-scaled graph $\tilde{\Gamma}_N$. The following lemma states that $\lambda_1 < 0$ and $\lambda_1 < \lambda_2$ in (\ref{order}).

\begin{lemma}
\label{first-eig-L}
Let $\lambda = \gamma_1$ be the first eigenvalue of $\SP_1$.
Then, $\lambda_1 = \gamma_1$, moreover, $\lambda_1$ is negative and simple
with a strictly positive eigenfunction $\tilde{\Upsilon}_1$ on $\tilde{\Gamma}_N$.
\end{lemma}

\begin{proof}
It follows from (\ref{quad-form-phi}) that $\lambda_1$ is negative. By the variational analysis on graphs,
as in \cite[Proposition 3.3]{AdamiCV}, the infimum (\ref{min-L1}) is uniquely attained at
some strictly positive $\tilde{\Upsilon}_1$ which belongs to $H^2_{HK}(\tilde{\Gamma}_N)$.
This positive $\tilde{\Upsilon}_1 = (v_1, v_2, \dots, v_N, v_0)$ is the corresponding eigenfunction in
the spectral problem (\ref{L-spectral-sym}). Hence, $v_0 \not \equiv 0$ and so, $\lambda_1$
coincides with the first eigenvalue $\gamma_1$ in the set $\mathcal{S}_1$ by Lemma \ref{lem-vertex-cond}.
Since $\mathcal{S}_1 \cap \mathcal{S}_2 = \emptyset$, whereas the first eigenvalue
of the Sturm--Liouville problem (\ref{sp1}) corresponds to an even eigenfunction,
it follows that $\lambda_1$ is not an eigenvalue in $\SP_2$, hence $\lambda_1$ is simple.
\end{proof}

Before proceeding with other eigenvalues, we review the Sturm--Liouville theory for
the boundary-value problem (\ref{sp1}). The following three propositions are well-known,
see, e.g., \cite[Section 5.5]{Teschl}.

\begin{proposition}
\label{SLP-eigenfunctions}
Let $\beta_n$ be the $n$-th eigenvalue of the Sturm--Liouville problem (\ref{sp1}) for $n \in \mathbb{N}$.
Then, $\beta_n$ is simple and its corresponding eigenfunction is even (odd) if
$n$ is odd (even). Moreover, the eigenfunction vanishes on
$(-\pi \epsilon, \pi \epsilon)$ at exactly $n-1$ nodal points.
\end{proposition}

\begin{proposition}
\label{positivity-sp1}
Let $\beta_1$ be the first eigenvalue of the Sturm--Liouville problem (\ref{sp1}).
Then, for $\beta < \beta_1$, the initial value problem
\begin{equation}
\label{ivp-interval}
\left\{ \begin{array}{l} -v''(z) + v(z) - 6 u_1(z)^{2} v(z) = \beta v(z), \quad z \in (-\pi \epsilon,\pi \epsilon), \\
v(0) = 1, \quad v'(0) = 0, \end{array} \right.
\end{equation}
has the unique solution $v$, which is even and strictly positive on
$[-\pi \epsilon, \pi \epsilon]$. For $\beta > \beta_1$, the unique solution $v$ is sign-indefinite.
\end{proposition}

\begin{proposition}
\label{odd-function-sp1}
Let $\beta_2$ be the second eigenvalue of the Sturm--Liouville problem (\ref{sp1}).
Then, for $\beta < \beta_2$, the initial value problem
\begin{equation}
\label{ivp-interval-odd}
\left\{ \begin{array}{l} -v''(z) + v(z) - 6 u_1(z)^{2} v(z) = \beta v(z), \quad z \in (-\pi \epsilon,\pi \epsilon), \\
v(0) = 0, \quad v'(0) = 1, \end{array} \right.
\end{equation}
has the unique solution $v$, which is odd on $[-\pi \epsilon,\pi \epsilon]$ and strictly positive on
$(0,\pi \epsilon]$. For $\beta > \beta_2$, the unique solution $v$ is sign-indefinite on $(0,\pi \epsilon]$.
\end{proposition}

The following three lemmas state the ordering between the second eigenvalue of the boundary-value problem $\SP_1$ and the
first two eigenvalues of the boundary-value problem $\SP_2$. These eigenvalues contribute to the order of
eigenvalues $\lambda_2$ and $\lambda_3$ in (\ref{order}).

\begin{lemma}
\label{second-eig-L}
Let $\lambda = \beta_1$ be the first eigenvalue of the boundary-value problem $\SP_2$ in (\ref{sp1})
and $\lambda = \gamma_2$ be the second eigenvalue of the boundary-value problem $\SP_1$ in (\ref{sp2}).
If $\lambda_2$ in (\ref{order}) is negative or zero, then $\lambda_2 = \beta_1 < \gamma_2$.
Moreover, the eigenvalue $\lambda_2$ has an algebraic multiplicity $(N-1)$ and is associated
with $(N-1)$ even eigenfunctions $\tilde{\Upsilon}$ on $\tilde{\Gamma}_N$.
\end{lemma}

\begin{proof}
Let $\lambda_2$ be the second eigenvalue of the spectral problem (\ref{L-spectral-sym})
with an eigenfunction $\tilde{\Upsilon}_2 = (v_1, v_2, \dots, v_N, v_0)$.
If $\lambda_2 \in (-\infty, 0]$, then either $v_0 \equiv 0$ or
$v_0(z)>0$ for all $z\geq 0$ thanks to the analytic form (\ref{v-0}).

If $v_0 \equiv 0$, then $\lambda_2$ coincides with the first eigenvalue in $\SP_2$, which is $\beta_1$.
Then, by Proposition \ref{SLP-eigenfunctions}, each $v_j$ is even and $\lambda_2$ belongs to the set $\mathcal{S}_2$ in Lemma \ref{lem-vertex-cond}.
Since $\mathcal{S}_1 \cap \mathcal{S}_2 = \emptyset$ in Lemma \ref{lem-vertex-cond},
then $\lambda_2 \neq \gamma_2$, and since $\gamma_2$ is also an eigenvalue of the spectral problem
(\ref{L-spectral-sym}), it follows that $\lambda_2 < \gamma_2$.

If $v_0(z)>0$ for all $z\geq 0$, we have that $\lambda_2 = \gamma_2$ belongs to set $\mathcal{S}_1$.
Since $\mathcal{S}_1 \cap \mathcal{S}_2 = \emptyset$ in Lemma \ref{lem-vertex-cond},
we have $\lambda_2 \neq \beta_1$, and since $\beta_1$ is also an eigenvalue of the spectral problem
(\ref{L-spectral-sym}), it follows that $\lambda_2 < \beta_1$.
Therefore, each even $v_j$ is constant proportional to the unique solution of the initial-value problem (\ref{ivp-interval})
with $\beta = \lambda_2 < \beta_1$. By Proposition \ref{positivity-sp1}, each $v_j$ is strictly positive on $[-\pi \epsilon,\pi \epsilon]$.
As a result, the eigenfunction $\tilde{\Upsilon}_2$ is strictly positive on $\tilde{\Gamma}_N$.
Since the eigenfunction $\tilde{\Upsilon}_1$ in Lemma \ref{first-eig-L} is also strictly positive on $\tilde{\Gamma}_N$,
the $L^2(\tilde{\Gamma}_N)$-inner product of $\tilde{\Upsilon}_1$ and $\tilde{\Upsilon}_2$ is not zero, which
contradicts to the orthogonality of eigenfunctions for distinct eigenvalues to the spectral problem (\ref{L-spectral-sym}).
Hence $v_0 \not \equiv 0$ is impossible so that $\lambda_2 = \beta_1 < \gamma_2$.
\end{proof}

\begin{lemma}
\label{gamma-2-nonzero}
Let $\lambda = \gamma_2$ be the second eigenvalue of the boundary-value problem $\SP_1$ in (\ref{sp2}).
Then, $\gamma_2 \neq 0$.
\end{lemma}

\begin{proof}
To show that $\gamma_2 \neq 0$, we consider the boundary-value problem
\begin{equation}
\label{bvp-i-2}
\left\{ \begin{array}{l} -u''(z) + u(z) - 2 u(z)^3 = 0, \quad
z \in (-T_+(p_0, q_0), T_+(p_0, q_0)), \\
u(-T_+(p_0, q_0)) = u(T_+(p_0, q_0)) = p_0, \\
u'(-T_+(p_0, q_0)) = -u'(T_+(p_0, q_0))= q_0,
\end{array} \right.
\end{equation}
where $T_+(p_0, q_0)$ is defined in (\ref{period}) with two independent parameters $p_0 \in (0,1)$ and $q_0 \in (0,\infty)$.
The unique solution of the boundary-value problem (\ref{bvp-i}) is obtained at $q_0 = \frac{1}{2N} \sqrt{A(p_0)}$, for which
$T(p_0,q_0) = \mathcal{T}(p_0) = \pi \epsilon$ in (\ref{root-i}). We use the notation $u(z) = u(z; p_0, q_0)$
and recall that $u(z; p_0, q_0)$ is a $C^1$ function with respect to  $p_0$ and $q_0$.

Define $s(z; p_0, q_0) := \partial_{q_0} u(z; p_0, q_0)$. Then, $s(z; p_0, q_0)$ is an even solution
of the following differential equation:
\begin{equation}
\label{eqn-s}
-s''(z) + s(z) - 6 u(z)^2 s(z) = 0, \quad
z\in (-T_+(p_0, q_0), T_+(p_0, q_0)).
\end{equation}
Moreover, since $u(0; p_0, q_0) = p_+$, where $p_+$ is defined by (\ref{p+p-}),
we have $s(0; p_0, q_0) = \partial_{q_0} p_+$, where $\partial_{q_0} p_+ \neq 0$.
Indeed, after differentiating $E(u,v) = q_0^2 - A(p_0) = -A(p_+)$ with respect to $q_0$, we have
$$
2q_0 = 2p_+(2p_+^2 -1) \partial_{q_0} p_+.
$$
Since $p_+ > p_* = \frac{1}{\sqrt{2}}$ and $q_0>0$, we have
$s(0; p_0, q_0) = \partial_{q_0} p_+ > 0$.

Similarly, we define $t(z; p_0, q_0) := \partial_{p_0} u(z; p_0, q_0)$, and notice that $t(z; p_0, q_0)$ is also an even
solution of the differential equation (\ref{eqn-s}). Differentiating $E(u,v) = q_0^2 - A(p_0) = -A(p_+)$ with respect to $p_0$ yields
$$
2p_0(2p_0^2-1) = 2p_+(2p_+^2-1)\partial_{p_0} p_+.
$$
If $p_0 = p_* = 1/\sqrt{2}$, then
$t(0; p_0, q_0) = \partial_{p_0} p_+ = 0$ so that
$t(z; p_0, q_0) \equiv 0$ is zero solution to (\ref{eqn-s}). Otherwise,
$t(0; p_0, q_0) = \partial_{p_0} p_+ \neq 0$ and $t(z; p_0, q_0)$ is a nonzero even solution
to (\ref{eqn-s}).

For $q_0 = \frac{1}{2N} \sqrt{A(p_0)}$,
we have $T_+(p_0,q_0) = \mathcal{T}(p_0) = \pi \epsilon$, and since $s(0;p_0,q_0) \neq 0$,
the solution $s(z; p_0, q_0)$ of the differential equation (\ref{eqn-s}) with this $q_0$
is constant proportional to the unique solution to the initial-value problem (\ref{ivp-interval})
with $\beta = 0$. Moreover, if $p_0 \neq p_*$, the above statement also applies to $t(z; p_0, q_0)$,
so that there exists a nonzero constant $C$ such that $t(z; p_0, q_0) = C s(z; p_0, q_0)$.

If $\lambda = \gamma_2 = 0$ in $\SP_1$, we know from (\ref{v-0}) that
$V_0(z; 0) = \frac{1}{2} \sech (z+a) \tanh(z+a)$, where $a$ is related to $p_0$ by $p_0 = \sech(a)$.
Moreover, by Lemma \ref{global-decreasing}, $a$ and $p$ are $C^1$ functions of $\epsilon$, that is
$a = a(\epsilon)$ and $p_0 = p_0(\epsilon)$. We also define $\varphi(z) := \sech(z)$, and rewrite the boundary values in the spectral problem $\SP_1$ as follows:
\begin{equation}
\label{v-0-values}
V_0(0; 0) = - \frac{1}{2} \varphi'(a), \quad {\rm and} \quad
V_0'(0; 0) = -\frac{1}{2} \varphi''(a).
\end{equation}
Solution to the differential equation in $\SP_1$ for $\lambda = 0$
is given by $v(z) = C_0 s(z; p_0, q_0)$, where $q_0 = \frac{1}{2N} p_0\sqrt{1-p_0^2}$
and $C_0$ is a real constant. By using the boundary conditions in $\SP_1$ and the representation (\ref{v-0-values}),
we obtain the following system of equations:
\begin{equation}
\label{bc-gamma-0}
\left\{ \begin{array}{l}
-2 C_0 s(\pi \epsilon; p_0, q_0) = \varphi'(a) \\
-4 N C_0 s'(\pi \epsilon; p_0, q_0) = \varphi''(a),
\end{array} \right.
\end{equation}
where $T_+(p_0, q_0) = \mathcal{T}(p_0) = \pi \epsilon$ by (\ref{root-i}).
Since $a(\epsilon) > 0$ for every positive $\epsilon$, we know $\varphi'(a) \neq 0$ and from (\ref{bc-gamma-0}) we obtain
\begin{equation}
\label{quotient-1}
\frac{2N s'(\pi \epsilon; p_0, q_0)}{s(\pi \epsilon; p_0, q_0)} =
\frac{\varphi''(a)}{\varphi'(a)}.
\end{equation}

On the other hand, using that $p_0 = \varphi(a)$ and
$q_0 = - \frac{1}{2N}\varphi'(a)$ we rewrite the boundary values in (\ref{bvp-i-2}) at
$T_+(p_0, q_0) = \mathcal{T}(p_0) = \pi \epsilon$ to be
\begin{equation}
\label{bc-u-gamma-0}
\left\{ \begin{array}{l}
u (\pi \epsilon; p_0, q_0) = \varphi(a) \\
2N u'(\pi \epsilon; p_0, q_0) = \varphi'(a),
\end{array} \right.
\end{equation}

For $p_0 \neq p_*$, we use that $a$, $p_0$, and $q_0$ are $C^1$ functions of $\epsilon$,
hence we differentiate (\ref{bc-u-gamma-0}) with respect to $\epsilon$ and since $t(z; p_0, q_0) = C s(z; p_0,q_0)$ we obtain
\begin{equation}
\label{bc-u-gamma-0-2}
\left\{ \begin{array}{l}
s (\pi \epsilon; p_0, q_0) \left[C p_0'(\epsilon) + q_0'(\epsilon) \right] = \varphi'(a) \left[ a'(\epsilon) - \frac{\pi}{2N} \right] \\
2N s' (\pi \epsilon; p_0, q_0) \left[C p_0'(\epsilon) + q_0'(\epsilon) \right] = \varphi''(a) \left[ a'(\epsilon) - 2 \pi N \right]
\end{array} \right.
\end{equation}
Note that $C p_0'(\epsilon) + q_0'(\epsilon) \neq 0$ since
$\varphi'(a) \neq 0 \neq \varphi''(a)$ for $p_0 \neq p_*$.
Hence, it follows from (\ref{bc-u-gamma-0-2}) that
$$
\frac{2N s'(\pi \epsilon; p_0, q_0)}{s(\pi \epsilon; p_0, q_0)} =
\frac{\varphi''(a) \left[ a'(\epsilon) - 2 \pi N \right]}{\varphi'(a) \left[ a'(\epsilon) - \frac{\pi}{2N} \right]},
$$
which contradicts to (\ref{quotient-1}) since $\varphi''(a) \neq 0$ for $p_0 \neq p_*$.

For $p_0 = p_*$, we have $s'(\pi \epsilon; p_0, q_0) = 0$ by
(\ref{bc-gamma-0}). Then, we differentiate the invariant relation
$ q_0^2 - p_0^2 + p_0^4 = \left[u'(z; p_0, q_0) \right]^2 -
u^2(z; p_0, q_0) +u^4(z; p_0, q_0)$ with respect to $q_0$ and obtain
\begin{equation}
\label{local-der-E}
2q_0 = 2 u'(z; p_0, q_0) s'(z; p_0, q_0) + 2 u(z; p_0, q_0)
\left[2 u^2(z; p_0, q_0) -1 \right] s(z; p_0, q_0).
\end{equation}
For $z = T_+(p_0, q_0) = \mathcal{T}(p_0) = \pi \epsilon$,
we substitute $s'(\pi \epsilon; p_0, q_0) = 0$ and $u(\pi \epsilon; p_0, q_0) = p_*$ in (\ref{local-der-E})
to get $2q_0 = 0$, which is a contradiction. In both cases, $\lambda = \gamma_2 = 0$
is impossible in $\SP_1$.
\end{proof}

\begin{lemma}
\label{second-pos}
Let $\lambda = \beta_2$ be the second eigenvalue of the boundary-value problem $\SP_2$ in (\ref{sp1}).
Then, $\beta_2>0$.
\end{lemma}

\begin{proof}
Define $r(z; p_0, q_0) := u'(z; p_0, q_0)$, where the prime stands for the derivative with respect to $z$.
We have that $r(z; p_0, q_0)$ is odd and that $r'(0; p_0, q_0) = u''(0; p_0, q_0) = (1-2p_+^2)p_+ <0$.
For $q_0 = \frac{1}{2N} \sqrt{A(p_0)}$,
we have $T_+(p_0,q_0) = \mathcal{T}(p_0) = \pi \epsilon$, and since $r'(0;p_0,q_0) \neq 0$,
$r(z;p_0,q_0)$ with this $q_0$ is constant proportional to
the unique solution to the initial-value problem (\ref{ivp-interval-odd}) with $\beta = 0$.
By the construction of $u(z; p_0,q_0)$ in (\ref{bvp-i-2}) and negativity of $r'(0;p_0,q_0)$, the function $-r(z;p_0,q_0)$
with this $q_0$ is strictly positive on $(0, \pi \epsilon]$, and by Proposition \ref{odd-function-sp1},
$0 = \beta < \beta_2$.
\end{proof}

\subsection{Existence of a zero eigenvalue in $\SP_2$.}

It follows from Lemmas \ref{first-eig-L}, \ref{second-eig-L}, \ref{gamma-2-nonzero}, and \ref{second-pos}
that, when the parameter $\epsilon$ is increased, the only eigenvalue of the spectral problem (\ref{L-spectral-sym})
which may cross zero and become the second negative eigenvalue $\lambda_2$ in addition to the eigenvalue $\lambda_1 = \gamma_1$
is the first eigenvalue $\lambda = \beta_1$ of the Sturm--Liouville problem $\SP_2$ in (\ref{sp1}).

Here we study the conditions for $\beta_1$ to become negative from the analytical properties
of the period function $T_+(p_0,q_0)$, which appears in the boundary-value problem (\ref{bvp-i-2}).
The following two lemmas state properties of $T_+(p_0,q_0)$ with respect to $q_0$ separately for
$p_0 \in (0,p_*]$ and $p_0 \in (p_*,1)$.

\begin{lemma}
\label{monotone-on-left}
For every $p_0 \in (0, p_*]$, $T_+(p_0,q_0)$ is a monotonically decreasing function of $q_0$ in $(0,\infty)$.
\end{lemma}

\begin{proof}
By using the same approach as in the proof of Lemma \ref{global-decreasing}, we write
\begin{eqnarray*}
\left[ E_0(p_0,q_0) + A(p_*) \right] T_+(p_0,q_0) & = &
\int_{p_0}^{p_+} \left[3 - \frac{2 (A(u) - A(p_*)) A''(u)}{[A'(u)]^2} \right] v du +
\frac{2(A(p_0)-A(p_*))q_0}{A'(p_0)},
\end{eqnarray*}
where $E_0(p_0,q_0) = q_0^2 - A(p_0)$ and the integrands are free of singularities.
Compared to Lemma \ref{global-decreasing}, $p_0 \in (0,1)$ and $q_0 \in (0,\infty)$ are
independent parameters. All terms in the representation are $C^1$ functions in $q_0$.
Differentiating in $q_0$ yields the expression
\begin{eqnarray*}
\left[ E_0(p_0,q_0) + A(p_*) \right] \frac{\partial}{\partial q_0} T_+(p_0,q_0) =
q_0 \int_{p_0}^{p_+} \left[1 - \frac{2 (A(u) - A(p_*)) A''(u)}{[A'(u)]^2} \right] \frac{du}{v} + \frac{2 (A(p_0) - A(p_*))}{A'(p_0)},
\end{eqnarray*}
or equivalently
\begin{eqnarray}
\label{derivative-E-plus}
\frac{E_0(p_0,q_0) + A(p_*)}{2q_0} \frac{\partial}{\partial q_0} T_+(p_0,q_0)  & = &
\int_{p_0}^{p_+} \frac{1-2u^2}{8 v u^2} du - \frac{1-2p_0^2}{8p_0 q_0}.
\end{eqnarray}
Recall from (\ref{E-positivity}) that $E_0(p_0,q_0) + A(p_*) > 0$ for every $p_0 \in (0,1)$ and $q_0 \in (0,\infty)$.
If $p_0 = p_*$, the first term in (\ref{derivative-E-plus}) is negative and the second term is zero,
hence $\frac{\partial}{\partial q_0} T_+(p_*,q_0) < 0$.

For any $p_0 \in (0, p_*)$, we intoduce the value $\tilde p_0 \in (p_*, 1)$ by setting $\tilde p_0^2 := 1-p_0^2$.
It follows from (\ref{invariant}) that $A(p_0) = A(\tilde{p}_0)$ with $0<p_0<p_*<\tilde p_0 < p_+ < 1$.
Next, we rewrite the equation (\ref{derivative-E-plus}) as
\begin{eqnarray}
\label{derivative-E-plus-2}
\frac{E_0(p_0,q_0) + A(p_*)}{2q_0} \frac{\partial}{\partial q_0} T_+(p_0,q_0)  & = &
\int_{p_0}^{p_*} \frac{1-2u^2}{8 v u^2} du +
\int_{p_*}^{\tilde p_0} \frac{1-2u^2}{8 v u^2} du \\
& \phantom{t} & +\int_{\tilde p_0}^{p_+} \frac{1-2u^2}{8 v u^2} du -
\frac{1-2p_0^2}{8p_0 q_0}.
\nonumber
\end{eqnarray}
The substitution $z = \sqrt{1-u^2}$ in the second integral implies that
$$
\int_{p_*}^{\tilde p_0} \frac{1-2u^2}{8 v u^2} du =
- \int_{p_0}^{p_*} \frac{(1-2z^2)z}{8 v (1-z^2)^{3/2}} dz.
$$
Substituting this equation into (\ref{derivative-E-plus-2}) and calling $z$ as $u$ again, we get
\begin{eqnarray}
\label{derivative-E-plus-3}
\frac{E_0(p_0,q_0) + A(p_*)}{2q_0} \frac{\partial}{\partial q_0} T_+(p_0,q_0)  & = &
\int_{p_0}^{p_*} \frac{1-2u^2}{8 v} \left(\frac{1}{u^2} -
\frac{u}{(1-u^2)^{3/2}} \right) du \\
& \phantom{t} & +\int_{\tilde p_0}^{p_+} \frac{1-2u^2}{8 v u^2} du - \frac{1-2p_0^2}{8p_0 q_0}.
\nonumber
\end{eqnarray}
The second term in the right-hand side of (\ref{derivative-E-plus-3}) is negative since $\tilde p_0 \in (p_*, p_+)$,
whereas the first and last terms satisfy
\begin{eqnarray*}
&& \qquad \int_{p_0}^{p_*} \frac{1-2u^2}{8 v} \left(\frac{1}{u^2} -
\frac{u}{(1-u^2)^{3/2}} \right) du  - \frac{1-2p_0^2}{8p_0 q_0} \\
&& \leq \frac{1-2p_0^2}{8 q_0} \int_{p_0}^{p_*}  \left(\frac{1}{u^2} - \frac{u}{(1-u^2)^{3/2}} \right) du - \frac{1-2p_0^2}{8p_0 q_0} \\
&& = \frac{(1-2p_0^2)}{8 q_0} \left[ \frac{1}{\sqrt{1-p_0^2}} - \frac{2}{p_*} \right],
\end{eqnarray*}
which is negative since $p_* < \tilde p_0 = \sqrt{1 - p_0^2}$.
As a result, the entire right-hand side of (\ref{derivative-E-plus-3}) is negative,
hence $\frac{\partial}{\partial q_0} T_+(p_0,q_0) < 0$ for $p_0 \in (0,p_*)$.
\end{proof}

\begin{lemma}
\label{nonmonotone-on-right}
For every $p_0 \in (p_*, 1)$,
$T_+(p_0,q_0)$ is a non-monotone function of $q_0$ in $(0,\infty)$ such that
$T_+(p_0,q_0) \to 0$ as $q_0 \to 0$ and $q_0 \to \infty$.
\end{lemma}

\begin{proof}
First we claim that $T_+(p_0,q_0) \to 0$ as $q_0 \to 0$.
Indeed, if $q_0 = 0$, the only admissible root for $p_+ \geq p_0$ in
the nonlinear equation (\ref{p+p-}) is $p_+ = p_0$.
Hence, as $q_0 \to 0$, the length of integration in $T_+(p_0,q_0)$ given by (\ref{period}) shrinks to zero
whereas the integrand remains absolutely integrable so that $T_+(p_0,q_0) \to 0$ as $q_0 \to 0$.

Next, we claim that $T_+(p_0,q_0) \to 0$ as $q_0 \to \infty$.
By (\ref{period}) and (\ref{p+p-}), we bound $T_+(p_0, q_0)$ as in
$$
T_+(p_0, q_0) = \int_{p_0}^{p_+} \frac{du}{\sqrt{E + u^2-u^4}}
\leq \int_{0}^{p_+} \frac{du}{\sqrt{u^2-u^4-p_+^2+p_+^4}}.
$$
By change of variables $u = p_+ x$, we rewrite the estimate as
\begin{equation}
\label{asym-int}
T_+(p_0, q_0)
\leq \frac{1}{p_+} \int_{0}^{1}
\frac{dx}{\sqrt{1-x^2}\sqrt{1+x^2 - \frac{1}{p_+^2}}}.
\end{equation}
We define $A(x):=\frac{1}{\sqrt{(1+x)(1+x^2 - 1/p_+^2)}}$, and using the integration by parts, we rewrite the integral in
(\ref{asym-int}) as
$$
\int_{0}^{1}
\frac{A(x) dx}{\sqrt{1-x}} = \left[ -2\sqrt{1-x}A(x) \right]\Big|_{0}^1 +
2 \int_0^1 \sqrt{1-x}A'(x) dx,
$$
which is finite for $p_+ > 1$ since $A(x)$ is continuously differentiable on $[0,1]$ for $p_+ > 1$.
Since for fixed $p_0$, we have $p_+ \to \infty$ as $q_0 \to \infty$, the representation (\ref{asym-int}) implies that
$$
T_+(p_0, q_0)
\leq \frac{1}{p_+} \int_{0}^{1} \frac{A(x) dx}{\sqrt{1-x}} \to 0
$$
as $q_0 \to \infty$.
\end{proof}

The following lemma defines the necessary and sufficient condition for the first eigenvalue $\beta_1$ of the Sturm--Liouville problem $\SP_2$ to cross zero when the parameter $\epsilon$ is increased. This condition
is given by the intersection of two curves $C_1$ and $C_2$ given by 
\begin{equation}
\label{curve-1}
C_1 := \left\{ (p_0,q_0) : \quad \frac{\partial T_+}{\partial q_0}(p_0,q_0) = 0, \quad p_0 \in (p_*,1) \right\}
\end{equation}
and
\begin{equation}
\label{curve-2}
C_2 := \left\{ (p_0,q_0) : \quad q_0 = \frac{1}{2N} \sqrt{A(p_0)}, \quad p_0 \in (0,1) \right\}.
\end{equation}
The uniqueness of $C_2$ is obvious (see red curve on Fig. \ref{fig-plane}). 
In the subsequent lemmas, we will also prove that $C_1$ is also uniquely defined.

\begin{lemma}
\label{beta-0-cond}
Let $s(z; p_0, q_0)$ be the even solution to the differential equation (\ref{eqn-s}).
Then,
$$
s(\pm T_+(p_0, q_0); p_0, q_0) = 0 \;\; \mbox{\rm if and only if} \;\;
\frac{\partial T_+}{\partial q_0}(p_0, q_0) = 0.
$$
Moreover, the first eigenvalue
$\lambda = \beta_1$ of the Sturm--Liouville problem $\SP_2$ is zero if and only if
$\frac{\partial T_+}{\partial q_0}(p_0, q_0) = 0$ at $q_0 = \frac{1}{2N} \sqrt{A(p_0)}$.
\end{lemma}

\begin{proof}
Since $u(z;p_0,q_0)$ satisfying (\ref{bvp-i-2}) and $s(z;p_0,q_0)$ satisfying
(\ref{eqn-s}) are even, it is sufficient to consider the left boundary condition at
$z = -T_+(p_0,q_0)$ rewritten again as
\begin{equation}
\label{bc-p0-q0}
\left\{
\begin{array}{l} u(-T_+(p_0, q_0); p_0, q_0) = p_0, \\
u'(-T_+(p_0, q_0); p_0, q_0) = q_0.
\end{array}\right.
\end{equation}
We differentiate the first equation in (\ref{bc-p0-q0}) with respect to $q_0$ and obtain
\begin{equation}
\label{tech-der-T}
\partial_{q_0} u(-T_+(p_0, q_0); p_0, q_0) - u'(-T_+(p_0, q_0); p_0, q_0)
\frac{\partial}{\partial q_0} T_+(p_0, q_0) = 0.
\end{equation}
By using the definition of $s(z;p_0,q_0)$ and
the second equation in (\ref{bc-p0-q0}), we rewrite (\ref{tech-der-T}) in the form:
\begin{equation}
\label{s-boundary}
s(-T_+(p_0, q_0); p_0, q_0) =
q_0 \frac{\partial}{\partial q_0} T_+(p_0, q_0).
\end{equation}
Since $q_0 \in (0,\infty)$, it follows from (\ref{s-boundary}) that
$s(-T_+(p_0, q_0); p_0, q_0) =0$ if and only if
$\frac{\partial T_+}{\partial q_0}(p_0, q_0) = 0$.

If $q_0 = \frac{1}{2N} \sqrt{A(p_0)}$, then we have $T_+(p_0,q_0) = \mathcal{T}(p_0) = \pi \epsilon$ so that
the differential equation (\ref{eqn-s}) coincides with that in the Sturm--Liouville problem
$\SP_2$ with $\lambda = 0$ in (\ref{sp1}). If $\frac{\partial T_+}{\partial q_0}(p_0, q_0) = 0$ for this $q_0$,
then it follows from (\ref{s-boundary}) that $s(\pm \pi \epsilon; p_0, q_0) = 0$,
hence $s(z;p_0,q_0)$ with this $q_0$ is the eigenfunction of $\SP_2$ with $\beta_1 = 0$.
On the other hand, if $\beta_1 = 0$, then the corresponding eigenfunction is even and hence
it coincides up to a scalar multiplication with $s(z;p_0,q_0)$ for this $q_0$
by uniqueness of solutions of the second-order differential equations. Then,
it follows from (\ref{s-boundary}) that $\frac{\partial}{\partial q_0} T_+(p_0,q_0) = 0$ for this $q_0$.
\end{proof}

The following lemma ensures that there is only one critical (maximum) point of
$T_+(p_0,q_0)$ with respect to $q_0$ at each energy level $E_0(p_0,q_0) = q_0^2 - A(p_0)$.

\begin{lemma}
\label{no-double-roots}
Let $E(p,q) = q^2-A(p)$ be the first-order invariant for the boundary-value problem
(\ref{bvp-i-2}). There are no distinct points
$(p_1, q_1)$ and $(p_2, q_2)$  in $(0,1)\times (0,\infty)$ with
$E(p_1, q_1) = E(p_2, q_2)$ such that
$\frac{\partial T_+}{\partial q_1}(p_1, q_1) = 0$ and
$\frac{\partial T_+}{\partial q_2}(p_2, q_2) = 0$.
\end{lemma}

\begin{proof}
Assume that such points $(p_1, q_1)$ and $(p_2, q_2)$ in
$(0,1) \times (0,\infty)$ do exist, and pick $p_1<p_2$ without loss of generality.
Then, we have
$\frac{\partial T_+}{\partial q_1}(p_1, q_1) = 0$ and
$\frac{\partial T_+}{\partial q_2}(p_2, q_2) = 0$.
For $j \in \{1,2\}$, consider the boundary-value problem (\ref{bvp-i-2})
with the boundary values $(p_j, q_j)$. By Lemma \ref{beta-0-cond},
we know that $s(z;p_j, q_j)$ is a solution to the differential equation (\ref{eqn-s})
such that $s(\pm T_+(p_j, q_j); p_j, q_j) = 0$, hence $s(z;p_j,q_j)$ is the
eigenfunction of the corresponding Sturm--Liouville problem.

Since $E(p_1, q_1) = E(p_2, q_2)$ and $p_1<p_2$ by assumption, we have $u(z; p_1, q_1) = u(z; p_2, q_2)$ for all
$z\in [-T_+(p_2, q_2), T_+(p_2, q_2)]$. Then, the function $s(z; p_1, q_1)$
is proportional to a solution to the initial-value problem (\ref{ivp-interval}) for $\beta = 0$
on $[-T_+(p_1, q_1), T_+(p_1, q_1)]$, where it vanishes at least at two internal points
$\pm T_+(p_2, q_2)$. By Proposition \ref{SLP-eigenfunctions}, $s(z; p_1, q_1)$ is
the eigenfunction of the Sturm--Liouville problem corresponding to (at least)
the third eigenvalue of $\SP_2$, which implies that the second eigenvalue $\beta_2$ is negative.
However, this contradicts to Lemma \ref{second-pos} which ensures that $\beta_2 > 0$.
Hence, no two distinct points exist as in the assertion of the lemma.
\end{proof}

By Lemma \ref{nonmonotone-on-right},
there exists at least one local maximum of $T_+(p_0,q_0)$ in $q_0$ for $p_0 \in (p_*,1)$. Let us denote
the corresponding value of $q_0$ by $\qmax (p_0)$. Since $T_+(p_0,q_0)$ is a $C^1$ function of $(p_0,q_0)$ in $(0,1) \times (0,\infty)$,
$\qmax$ is a continuous function of $p_0$. The following lemma shows that $\qmax(p_0)$ is the unique critical point of $T_+(p_0, q_0)$ in $q_0$
inside $(0,\sqrt{A(p_0)})$. This given uniqueness of the curve $C_1$ defined by (\ref{curve-1}).

\begin{lemma}
\label{behaviour-qmax}
There exists $p_{**} \in (p_*, 1)$ such that for every $p_0 \in (p_*, p_{**})$, there is exactly one critical point of $T_+(p_0, q_0)$ in $q_0$
inside $(0, \sqrt{A(p_0)})$. For $p_0 \in [p_{**}, 1)$, $T_+(p_0, q_0)$ has no critical points in
$q_0$ inside $(0, \sqrt{A(p_0)})$.
\end{lemma}

\begin{proof}
Let $\qmax (p_0)$ be the point of maximum of $T_+(p_0,q_0)$ in $q_0$ for $p_0 \in (p_*,1)$.
We first show that $\qmax(p_0) \to 0$ as $p_0 \to p_*$ and
$\qmax(p_0) > \sqrt{A(p_0)}$ for $p_0$ near $1$.

It follows from (\ref{derivative-E-plus}) that if
$\frac{\partial}{\partial q_0} T_+(p_0, \qmax (p_0)) = 0$,
then on the energy level $E = E_0(p_0,\qmax (p_0))$ we have
\begin{equation}
\label{crit-point-cond}
\qmax(p_0) \int_{p_0}^{p_+} \frac{2u^2-1}{ v u^2}  du =
\frac{2p_0^2-1}{p_0}.
\end{equation}
Integration by parts with the help of
\begin{equation}
\label{integr-by-parts}
d \left( \frac{v}{u^3} \right) = -\frac{2u^2-1}{vu^2}du - \frac{3v}{u^4}du
\end{equation}
yields
\begin{equation}
\label{crit-point-cond-2}
\qmax(p_0)^2 - (2p_0^2-1)p_0^2 =
3p_0^3 \qmax (p_0) \int_{p_0}^{p_+} \frac{v}{u^4}  du > 0.
\end{equation}
This gives the lower bound for $\qmax (p_0)$ as
$$
\qmax (p_0) > p_0 \sqrt{2 p_0^2 -1}.
$$
Recall that $\sqrt{A(p_0)} = p_0 \sqrt{1- p_0^2}$. Hence,
$\qmax(p_0) > \sqrt{A(p_0)}$ if $p_0 > \sqrt{2/3}$. By continuity of
$\qmax$ and Lemma \ref{no-double-roots}, there exists unique $p_{**} \in (p_*, \sqrt{2/3})$ such that $\qmax(p_{**}) = \sqrt{A(p_*)}$.

To prove that $\qmax(p_0) \to 0$ as $p_0 \to p_*$, we assume the contrary. That is, let $\qmax(p_0)>\epsilon$ for some $\epsilon>0$ whenever
$0<p_0-p_*<\delta_0$ with sufficiently small $\delta_0>0$. Then, there is some positive $\delta_1$ such that $p_+ > p_0 + \delta_1$.
Then,
\begin{equation}
\label{inter-qmax}
\int_{p_0}^{p_+} \frac{2u^2-1}{ v u^2}  du >
\int_{p_0+\delta_1}^{p_+} \frac{2u^2-1}{ v u^2}  du >
\frac{2(p_0+\delta_1)^2-1}{p_+^2} \int_{p_0+\delta_1}^{p_+} \frac{du}{v}.
\end{equation}
Since $p_0 \in (p_*,p_*+\delta_0)$ and $q_{\rm max}(p_0)$ is continuous, $p_+$ is bounded from above,
so that there exists some $\delta_2>0$ such that
$$
\frac{2(p_0+\delta_1)^2-1}{p_+^2}> \delta_2.
$$
Since $\qmax (p_0) >\epsilon$ and the integration in (\ref{inter-qmax}) goes along the energy level containing $(p_0, \qmax (p_0))$, there exists some $\delta_3>0$ such that
$$
\int_{p_0+\delta_1}^{p_+} \frac{du}{v} = T_+ (p_0+\delta_1, \qmax (p_0)) > \delta_3.
$$
Combining the computations above, we get that (\ref{crit-point-cond}) becomes
$$
\frac{2p_0^2-1}{p_0} = \qmax(p_0) \int_{p_0}^{p_+} \frac{2u^2-1}{ v u^2}  du > \epsilon \delta_2 \delta_3,
$$
which is the contradiction since $\frac{2p_0^2-1}{p_0} \to 0$ as $p_0 \to p_*$.
Hence $q_{\rm max}(p_0) \to 0$ as $p_0 \to p_*$.

Thus, the graph of the function $p_0 \mapsto \qmax (p_0)$ starts from zero at $p_0 = p_*$ and
traverses beyond the homoclinic orbit for $p_0 > p_{**}$. By continuity of $\qmax$ in $p_0$,
$\qmax$ intersects at least once each energy level (\ref{energy-level})
inside the homoclinic orbit. By Lemma \ref{no-double-roots}, the intersection of $\qmax$
with each energy level is unique. This proves the assertion of this lemma.
\end{proof}

By Lemma \ref{behaviour-qmax}, the curve $C_1$ in (\ref{curve-1}) intersects at least once with every energy level $E(u,v) = E_0(p_0,q_0)$
inside the homoclinic orbit. On the other hand, the curve $C_2$ in (\ref{curve-2}) lies entirely within
the homoclinic orbit, hence there exists an intersection between the curves $C_1$ and $C_2$. The following lemma shows that
this intersection is in fact unique.

\begin{lemma}
\label{no-double-intersections}
There exists exactly one value of $p_0 \in (p_*,p_{**})$ for which $\qmax(p_0) = \frac{1}{2N} \sqrt{A(p_0)}$.
\end{lemma}

\begin{proof}
Consider the function $\mathcal{F} : (p_*,p_{**}) \to \mathbb{R}$ given by
\begin{equation}
\label{local-function}
\mathcal{F}(p_0) = p_0^2 (2p_0^2-1) - q_0 p_0^3 \int_{p_0}^{p_+} \frac{2u^2-1}{ v u^2}  du,
\end{equation}
where $q_0=\frac{1}{2N} \sqrt{A(p_0)}$ and the integration is performed along the level curve with $E(u,v) = E_0(p_0,q_0)$. By (\ref{crit-point-cond}), $\mathcal{F}(p_0) = 0$ if and only if
$\qmax (p_0) = \frac{1}{2N}\sqrt{A(p_0)}$. Since by Lemma \ref{behaviour-qmax}, $\mathcal{F}(p_0) = 0$
has at least one root in $(p_*,p_{**})$, it suffices to show that there are no other roots.

By using (\ref{integr-by-parts}), we obtain
\begin{equation}
\label{local-function-2}
\mathcal{F}(p_0) = (2p_0^2-1)p_0^2 - q_0^2 + 3 q_0 p_0^3
\int_{p_0}^{p_+} \frac{v}{u^4} du.
\end{equation}
We claim that $\mathcal{F}'(p_{\rm bif})$ at the root $p_{\rm bif}$ of $\mathcal{F}(p_0) = 0$, so that
the root $p_{\rm bif}$ is unique. Indeed, taking the derivative in (\ref{local-function-2}) with respect to $p_0$, and using that
$q_0 = \frac{1}{2N}\sqrt{A(p_0)}$ and $\mathcal{F}(p_0)=0$ we obtain
$$
\mathcal{F}'(p_0) = p_0(2p_0^2+1) - \frac{\partial_{p_0} q_0}{q_0}
\left[q_0^2 + (2p_0^2-1) p_0^2 \right] +
p_0(2p_0^2-1) \left[1 - \frac{1}{4N^2} \right] \int_{p_0}^{p_+} \frac{du}{vu^4},
$$
which is strictly positive since $\partial_{p_0} q_0 = \frac{A'(p_0)}{4 N \sqrt{A(p_0)}} < 0$ for
$p_0 \in (p_*, 1)$. This completes the proof.
\end{proof}

Figure \ref{fig:tmax} illustrates the results of Lemmas \ref{behaviour-qmax} and \ref{no-double-intersections}.
The black dashed curve displays the homoclinic orbit at the energy level $E(u,v) = 0$.
The red dashed curve gives the curve $C_2$ for $N = 3$. The blue solid curve shows the curve $C_1$. There exists
only one intersection of curves $C_1$ and $C_2$ and it occurs at $p_{\rm bif} \approx 0.711$ (for $N = 3$)
The existence of the unique value of $p_{\rm bif}$ is stated in Lemma \ref{no-double-intersections}.
Moreover, $C_1$ crosses the homoclinic orbit at $p_{**} \approx 0.782$
in agreement with Lemma \ref{behaviour-qmax}.

\begin{figure}[htbp]
\includegraphics[scale=0.4]{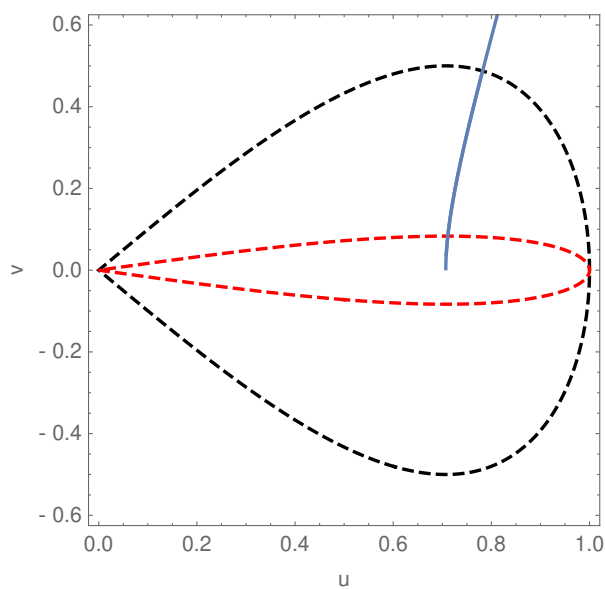}
\caption{Numerical illustration of Lemmas \ref{behaviour-qmax} and \ref{no-double-intersections} on the $(u,v)$-plane. }
\label{fig:tmax}
\end{figure}

\subsection{Proof of Theorem \ref{global-stability}.}

By Lemma \ref{global-decreasing} for $\mathcal{T}(p_0) = \pi \epsilon$ defined in (\ref{root-i}),
the mapping from $p_0 \in (0,1)$ to $\epsilon = \frac{1}{\pi} \mathcal{T}(p_0) \in (0, \infty)$ is a monotonic bijection.

For sufficiently small values of $\epsilon>0$, the value of $p_0$ is near $1$. Then, by Lemmas
\ref{nonmonotone-on-right} and \ref{behaviour-qmax},
$T_+(p_0, q_0)$ has no critical points with respect to $q_0$ in $(0, \sqrt{A(p_0)})$ and is monotonically increasing in $q_0$.
In this case, the solution $s(z;p_0,q_0)$ to the differential equation (\ref{eqn-s}) with $q_0 = \frac{1}{2N} \sqrt{A(p_0)}$
satisfies $s(z;p_0,q_0) > 0$ for $z \in [-\pi \epsilon,\pi \epsilon]$.
By Proposition \ref{positivity-sp1}, we conclude that the first eigenvalue $\lambda = \beta_1$ in $\SP_2$ is positive.
Therefore, Lemmas \ref{first-eig-L} and \ref{second-eig-L} imply that the spectral problem (\ref{L-spectral-sym})
has exactly one negative eigenvalue and no zero eigenvalues, so that $n(\mathcal{L}) = 1$ and $z(\mathcal{L}) = 0$
for sufficiently small $\epsilon>0$.

Let $\beta_1$ be the first eigenvalue in $\SP_2$ and $\gamma_2$ be the second eigenvalue in $\SP_1$.
Since $\beta_1 > 0$ and $\gamma_2 > 0$ for sufficiently small $\epsilon>0$, it suffices to show that $\beta_1=0$ at some unique point
$\epsilon_* \in (0, \infty)$ so that $\beta_1<0$ for all $\epsilon> \epsilon_*$, whereas $\gamma_2 > 0$ for all $\epsilon > 0$.
By Lemma \ref{gamma-2-nonzero} it follows that $\gamma_2 \neq 0$ for every $\epsilon > 0$,
hence $\gamma_2 > 0$ for all $\epsilon > 0$.

Next, we show that $\beta_1 = 0$ for some $\epsilon_* \in (0, \infty)$.
Indeed, by Lemmas \ref{behaviour-qmax} and \ref{no-double-intersections},
the curves $C_1$ and $C_2$ defined by (\ref{curve-1}) and (\ref{curve-2}) intersect 
exactly once at some $p_{\rm bif} \in (p_*, 1)$. By Lemma \ref{beta-0-cond}, $\beta_1 = 0$ at this $p_{\rm bif}$
and by Lemma \ref{global-decreasing}, there exists a unique value $\epsilon_*$ for this $p_{\rm bif}$.
By Lemma \ref{lem-vertex-cond}, $\beta_1$ has multiplicity $N-1$ in the spectral problem
(\ref{L-spectral-sym}) so that $z(\mathcal{L}) = N-1$ for this $\epsilon_*$. No other intersections
exist so that $z(\mathcal{L}) = 0$ for $\epsilon \neq \epsilon_*$.

Finally, for $\epsilon > \epsilon_*$, $\qmax(p_0) < \frac{1}{2N} \sqrt{A(p_0)}$ for $p_0 \in (p_*,p_{\rm bif})$
or does not exist if $p_0 \in (0,p_*]$ by Lemma \ref{monotone-on-left}. In both cases, the solution $s(z;p_0,q_0)$
to the differential equation (\ref{eqn-s}) with $q_0 = \frac{1}{2N} \sqrt{A(p_0)}$
vanishes at some internal points in $[-\pi \epsilon,\pi \epsilon]$. By Proposition \ref{positivity-sp1},
it follows that $\beta_1 <0$ for $\epsilon > \epsilon_*$, so that $n(\mathcal{L}) = N$ for $\epsilon > \epsilon_*$.

Theorem \ref{global-stability} is proven. Figure \ref{fig-second-eigenvalue} illustrates
the result of Theorem \ref{global-stability}. The second eigenvalue $\lambda_2$ of
the spectral problem (\ref{L-spectral-sym}) is computed by using numerical approximation of
the first eigenvalue $\lambda = \beta_1$ in the Sturm--Liouville problem $\SP_2$
and is shown versus $\omega$. It follows from Fig. \ref{fig-second-eigenvalue}
that there exists a value $\omega_* \in (-\infty,0)$ for which $\lambda_2 = \beta_1$ crosses zero.
This is the bifurcation point for the positive single-lobe symmetric state $\Phi$ in Theorem \ref{global-stability}.

\begin{figure}[htbp] %  figure placement: here, top, bottom, or page
   \centering
  \includegraphics[width=2.8in, height = 2in]{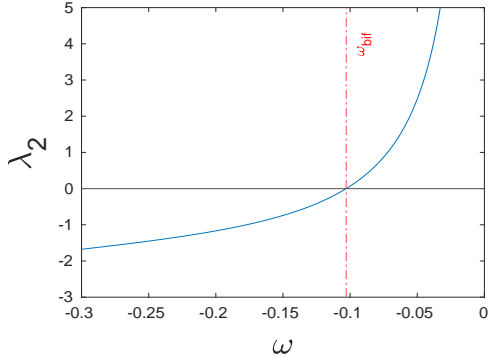}
   \caption{The second eigenvalue $\lambda_2 = \beta_1$ of the spectral problems (\ref{L-spectral-sym})
   and (\ref{sp1})
   as a function of $\omega$ for the positive single-lobe symmetric state $\Phi$ on the flower graph
   $\Gamma_N$ with $N = 3$. The eigenvalue crosses zero at $\omega = \omega_*$.}
   \label{fig-second-eigenvalue}
\end{figure}

\section{Existence of other positive single-lobe states}
\label{sec-global-bifurcations}

Recall that by Theorem \ref{global-stability}, there exists a unique
$\omega_* \in (-\infty, 0)$, and unique corresponding $p_{\rm bif} \in (p_*,1)$,
at which the single-lobe symmetric state $\Phi$ defined in Theorem \ref{global-existence} admits a bifurcation
in the sense of Definition \ref{def-bifurcation}.

Here we are interested in the existence of asymmetric, $K$-split, single-lobe states of Definition \ref{def-asymmetric} for $p_0 \in (0,p_*)$. This range of values of $p_0$ does not cover the entire admissible interval since $p_{\rm bif} \in (p_*,1)$ but it is sufficient for the proof of Theorem \ref{global-bifurcations}.

After the scaling transformation (\ref{scaling-transform}),
the asymmetric positive state $(u_1, u_2, \dots, u_N, u_0)$ satisfies the system of
differential equations given by (\ref{NLS-scaled})--(\ref{bvp-scaled}).
Taking into account the solution (\ref{soliton}) for $u_0$ with $p_0 = {\rm sech}(a) = u_0(0)$,
each component $u_j$ for $j=1, \dots, N$ satisfies the following boundary-value problem
\begin{equation}
\label{bvp-components}
\left\{ \begin{array}{l} -u_j''(z) + u_j(z) - 2 u_j(z)^3 = 0, \quad
z \in (-\pi \epsilon, \pi \epsilon), \\
u_j(-\pi \epsilon) = u_j(\pi \epsilon) = p_0.
\end{array} \right.
\end{equation}
Assuming that $u_j$ is even, the derivative condition
in (\ref{bvp-scaled}) is satisfied if the derivative of the components satisfy the
scalar equation
\begin{equation}
\label{NKcondition}
2 \sum_{j=1}^N u_j'(-\pi \epsilon) = \sqrt{A(p_0)}.
\end{equation}

Using the first-order invariant in (\ref{invariant}), any single-lobe solution to
the boundary-value problem (\ref{bvp-components}) satisfies either
\begin{equation}
\label{asym-sol-plus}
\left\{ \begin{array}{l}
E(u_j, u_j') = E(p_0, q_j), \\
u_j(-T_+(p_0, q_j)) = p_0, \\
u_j'(-T_+(p_0, q_j)) = q_j \geq 0, \\
T_+(p_0, q_j) = \pi \epsilon,
\end{array} \right.
\end{equation}
or
\begin{equation}
\label{asym-sol-minus}
\left\{ \begin{array}{l}
E(u_j, u_j') = E(p_0, q_j), \\
u_j(-T_-(p_0, q_j)) = p_0, \\
u_j'(-T_-(p_0, q_j)) = -q_j \leq 0, \\
T_-(p_0, q_j) = \pi \epsilon,
\end{array} \right.
\end{equation}
where the period functions $T_+$ and $T_-$ are given in (\ref{period}) with fixed value of $E(u_j,u_j') = E$.
Therefore, any asymmetric single-lobe state is a combination of
the solutions of type (\ref{asym-sol-plus}) or (\ref{asym-sol-minus}).

In order to prove Theorem \ref{global-bifurcations}, we first study monotonicity of the period function $T_-(p_0,q_0)$ in $q_0$
for $p_0 \in (0,p_*)$. Then, we prove existence and uniqueness of the asymmetric positive single-lobe states with $K$-split profile
described by Definition \ref{def-asymmetric}.
Finally, we study the mapping from $p_0 \in (0,p_*)$ to $\epsilon \in (0,\infty)$, which extends to the limit $\epsilon \to \infty$
that corresponds to the limit $\omega \to -\infty$.

\subsection{Monotonicity of the period function $T_-$.}

The following lemma shows that the period function $T_-(p_0,q_0)$ defined by (\ref{period})
is monotonically increasing for $p_0 \in (0, p_*)$.

\begin{lemma}
\label{monotone-T-minus}
For every $p_0 \in (0,p_*)$, $T_-(p_0,q_0)$ is a monotonically
increasing function of $q_0$ in $(0,\sqrt{A(p_0)})$.
Moreover, $T_-(p_0, q_0)\to 0$ as $q_0 \to 0$, and
$T_-(p_0, q_0) \to \infty$ as $q_0 \to \sqrt{A(p_0)}$.
\end{lemma}

\begin{proof}
We write
\begin{eqnarray*}
\left[ E_0(p_0,q_0) + A(p_*) \right] T_-(p_0,q_0) & = & \int_{p_-}^{p_0} \left[ v - \frac{A(u)-A(p_*)}{v} \right] du \\
& = & \int_{p_-}^{p_0}  \left[3 - \frac{2 (A(u) - A(p_*)) A''(u)}{[A'(u)]^2} \right] v du - \frac{2 [A(p_0) - A(p_*)] q_0}{A'(p_0)},
\end{eqnarray*}
where $E_0(p_0,q_0) = q_0^2 - A(p_0)$ and the integrands are non-singular for every $u \in (0,1)$.
Since $dE_0 = 2 q_0 dq_0$ at fixed $p_0 \in (0,1)$
and $dE = 2 v dv$ at fixed $u \in (0,1)$, we differentiate the previous
expression in $q_0$ and obtain
\begin{eqnarray*}
\frac{E_0(p_0,q_0) + A(p_*)}{2q_0} \frac{\partial}{\partial q_0} T_-(p_0,q_0) =
\int_{p_-}^{p_0} \left[1 - \frac{2 (A(u) - A(p_*)) A''(u)}{[A'(u)]^2} \right] \frac{du}{2v} - \frac{A(p_0) - A(p_*)}{q_0 A'(p_0)}.
\end{eqnarray*}
Recall that $E_0(p_0,q_0) + A(p_*) > 0$ for every $p_0 \in (0,1)$ and $q_0 \in (0,\infty)$ due to (\ref{E-positivity}).
Substituting $A(u)$ transforms the previous expression to the form:
\begin{eqnarray}
\label{derivative-E}
\frac{E_0(p_0,q_0) + A(p_*)}{2q_0} \frac{\partial}{\partial q_0} T_-(p_0,q_0)  & = &  \int_{p_-}^{p_0} \frac{1-2u^2}{8 v u^2} du + \frac{1-2p_0^2}{8p_0 q_0}.
\end{eqnarray}
Since both terms in the right-hand side of (\ref{derivative-E}) are strictly positive if $p_0 \in (0,p_*)$ with $q_0 \in (0,\infty)$,
we conclude that $\frac{\partial}{\partial q_0} T_-(p_0,q_0) > 0$ if $p_0 \in (0,p_*)$.

It follows that $T_-(p_0, q_0) \to 0$ as $q_0 \to 0$ similarly as in Lemma \ref{nonmonotone-on-right}.
On the other hand, $p_- \to 0$ as $q_0 \to \sqrt{A(p_0)}$, hence $T_-(p_0, q_0) \to \infty$ as $q_0 \to \sqrt{A(p_0)}$.
\end{proof}

The following lemma follows from monotonicity of the period functions $T_+$ and $T_-$ in $q_0$ for every $p_0 \in (0,p_*)$,
thanks to Lemmas \ref{monotone-on-left} and \ref{monotone-T-minus}.

\begin{lemma}
\label{no-same-state}
For every $p_0 \in (0, p_*)$, there are no distinct solutions $u_j(z)$ and $u_i(z)$
to the boundary-value problem (\ref{bvp-components}) such that $u_j(z)$ and $u_i(z)$
are either both of type (\ref{asym-sol-plus}) or both of type (\ref{asym-sol-minus}).
\end{lemma}

\begin{proof}
If $u_j(z)$ and $u_i(z)$ are distinct and both have the type
(\ref{asym-sol-plus}), then $q_j \neq q_i$. By Lemma
\ref{monotone-on-left}, we have $T_+(p_0, q_j) \neq T_+(p_0, q_i)$ which contradicts
to the condition $T_+(p_0, q_j) =\pi \epsilon = T_+(p_0, q_i)$ in (\ref{asym-sol-plus}).

Similarly, if $u_j(z)$ and $u_i(z)$ are distinct and both have the type
(\ref{asym-sol-minus}), then $q_j \neq q_i$. By Lemma \ref{monotone-T-minus},
we have $T_-(p_0, q_j) \neq T_-(p_0, q_i)$ which contradicts
to the condition $T_-(p_0, q_j) =\pi \epsilon = T_-(p_0, q_i)$ in (\ref{asym-sol-minus}).
\end{proof}

\subsection{Construction of asymmetric single-lobe states}

By Lemma \ref{no-same-state}, every asymmetric single-lobe state must have the particular
structure of Definition \ref{def-asymmetric} if $p_0 \in (0,p_*)$ with $K$
components being of type (\ref{asym-sol-plus}) and $(N-K)$ components being of type (\ref{asym-sol-minus}).
Up to permutation between the components in the $N$ loops, we order the $K$-split state as follows:
\begin{equation}
\label{q-cond}
q_1 = q_2 = \dots = q_K \geq 0 \quad \textrm{and} \quad
q_{K+1} = q_{K+2} = \dots = q_N \geq 0.
\end{equation}
The existence of asymmetric, $K$-split, single-lobe states for a given $p_0 \in (0,p_*)$
is equivalent to the existence of $(q_1, q_2, \dots, q_N)$ satisfying (\ref{q-cond}) and
solving the system of two nonlinear equations on $q_1$ and $q_N$:
\begin{equation}
\label{split-system}
\left\{ \begin{array}{l}
T_+(p_0, q_1) = T_-(p_0, q_N), \\
2 K q_1 - 2 (N-K)q_N = \sqrt{A(p_0)},
\end{array} \right.
\end{equation}
where the second equation comes from the boundary condition (\ref{NKcondition}).
The following lemma provides the unique solution to the system (\ref{split-system}) for each $K$.

\begin{lemma}
\label{up-down-state}
Let $p_0 \in (0, p_*)$. For every $K=1,2,\dots, N-1$, there exists the unique solution to the system (\ref{split-system})
and the unique asymmetric, $K$-split, single-lobe state in the sense of Definition \ref{def-asymmetric}.
\end{lemma}

\begin{proof}
By Lemma \ref{no-same-state}, for every asymmetric single-lobe state, there are no distinct components
$u_j(z)$ and $u_i(z)$ of the same type. If $u_j(z)$ and $u_i(z)$ are distinct, then one of them
is uniquely given by (\ref{asym-sol-plus}), while the other one is uniquely given by (\ref{asym-sol-minus}).
Hence, the assertion of the lemma holds if we can prove the existence of the unique solution to the system
(\ref{split-system}).

Consider the function $F(q_1)$ defined by
\begin{equation}
\label{Fq1}
F(q_1) := T_+(p_0, q_1) - T_-(p_0, q_N(q_1)),
\end{equation}
where $q_N(q_1)$ is obtained from the second equation of system (\ref{split-system}) in the form:
\begin{equation}
\label{q-N-unique}
q_N(q_1) = \frac{K}{N-K}q_1 - \frac{1}{2(N-K)}\sqrt{A(p_0)}.
\end{equation}
Since $q_N \geq 0$, we have $q_1 \geq \frac{1}{2K} \sqrt{A(p_0)}$.
In addition, it follows from positivity of the single-lobe solution that $q_N \leq \sqrt{A(p_0)}$,
so that $q_1 \leq \frac{2(N-K)+1}{2K}\sqrt{A(p_0)}$. Hence, we are only interested
in the behavior of $F$ on the interval
$$
\mathcal{I}(p_0;K) := \left[ \frac{1}{2K}\sqrt{A(p_0)}, \frac{2(N-K)+1}{2K}\sqrt{A(p_0)} \right].
$$
Since $q_N$ is monotonically increasing function of $q_1$, Lemmas \ref{monotone-on-left} and \ref{monotone-T-minus}
imply that the function $F$ is monotonically decreasing in $q_1$. We show that $F(q_1) = 0$ has an unique root in $\mathcal{I}(p_0;K)$.
As $q_1 \to \frac{1}{2K}\sqrt{A(p_0)}$, we have $q_N(q_1) \to 0$, and by Lemma \ref{monotone-T-minus},
$F(q_1) \to T_+(p_0, \frac{1}{2K}\sqrt{A(p_0)}) >0$.
On the other hand, as $q_1 \to \frac{2(N-K)+1}{2K}\sqrt{A(p_0)}$, we have $q_N(q_1) \to \sqrt{A(p_0)}$, and by Lemma \ref{monotone-T-minus},
$F(q_1) \to -\infty$. Therefore, by monotonicity of $F$, there exists the unique root of $F$ in $\mathcal{I}(p_0;K)$.
\end{proof}

\begin{figure}[htbp] %  figure placement: here, top, bottom, or page
	\centering
	\includegraphics[width=3in, height = 2in]{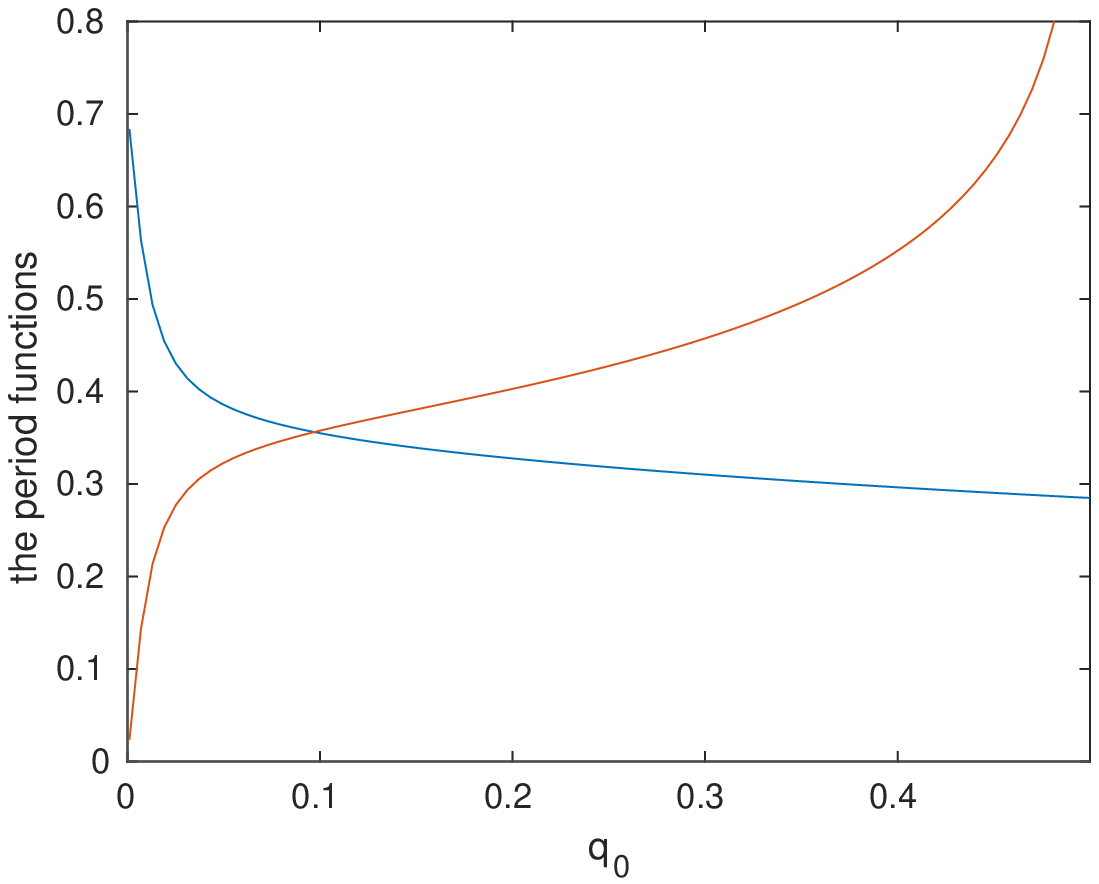}
	\includegraphics[width=3in, height = 2in]{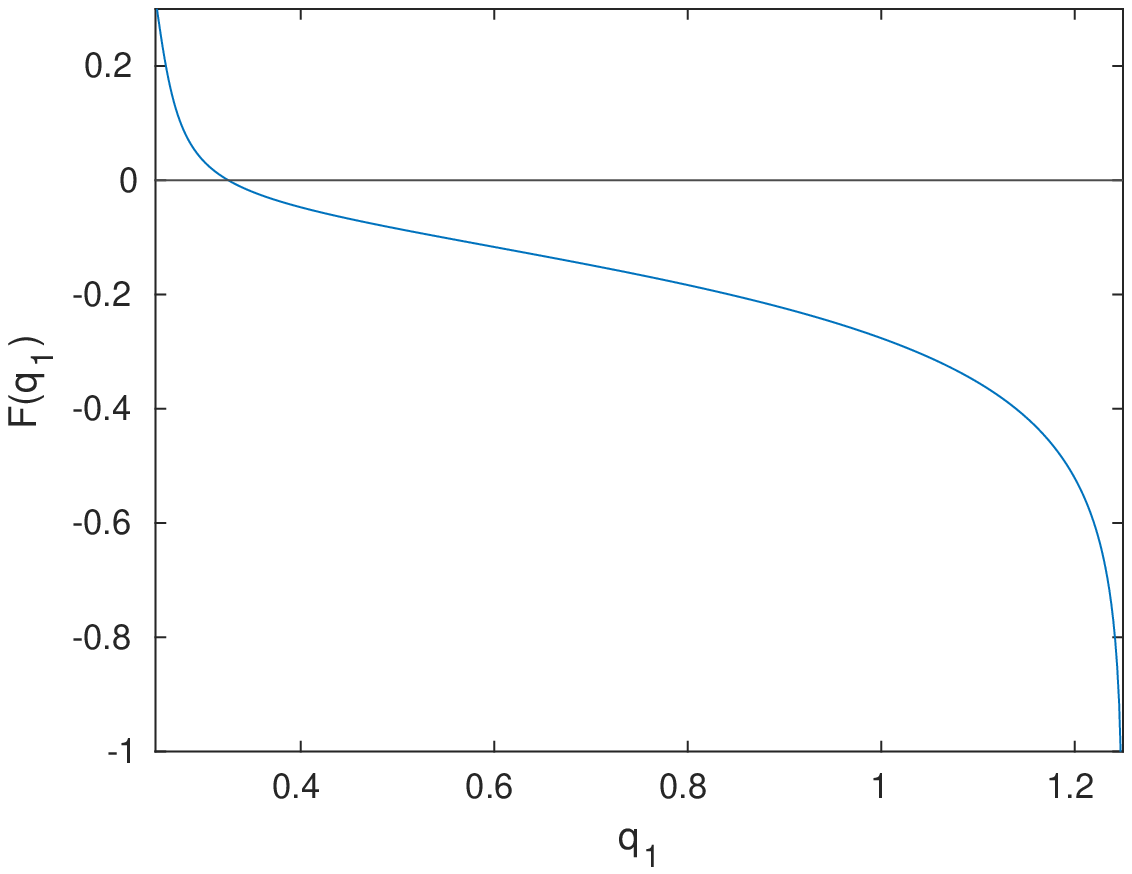}
	\caption{Numerical illustration to the statement of Lemma \ref{up-down-state} for
		$p_0 = 0.7003 \in (0, p_*)$, $N = 3$, and $K = 1$. Left: the blue and red lines show respectively the dependence of
		$T_+(p_0, q_0)$ and $T_-(p_0, q_0)$ in $q_0$. Right: The graph of the function $F$ defined in (\ref{Fq1}) with the only root.}
	\label{periods-section4}
\end{figure}
\begin{figure}[htbp]
	\centering
	\includegraphics[width=3in, height = 2in]{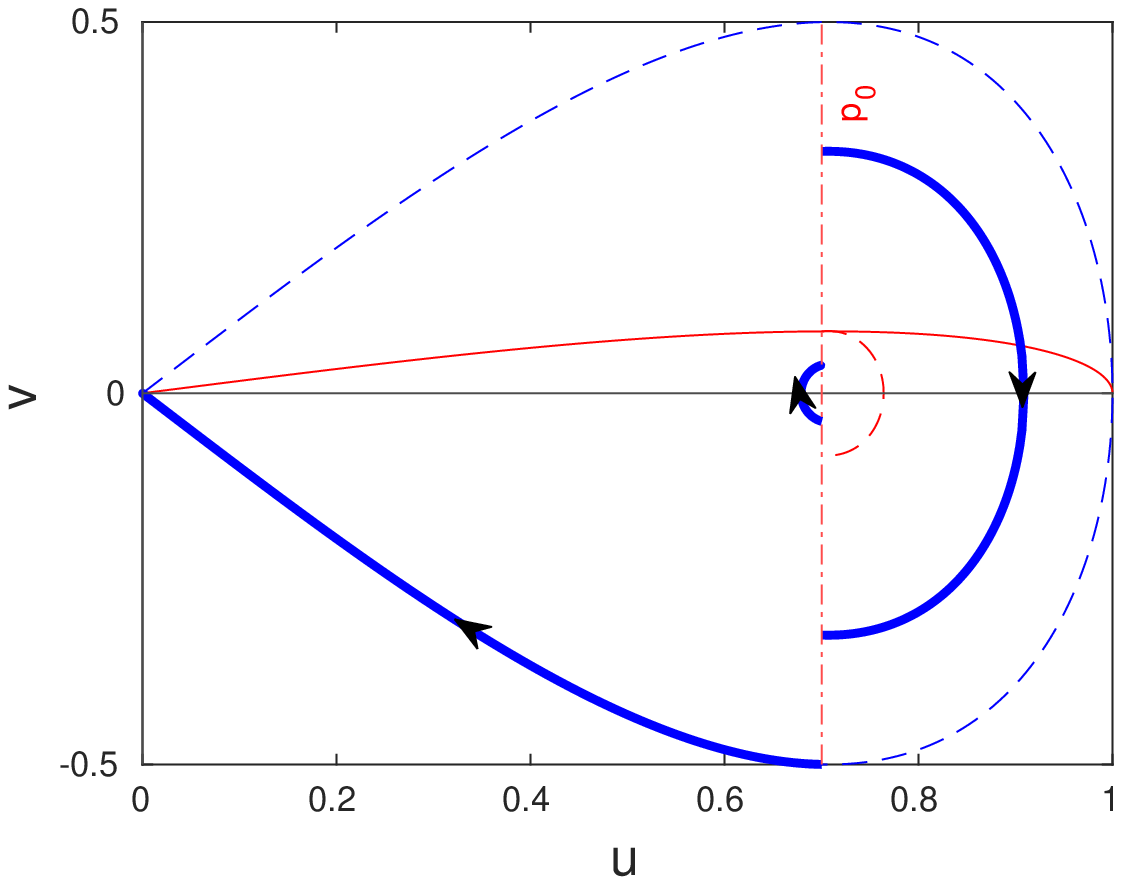}
	\includegraphics[width=3in, height = 2in]{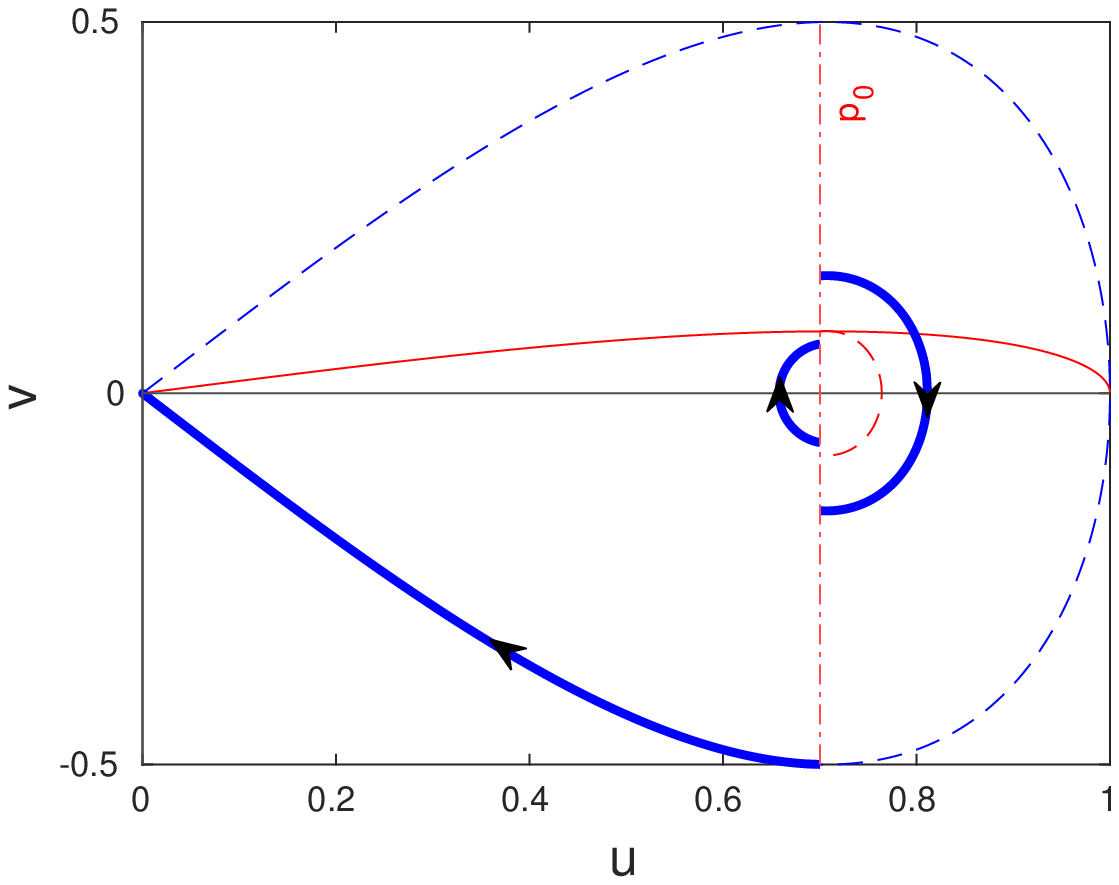}
	\caption{Construction of the positive, asymmetric, $K$-split, single-lobe states on the $(u,v)$-plane
		for $a=0.895$ and $N = 3$ in the case of $K = 1$ (left) and $K = 2$ (right).}
	\label{fig-3}
\end{figure}

The conclusion of Lemma \ref{up-down-state} is illustrated on Fig. \ref{periods-section4}.
The left panel shows plots of $T_+(p_0,q_0)$ and $T_-(p_0,q_0)$ in $q_0$
for a fixed value of $p_0 \in (p_*,1)$. The dependencies are monotonic in agreement with
Lemmas \ref{monotone-on-left} and \ref{monotone-T-minus}. The right panel shows the function
$F$ in $q_1$ defined by (\ref{Fq1})  for $K = 1$ and $N = 3$. The function is monotonic and has a unique root
in the interval $\mathcal{I}(p_0;K)$. A similar picture holds for $K = 2$ and $N = 3$.

Figure \ref{fig-3} show how the asymmetric, $K$-split, single-lobe states are constructed for the same
value of $p_0$ and $N = 3$. The left panel shows the state with $K = 1$ and the right panel shows the state with $K = 2$
by using orbits on the $(u,v)$-plane.

\subsection{The mapping $(0,p_*) \ni p_0 \mapsto \epsilon \in (0,\infty)$}

Fix $K=1, 2, \dots, N-1$. By Lemma \ref{up-down-state}, for every $p_0 \in (0, p_*)$, there is a unique
vector $(q_1, q_2, \dots, q_N)$ satisfying (\ref{q-cond}) and (\ref{split-system}), and this defines uniquely
the following mappings:
\begin{equation}
\label{unique-q1-qN}
(0,p_*) \ni p_0 \mapsto q_1(p_0; K) \in (0, \infty) \quad
\textrm{and} \quad
(0,p_*) \ni p_0 \mapsto q_N(p_0; K) \in (0, \sqrt{A(p_0)}),
\end{equation}
where $q_1(p_0; K) \in \mathcal{I}(p_0;K)$ is uniquely defined as the root of $F$ given by (\ref{Fq1})
and $q_N(p_0;K) \in (0,\sqrt{A(p_0)})$ is uniquely defined by (\ref{q-N-unique}).
By using the first equation in (\ref{split-system}) we also define a unique mapping
\begin{equation}
\label{unique-T-plus}
(0,p_*) \ni p_0 \mapsto T_+(p_0, q_1(p_0; K)) \in (0, \infty).
\end{equation}
The following lemmas describe the dependence of
$T_+(p_0, q_1(p_0; K))$ on $p_0$ which gives
monotonicity of the mapping $(0,p_*) \ni p_0 \mapsto \epsilon = \frac{1}{\pi}  T_+(p_0, q_1(p_0; K)) \in (0,\infty)$.

\begin{lemma}
\label{unique-f-c1}
For every $K=1, 2, \dots, N-1$, the mappings (\ref{unique-q1-qN}) and (\ref{unique-T-plus}) are $C^1$
for every $p_0 \in (0, p_*)$.
\end{lemma}

\begin{proof}
Recall that the period functions
$T_+(p_0, q_0)$ and $T_-(p_0, q_0)$ are $C^1$ in both $p_0$ and $q_0$
 thanks to the representation (\ref{integration}), see the proofs
 of Lemmas \ref{global-decreasing}, \ref{monotone-on-left}, and \ref{monotone-T-minus}.

Consider the function
$G(p_0, q_1, q_N): (0, p_*) \times (0, \infty) \times (0, \sqrt{A(p_0)}) \to \mathbb{R}^2$ given by
\begin{equation}
G(p_0, q_1, q_N) = \begin{pmatrix}
T_+(p_0, q_1) - T_-(p_0, q_N) \\
2 Kq_1 - 2 (N-K)q_N - \sqrt{A(p_0)}
\end{pmatrix}.
\end{equation}
Note that the system (\ref{split-system}) is equivalent to
$G(p_0, q_1, q_N) = 0$. The $C^1$ dependence of $q_1(p_0; K)$ and $q_N(p_0; K)$
with respect to $p_0$ is a direct consequence of the Implicit Function Theorem applied
to the function $G$. Indeed, $G$ is a $C^1$ function in all its variables, and
the Jacobian matrix $D_{(q_1, q_N)} G(p_0,q_1,q_N)$ is invertible since
the determinant of
$$
D_{(q_1, q_N)} G(p_0,q_1,q_N) = \begin{pmatrix}
\frac{\partial}{\partial q_1} T_+(p_0, q_1) & -\frac{\partial}{\partial q_N} T_-(p_0, q_N) \\
2K & -2(N-K)
\end{pmatrix}
$$
is strictly positive due to monotonicity results in Lemmas \ref{monotone-on-left} and \ref{monotone-T-minus}.

The differentiability of the function
$T_+(p_0, q_1(p_0; K))$ in $p_0$ comes from differentiability of $T_+(p_0, q_0)$ and $q_1(p_0; K)$
in its variables.
\end{proof}

\begin{lemma}
\label{unique-T-div}
There exists $p_{\infty} \in (0, p_*)$ such that the mapping
$p_0 \mapsto T_+(p_0, q_1(p_0; K))$ defined in
(\ref{unique-T-plus}) is monotonically decreasing for every
$p_0 \in (0, p_{\infty})$ and every $K=1, 2, \dots, N-1$.
\end{lemma}

\begin{proof}
We shall prove that for every $K=1, 2, \dots, N-1$, it follows that $T_+(p_0, q_1(p_0; K)) \to \infty$ as $p_0 \to 0$.
Since this function is $C^1$ for every $p \in (0,p_*)$ by Lemma \ref{unique-f-c1},
the mapping $p_0 \mapsto T_+(p_0, q_1(p_0; K))$ is monotonically decreasing for small positive $p_0$
and the assertion of the lemma follows.

Set $C_{N,K} := \frac{2(N-K)+1}{2K}$ for simplicity.
Since $q_1(p_0;K) \in \mathcal{I}(p_0;K)$, it is true that
$q_1(p_0; K) \leq C_{N,K} \sqrt{A(p_0)}$. Using the monotonicity of the period function
in Lemma \ref{monotone-on-left}, we obtain
$$
T_+(p_0, q_1(p_0; K)) \geq T_+(p_0, C_{N,K}\sqrt{A(p_0)}),
$$
where the lower bound diverges by Remark \ref{T-plus-C-divergence}:
$$
T_+(p_0, C_{N,K} \sqrt{A(p_0)}) \to \infty
\quad \mbox{\rm as} \quad p_0 \to 0.
$$
Hence, $T_+(p_0, q_1(p_0; K)) \to \infty$ as $p_0 \to 0$.
\end{proof}

\subsection{Proof of Theorem \ref{global-bifurcations}.}

By Lemma \ref{up-down-state}, for every $p_0 \in (0, p_*)$,
there are exactly $N$ positive single-lobe states
$\Phi^{(K)}$ with $1\leq K \leq N$ satisfying the system of differential equations
(\ref{NLS-scaled})--(\ref{bvp-scaled}) with $u_0(0) = p_0$ completed with
the symmetry and monotonicity conditions of Theorem \ref{global-bifurcations}.

For every $K=1, 2, \dots, N-1$, by using the fact that
$T_+(p_0, q_1(p_0;K)) = \pi \epsilon$, we obtain the mapping
$(0, p_*) \ni p_0 \mapsto \epsilon(p_0;K) = \frac{1}{\pi} T_+(p_0, q_1(p_0;K)) \in (0, \infty)$. By smoothness result in Lemma \ref{unique-f-c1}
monotonicity result in Lemma \ref{unique-T-div}, we get the bijection
$$
(0, p_{\infty}) \ni p_0 \mapsto \epsilon(p_0;K) \in (\epsilon_{\infty}(K), \infty),
$$
where $p_{\infty} \in (0,p_*)$ is defined in Lemma \ref{unique-T-div} independently of $K$.
Defining $\epsilon_{\infty} := \max_{1\leq K \leq N-1} \epsilon_{\infty}(K)$,
we get all asymmetric, positive, single-lobe, $K$-split states exist for
$\omega \in (-\infty,\omega_{\infty})$, where $\omega_{\infty} = -\epsilon_{\infty}^2$.
For $K = N$, the existence of symmetric, positive, single-lobe state $\Phi \equiv \Phi^{(N)}$
follows by Theorem \ref{global-existence}.

Moreover, for every $K=1, 2, \dots, N$, the mapping
$(-\infty, \omega_{\infty}) \ni \omega \mapsto \Phi^{(K)}(\cdot,\omega) \in H^2_{\rm NK}(\Gamma_N)$
is $C^1$ by Lemma \ref{unique-f-c1}. By construction,
the mass $\mu^{(K)}(\omega) : = Q(\Phi^{(K)}(\cdot,\omega))$ is equal to
\begin{eqnarray*}
\mu^{(K)}(\omega) = K \int_{-\pi}^{\pi} \phi_1^2 dx +
(N-K) \int_{-\pi}^{\pi} \phi_N^2 dx +
 \int_0^{\infty} \phi_0^2 dx,
\end{eqnarray*}
which yields
\begin{eqnarray*}
\mu^{(K)}(\omega) & = & 2K \epsilon(p_0; K) \int_{p_0}^{p_+} \frac{u^2 du}{\sqrt{A(u)-A(p_+)}}
+ 2(N-K) \epsilon(p_0; K) \int_{p_-}^{p_0} \frac{u^2 du}{\sqrt{A(u)-A(p_-)}} \\
& \phantom{t} & + \epsilon(p_0; K) \left[ 1 - \tanh(a) \right],
\end{eqnarray*}
where the first integral is defined along the level curve with
$E(u,v) = E(p_0, q_1(p_0))$ and the second integral
is defined along the level curve with $E(u,v) = E(p_0, q_N(p_0))$.

As $p_0 \to 0$, we have $a \to \infty$ and $\epsilon(p_0; K) \to \infty$, and so $\mu^{(K)}(\omega) \to \infty$ as $\omega \to \infty$
with the following precise limit:
$$
\lim_{\epsilon \to \infty} \frac{\mu}{\epsilon} = 2K \int_0^1 \frac{u du}{\sqrt{1-u^2}} = 2K.
$$
This asymptotic result justifies the ordering of $\mu^{(K)}(\omega)$ given by (\ref{ordering-mass}) by
redefining $\omega_{\infty}$ if needed.

\section{Numerical approximation of positive single-lobe states}
\label{sec-numerics}

The analytical results on asymmetric, $K$-split, single-lobe states in Section \ref{sec-global-bifurcations}
were restricted to the region $p_0 \in (0,p_*)$, for which monotonicity results of Lemmas \ref{monotone-on-left} and \ref{monotone-T-minus}
were sufficient to guarantee that the $K$-split states satisfy (\ref{q-cond}) and are found from the system
(\ref{split-system}). In other words, the $K$ components are of the type (\ref{asym-sol-plus}) and $(N-K)$ components
are of the type (\ref{asym-sol-minus}).

Here we explore numerically the asymmetric, $K$-split, single-lobe states for the case $p_0 \in (p_*,1)$
in particular, near the bifurcation point $p_{\rm bif} \in (p_*,1)$ found in Section \ref{sec-global-stability}. Figure \ref{periods-T-plus-minus}
suggests that the graphs of $T_+(p_0,q_0)$ and $T_-(p_0,q_0)$ in $q_0$ do not intersect for $p_0 \in (p_*,1)$.
Therefore, the $K$-split single-lobe states may only be combinations of $K$ components of the type (\ref{asym-sol-plus})
and different $(N-K)$ components of the same type (\ref{asym-sol-plus}). Note that if all components
are of the same type (\ref{asym-sol-minus}), the boundary condition (\ref{NKcondition}) is not satisfied
since the left-hand side is negative and the right-hand side is positive.

\begin{figure}[htbp] %  figure placement: here, top, bottom, or page
   \centering
   \includegraphics[width=3in, height = 2in]{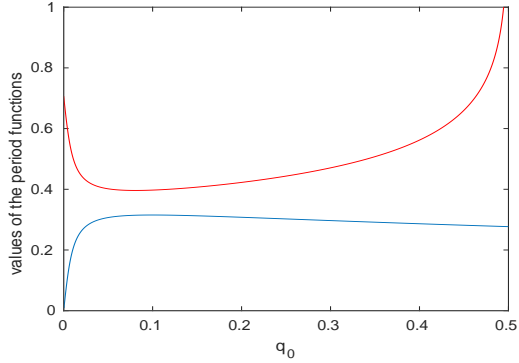}
   \caption{The blue and red lines show respectively the dependence of
   $T_+(p_0, q_0)$ and $T_-(p_0, q_0)$ in $q_0$ for $p_0 = 0.7078$.}
   \label{periods-T-plus-minus}
\end{figure}

Hence, we are looking for the asymmetric, $K$-split, single-lobe states from the roots of the following system:
\begin{equation}
\label{split-system-2}
\left\{ \begin{array}{l}
T_+(p_0, q_1) = T_+(p_0, q_N), \\
2 K q_1 + 2 (N-K)q_N = \sqrt{A(p_0)}
\end{array} \right.
\end{equation}
where $q_1 \neq q_N$ and $q_1,q_N \geq 0$.
Using Lemma \ref{behaviour-qmax}, for every $p_0 \in (p_*, p_{**})$, the period function $T_+(p_0, q_0)$
has the unique critical point $q_0 = \qmax (p_0)$, which corresponds to its maximum. Therefore, assuming
$q_1 > q_N$, the first equation in system (\ref{split-system-2}) yields the one-to-one function
\begin{equation}
\label{mapping-q1-qN}
(0, \qmax(p_0)) \ni q_N \mapsto q_1(q_N) \in (\qmax(p_0), \infty),
\end{equation}
for any $p_0 \in (p_*,p_{**})$. It remains to compute numerically the value of
$q_N \in (0, \qmax(p_0))$ for which the second equation in system
(\ref{split-system-2}) with $q_1(q_N)$ given by the mapping (\ref{mapping-q1-qN}) is satisfied.
Therefore, for $p_0 \in (p_*,p_{**})$, we construct the function $F:(0, \qmax(p_0)) \to \mathbb{R}$ defined as
\begin{equation}
\label{F-qN}
F(q_N) := K q_1(q_N) + (N-K) q_N
\end{equation}
for every $q_N \in (0, \qmax(p_0))$. The second equation in system (\ref{split-system-2}) is equivalent to
the equation $F(q_N) = \frac{1}{2}\sqrt{A(p_0)}$.

Figures \ref{fig-up-up-left1} and \ref{fig-up-up-left2} show the graph of the function $F$ defined by (\ref{F-qN})
in $q_N$ (left) and the asymmetric, $K$-split, single-lobe state constructed from the level curves on the $(u,v)$-plane (right)
for $p_0\in (p_*, p_\bif)$, $N = 3$ with $K = 1$ and $K = 2$ respectively. There exist exactly one value of
$q_N \in (0, \qmax(p_0))$ such that $F(q_N) = \frac{1}{2} \sqrt{A(p_0)}$ for both cases, which give
only one state $\Phi^{(1)}$ and $\Phi^{(2)}$ for this $p_0$.

\begin{figure}[htbp] %  figure placement: here, top, bottom, or page
   \centering
   \includegraphics[width=3in, height = 2in]{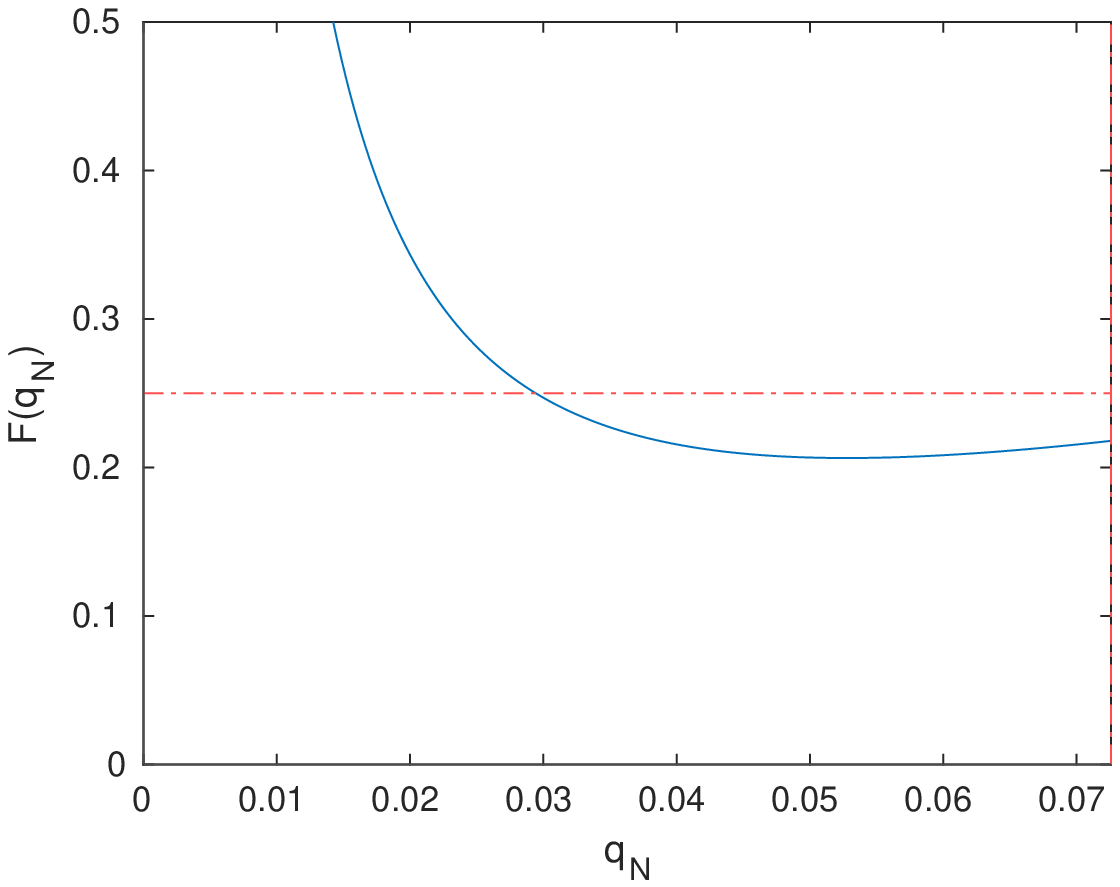}
    \includegraphics[width=3in, height = 2in]{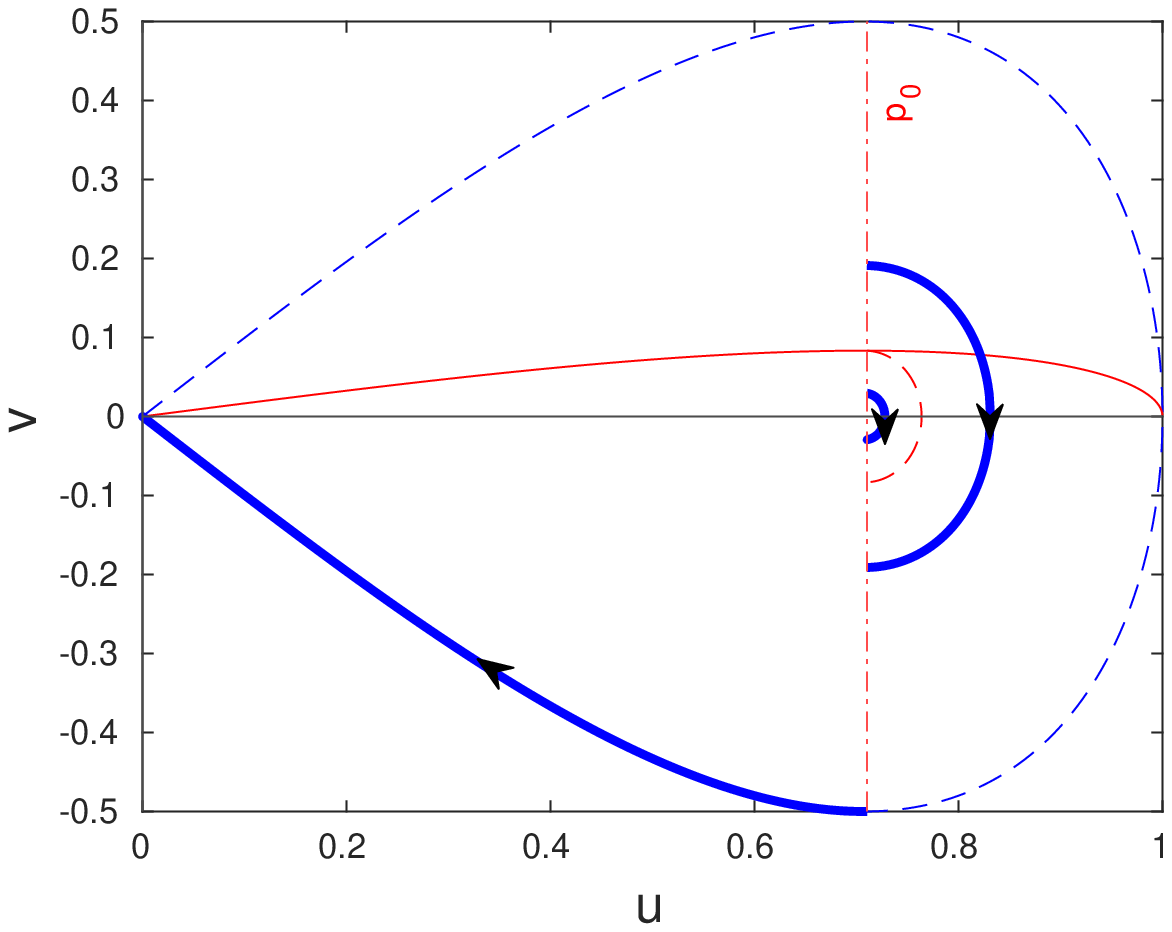}
   \caption{The graph of the function $F(q_N)$ (left)
   and the construction of the positive, asymmetric, $K$-split, single-lobe state on the $(u,v)$-plane (right)
   for $a=0.875$ ($p_0 = 0.7103 \in (p_*, p_\bif)$), $N = 3$, and $K = 1$. The red dashed horizontal line
   on the left panel corresponds to the value of $\frac{1}{2}\sqrt{A(p_0)}$.}
   \label{fig-up-up-left1}
\end{figure}

\begin{figure}[htbp] %  figure placement: here, top, bottom, or page
   \centering
   \includegraphics[width=3in, height = 2in]{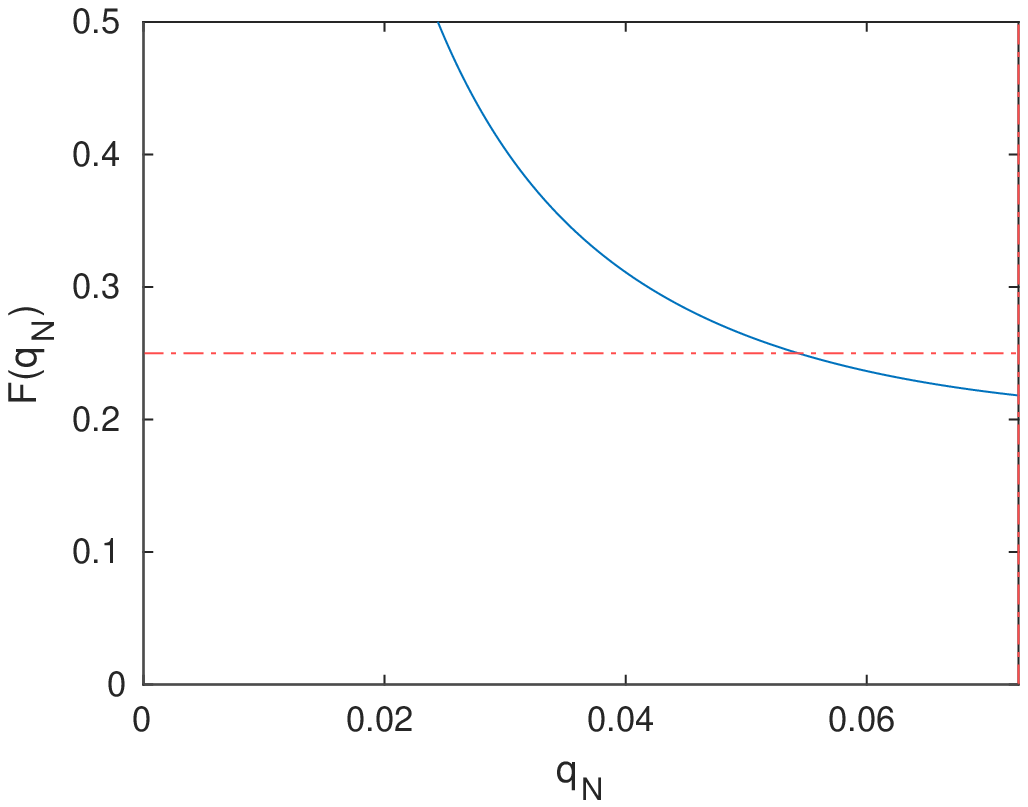}
    \includegraphics[width=3in, height = 2in]{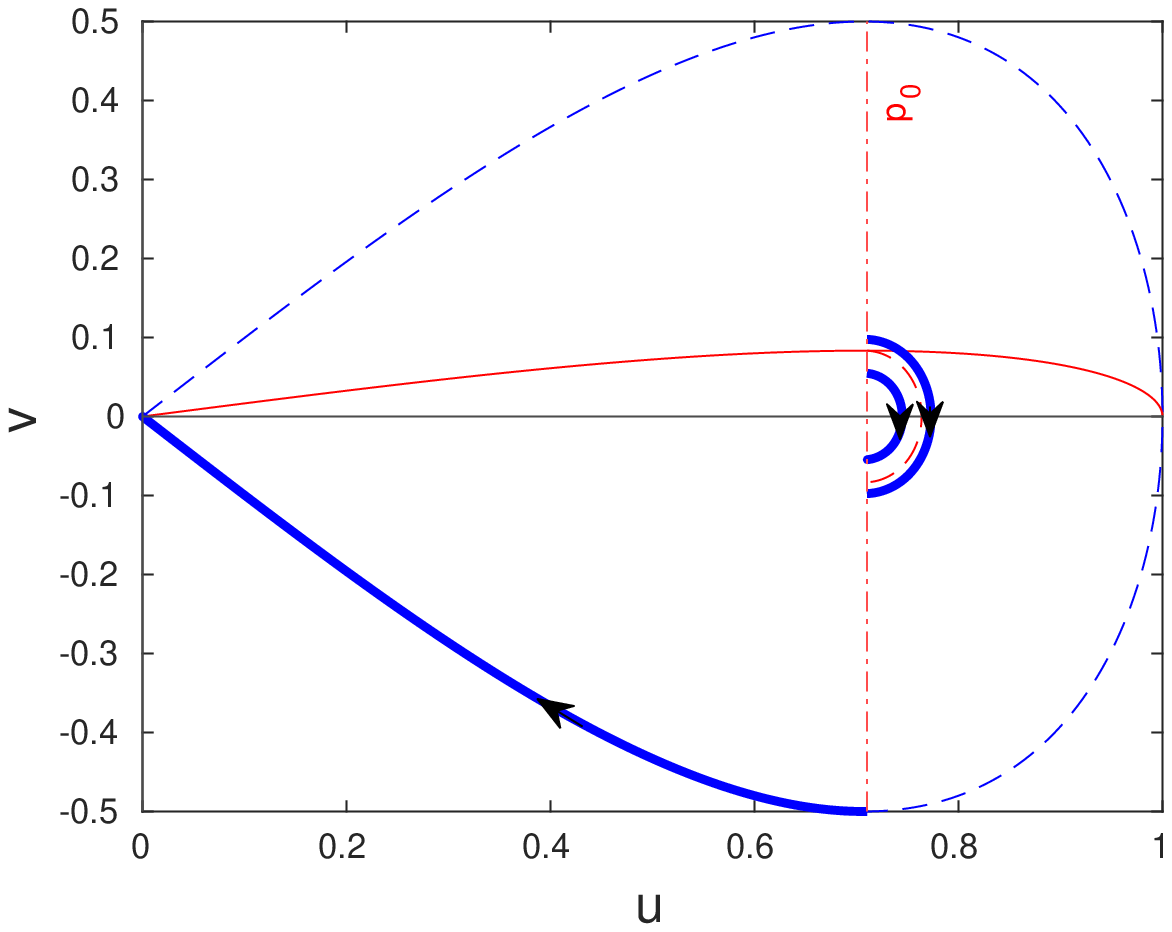}
   \caption{The same as in Figure \ref{fig-up-up-left1} but for $a=0.875$, $N = 3$, and $K = 2$.}
   \label{fig-up-up-left2}
\end{figure}

Figure \ref{fig-bif-right} shows the graph of the function $F$ in $q_N$
for $p_0\in (p_\bif,p_{**})$, $N = 3$, with $K = 1$ (left) and $K = 2$ (right).
For $K = 1$, there exist two values of
$q_N \in (0, \qmax(p_0))$ such that $F(q_N) = \frac{1}{2} \sqrt{A(p_0)}$, which give
two states $\Phi^{(1)}$ for this $p_0$. The two states constructed from the level curves on the $(u,v)$-plane
are shown on Fig. \ref{fig-up-up-right}. The coexistence of two states $\Phi^{(1)}$ for $p_0 \gtrsim p_{\rm bif}$
explains the fold bifurcation seen for the red line on the insert of Fig. \ref{fig-branches} (right).
On the other hand, there are no values of $q_N \in (0, \qmax(p_0))$ such that $F(q_N) = \frac{1}{2} \sqrt{A(p_0)}$
for $K = 2$. As a result, the state $\Phi^{(2)}$ only exists for $p_0 \lesssim p_{\rm bif}$, as on the green line shown on the insert of
Fig. \ref{fig-branches} (right).

\begin{figure}[htbp] %  figure placement: here, top, bottom, or page
   \centering
    \includegraphics[width=3in, height = 2in]{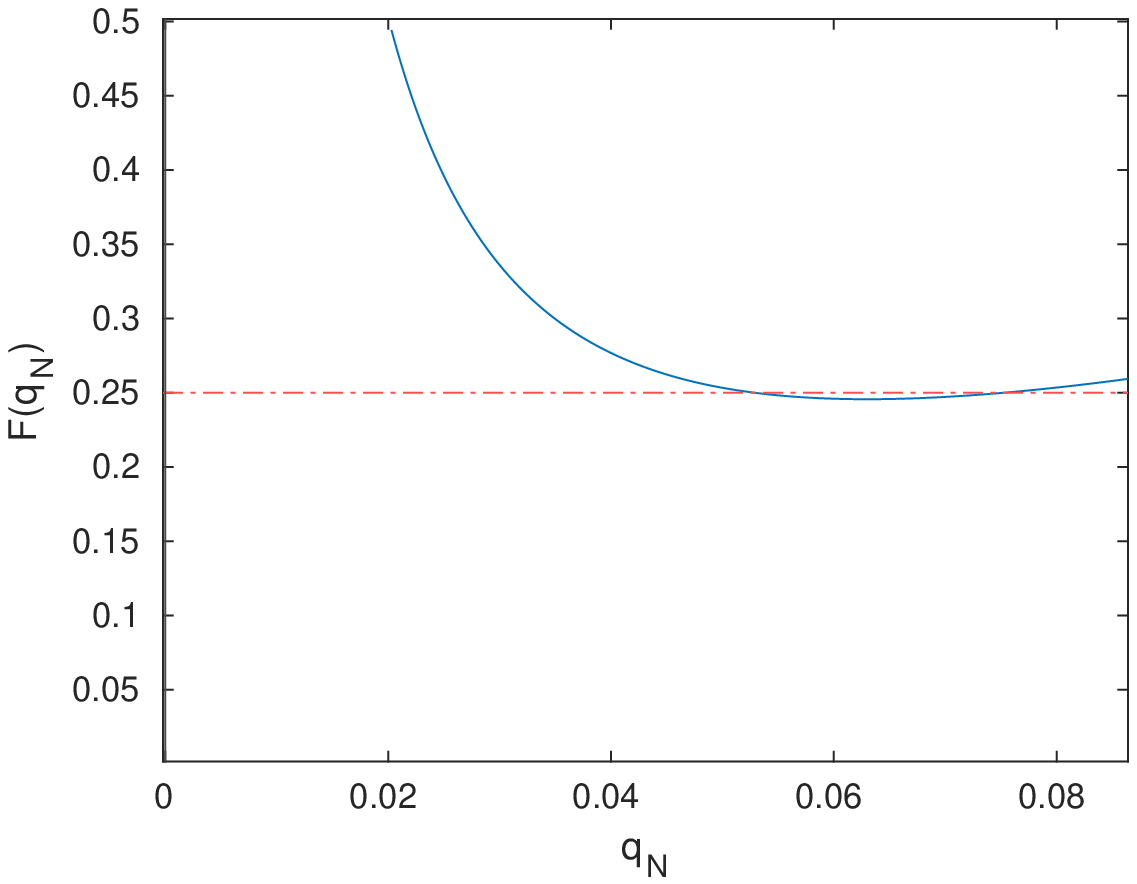}
   \includegraphics[width=3in, height = 2in]{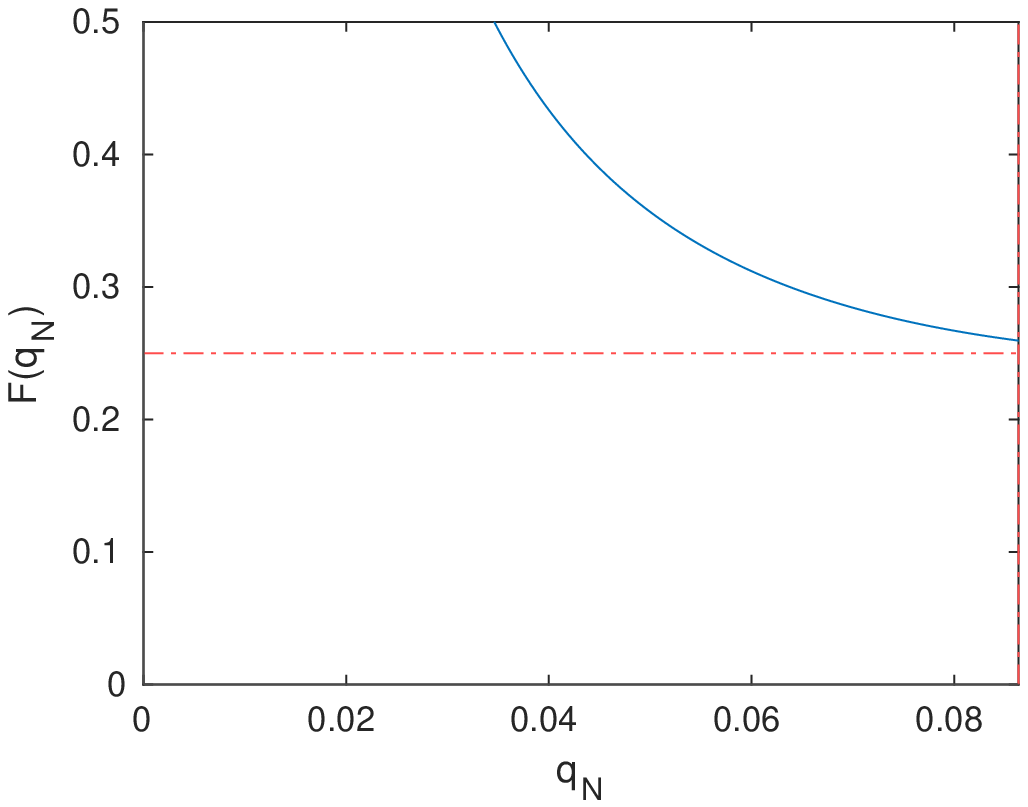}
   \caption{The graph of the function $F(q_N)$ for $a=0.8726$ ($p_0 = 0.7115 \in (p_\bif,p_{**})$), $N = 3$, with $K = 1$ (left)
   and $K = 2$ (right). The red dashed horizontal line on the left panel corresponds to the value of $\frac{1}{2}\sqrt{A(p_0)}$.}
   \label{fig-bif-right}
\end{figure}

\begin{figure}[htbp] %  figure placement: here, top, bottom, or page
   \centering
   \includegraphics[width=3in, height = 2in]{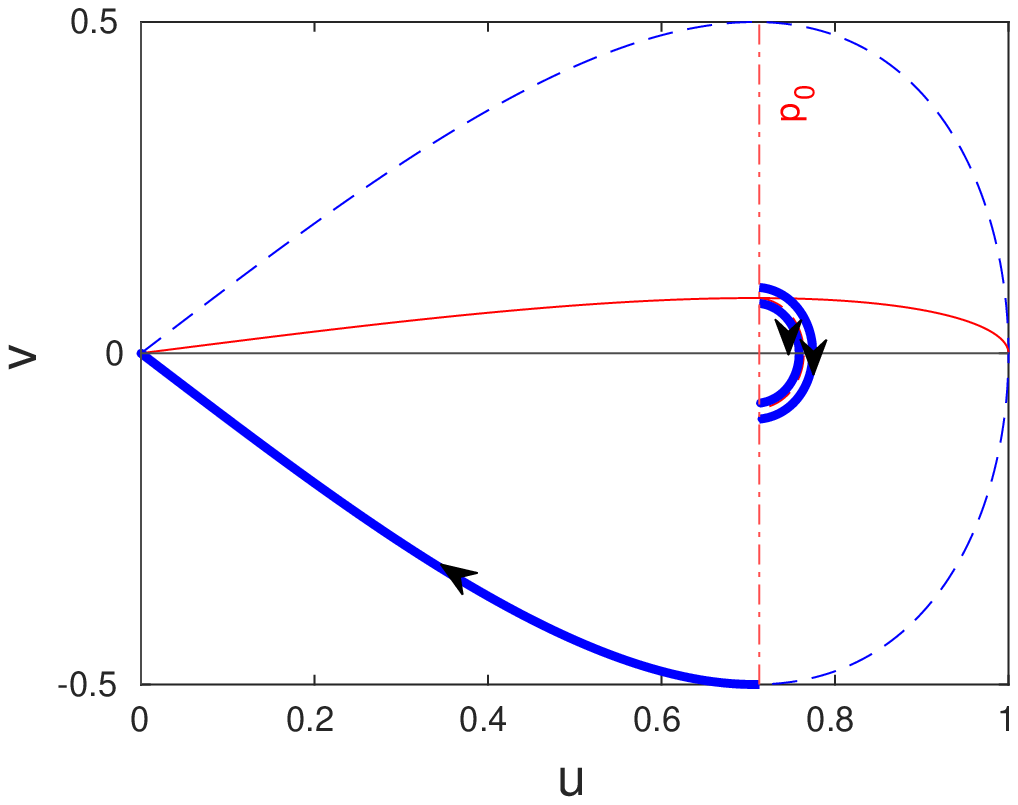}
    \includegraphics[width=3in, height = 2in]{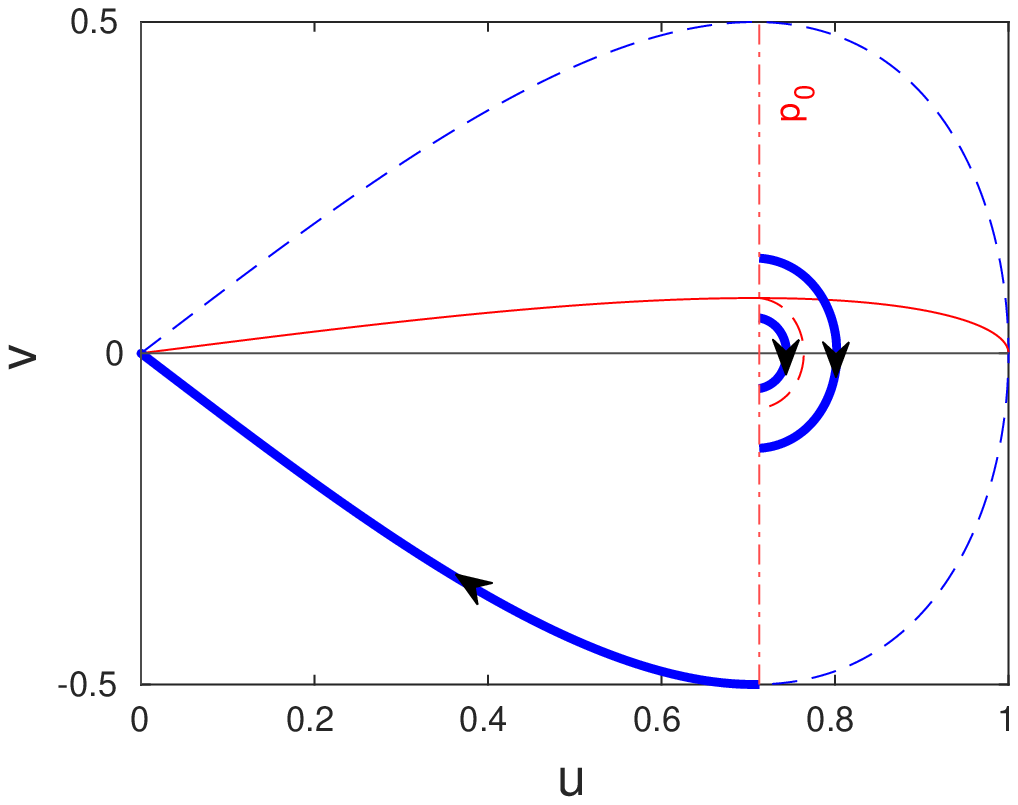}
   \caption{Construction of the positive, asymmetric, $K$-split, single-lobe state $\Phi^{(1)}$ on the $(u,v)$-plane
   for each of the two roots on Fig. \ref{fig-bif-right} (left) for $a=0.8726$, $N = 3$, and $K = 1$.}
   \label{fig-up-up-right}
\end{figure}

%\section{Further directions}
%\label{sec-conclusion}

\appendix

\section{Spectrum of $-\Delta$ in $L^2(\Gamma_N)$}
\label{appendix-spectrum}

Here we show that the spectrum of $-\Delta$ in $L^2(\Gamma_N)$ consists of continuous spectrum on $[0,\infty)$
and a set of embedded eigenvalues $\{ n^2 \}_{n \in \mathbb{N}}$ of multiplicity $N$ and
$\{ \left(n - \frac{1}{2}\right)^2 \}_{n \in \mathbb{N}}$ of multiplicity $N-1$.

We first look for the discrete spectrum of eigenvalues $\lambda$, for which there exists $\Phi \in H^2_{\rm NK}(\Gamma_N)$
such that $-\Delta \Phi = \lambda \Phi$. The discrete spectrum consists of two sets, depending whether $\phi_0 \equiv 0$
or $\phi_0 \neq 0$. If $\phi_0(x) = 0$ for every $x \in [0,\infty)$, then the general solutions
$$
\phi_j(x) = c_j \cos(\sqrt{\lambda} x) + d_j \sin(\sqrt{\lambda} x), \quad x \in [-\pi,\pi], \quad j \in \{1,\dots,N\}.
$$
satisfy $\phi_j(\pm \pi) = 0$ from the continuity boundary conditions in (\ref{KBC}). This yields
$$
\left\{ \begin{array}{l} c_j \cos(\pi \sqrt{\lambda}) = 0, \\
d_j \sin(\pi \sqrt{\lambda}) = 0, \end{array} \right.  \quad j \in \{1,\dots,N\}.
$$
From the derivative boundary condition in (\ref{KBC}), we have $\sum_{j=1}^N \left[ \phi_j'(\pi) - \phi_j'(-\pi) \right] = 0$
which yields
$$
\sqrt{\lambda} \sum_{j=1}^N c_j \sin(\pi \sqrt{\lambda}) = 0.
$$
If $c_j = 0$ for every $j$, then the eigenvalues correspond to the roots of $\sin(\pi \sqrt{\lambda})$,
which are located at $\{ n^2 \}_{n \in \mathbb{N}}$. Each eigenvalue has multiplicity $N$ since
coefficients $(d_1,\dots,d_N)$ are independent of each other.

If $d_j = 0$ for every $j$, then the eigenvalues correspond to the roots of $\cos(\pi \sqrt{\lambda})$,
which are located at $\{ \left(n - \frac{1}{2}\right)^2 \}_{n \in \mathbb{N}}$.
In addition, coefficients $(c_1,\dots,c_N)$ satisfy the constraint $\sum_{j=1}^N c_j = 0$ which follows
from the derivative boundary condition. Therefore, each eigenvalue has multiplicity $N-1$.

The second part of the discrete spectrum, if it is non-empty, correspond to $\phi_0 \neq 0$.
Since the half-line tail is semi-infinite, we have $\phi_0 \in H^2(0,\infty)$
if and only if $\lambda < 0$, for which we obtain
$$
\phi_0(x) = c_0 e^{-\sqrt{|\lambda|} x}, \quad x \in [0,\infty),
$$
with some $c_0$ and
$$
\phi_j(x) = c_j \cosh(\sqrt{|\lambda|} x) + d_j \sinh(\sqrt{|\lambda|} x), \quad x \in [-\pi,\pi], \quad j \in \{1,\dots,N\}.
$$
From the continuity boundary conditions in (\ref{KBC}), we have $\phi_j(\pm \pi) = c_0$ which yield
$$
\left\{ \begin{array}{l} c_j \cosh(\pi \sqrt{|\lambda|}) = c_0, \\
d_j \sinh(\pi \sqrt{|\lambda|}) = 0, \end{array} \right.  \quad j \in \{1,\dots,N\}.
$$
Hence, $d_j = 0$ for every $j$ and $c_j$ are uniquely expressed for every $j$ by $c_0$ and $\lambda < 0$.
From the derivative boundary condition in (\ref{KBC}), we have
$\sum_{j=1}^N \left[ \phi_j'(\pi) - \phi_j'(-\pi) \right] = -c_0 \sqrt{|\lambda|}$
which yields
$$
\sqrt{\lambda} c_0 \left( 2N \tanh(\pi \sqrt{|\lambda|}) + 1 \right) = 0.
$$
This equation yields $c_0 = 0$ since $\tanh(\pi \sqrt{|\lambda|}) > 0$. Hence, the second
part of the discrete spectrum is empty.

Finally, the continuous part of the spectrum of $-\Delta$ in $L^2(\Gamma_N)$
is due to the non-compact tail and it is equivalent to the spectrum of $-\Delta : H^2(0,\infty) \subset L^2(0,\infty) \to L^2(0,\infty)$
which is located at $[0,\infty)$. Hence, all eigenvalues of the discrete spectrum of $-\Delta$ in $L^2(\Gamma_N)$
are embedded into the continuous spectrum.

\section{The symmetric state $\Phi$ for small mass}
\label{appendix-small}

Here we show that there exists $\omega_0 < 0$ such that for every $\omega \in (\omega_0,0)$,
the positive single-lobe symmetric state $\Phi$ of Theorem \ref{global-existence}
is the ground state of the constrained minimization problem (\ref{min}) for small $\mu$.

Let us parameterize the negative values of $\omega$ by $\omega = -\epsilon^2$ with $\epsilon > 0$
and use the scaling transformation (\ref{scaling-transform}). By using the shifted NLS soliton
(\ref{soliton}) for $u_0$ and the symmetry condition (\ref{sym-state-scaled}) for $u_1 = \dots = u_N$,
we obtain the boundary-value problem:
\begin{equation}
\label{bvp-small-mass}
\left\{ \begin{array}{l} -u_1''(z) + u_1(z) - 2 u_1(z)^3 = 0, \quad z \in (-\pi \epsilon,\pi \epsilon),  \\
u_1(- \pi \epsilon) = u_1(\pi \epsilon) = p_0, \\
u_1'(- \pi \epsilon) = -u_1'(\pi \epsilon) = q_0,
\end{array} \right.
\end{equation}
where $p_0 = {\rm sech}(a)$ and $q_0 = \frac{1}{2N} {\rm sech}(a) \tanh(a)$
are defined by $a > 0$.

Since the support of $[-\pi \epsilon,\pi \epsilon]$ shrinks to zero as $\epsilon \to 0$,
the power series solution provides an asymptotic expansion in powers of $\epsilon$:
$$
u_1(z) = u_1(0) + \frac{1}{2} u_1(0) \left[ 1 - 2 |u_1(0)|^2 \right] z^2 + \mathcal{O}(z^4), \quad z \in [-\pi \epsilon,\pi \epsilon].
$$
The continuity and derivative boundary conditions imply that
\begin{equation*}
\left\{ \begin{array}{l}
p_0 = u_1(0) + \mathcal{O}(\epsilon^2), \\
p_0 \sqrt{1-p_0^2} = 2N \pi \epsilon u_1(0) \left[ 1 - 2 |u_1(0)|^2  \right] + \mathcal{O}(\epsilon^3),
\end{array} \right.
\end{equation*}
which admits a unique asymptotic solution with $u_1(0) = p_0 + \mathcal{O}(\epsilon^2)$ and
$p_0 = 1 - 2N^2 \pi^2 \epsilon^2 + \mathcal{O}(\epsilon^4)$ or equivalently,
$a = 2N \pi \epsilon + \mathcal{O}(\epsilon^3)$ as $\epsilon \to 0$.

We compute asymptotically the mass $\mu(\omega) = Q(\Phi(\cdot,\omega))$ as follows:
\begin{eqnarray*}
\mu(\omega) = 2N \epsilon \int_0^{\pi \epsilon} u_1^2(z) dz + \epsilon \left[ 1 - \tanh(a) \right] = \epsilon + \mathcal{O}(\epsilon^2) \quad \mbox{\rm as} \quad \epsilon \to 0.
\end{eqnarray*}
Similarly, we compute asymptotically the energy $\eta(\omega) := E(\Phi(\cdot,\omega))$ as follows:
\begin{eqnarray*}
\eta(\omega) & = & 2N \epsilon^3 \int_0^{\pi \epsilon} \left[ [u_1'(z)]^2  - u_1(z)^4 \right] dz
+ \epsilon^3 \left[ \frac{2}{3} \tanh(a) {\rm sech}^2(a) - \frac{1}{3} + \frac{1}{3} \tanh(a) \right] \\
& = & -\frac{1}{3} \epsilon^3 + \mathcal{O}(\epsilon^4)  \quad \mbox{\rm as} \quad \epsilon \to 0.
\end{eqnarray*}
Therefore, $\mathcal{E}_{\mu} = -\frac{1}{3} \mu^3 + \mathcal{O}(\mu^4)$, which implies that
$\mathcal{E}_{\mu}$ belongs to the interval (\ref{bounds-on-E}). By Theorem 2.2 of \cite{AdamiCV},
this implies that $\Phi$ is a ground state of the constrained minimization problem (\ref{min}) for small $\mu$.

\section{The asymmetric state $\Phi^{(K=1)}$ for large mass}
\label{appendix-large}

Here we show that there exists $\omega_{\infty} < 0$ such that for every $\omega \in (-\infty,\omega_{\infty})$,
the positive single-lobe asymmetric state $\Phi^{(K=1)}$ of Theorem \ref{global-bifurcations}
is not the ground state of the constrained minimization problem (\ref{min}) for large $\mu$ with $N \geq 2$.

In the limit $\omega \to -\infty$ (or $\epsilon \to \infty$ after rescaling),
the solution $\Phi^{(K=1)}$ of Theorem \ref{global-bifurcations}
consists of the truncated NLS soliton in one component, say in $u_1$,
and exponentially small solution in the other components $(u_2,\dots,u_N)$ and $u_0$.
The truncated NLS soliton is given exactly by either the {\em cnoidal} wave
\begin{equation}
\label{cnoidal}
u_1(z) = \frac{k}{\sqrt{2k^2-1}} {\rm cn}\left(\frac{z}{\sqrt{2k^2-1}};k\right), \quad z \in \mathbb{R},
\end{equation}
or the {\em dnoidal} wave
\begin{equation}
\label{dnoidal}
u_1(z) = \frac{1}{\sqrt{2-k^2}} {\rm dn}\left(\frac{z}{\sqrt{2-k^2}};k\right), \quad z \in \mathbb{R},
\end{equation}
where $k \in (0,1)$ is the elliptic modulus and $cn$, $dn$ are Jacobian elliptic functions.
The parameter $k$ is selected uniquely near $k = 1$, where $u_1(z) = {\rm sech}(z)$. In fact, the Jacobi real
transformation $k \mapsto k^{-1}$ maps the {\em cnoidal} wave (\ref{cnoidal}) with $k < 1$ to the
{\em dnoidal} wave (\ref{dnoidal}) with $k > 1$, therefore, it is sufficient to consider the single analytic
expression (\ref{dnoidal}) for $k$ near $1$.

The Dirichlet and Neumann data at the end points of $[-\pi \epsilon,\pi \epsilon]$ are given by
$$
p_0 = u_1(-\pi \epsilon) = \frac{1}{\sqrt{2-k^2}} {\rm dn}\left(\frac{\pi \epsilon}{\sqrt{2-k^2}};k\right),
$$
and
$$
q_0 = u_1'(-\pi \epsilon) = \frac{k^2}{2-k^2} {\rm sn}\left(\frac{\pi \mu}{\sqrt{2-k^2}};k\right)
{\rm cn}\left(\frac{\pi \mu}{\sqrt{2-k^2}};k\right).
$$
Applying the main result of \cite{BMP} on the looping edge to the flower graph $\Gamma_N$,
it follows that $k$ is found from the nonlinear equation $2 q_0 = (2N-1) p_0 + R_{\mu}(p_0,q_0)$,
where $R_{\mu}(p_0,q_0)$ denotes the remainder terms which are exponentially smaller than the linear terms in $p_0$ and $q_0$.
By Theorem 4.3 in \cite{BMP}, $k$ is found uniquely in the form
$$
k = 1 + 8 \frac{2N-3}{2N+1} e^{-2\pi \epsilon} + \mathcal{O}(e^{-4 \pi \epsilon}) \quad \mbox{\rm as} \quad \epsilon \to \infty,
$$
whereas the mass $\mu(\omega) = Q(\Phi(\cdot,\omega))$ and energy $\eta(\omega) := E(\Phi(\cdot,\omega))$
are given asymptotically by
\begin{eqnarray*}
\mu(\omega) = 2 \epsilon - 16 \pi \frac{2N-3}{2N+1} \epsilon^2 e^{-2 \pi \epsilon} + \mathcal{O}(\epsilon e^{-2\pi \epsilon})
\quad \mbox{\rm as} \quad \epsilon \to \infty.
\end{eqnarray*}
and
\begin{eqnarray*}
\eta(\omega) = -\frac{2}{3} \epsilon^3 + \mathcal{O}(\epsilon^4 e^{-2\pi \epsilon}) \quad \mbox{\rm as} \quad \epsilon \to \infty.
\end{eqnarray*}
By the Comparison Lemma (Lemma 5.2 in \cite{BMP}), $\Phi^{(K=1)}$ is not the ground state for $N \geq 2$ which follows from
$\mu(\omega) < 2 \epsilon$. On the other hand, $\Phi^{(K=1)} = \Phi$ is the ground state for $N = 1$, for which $\mu(\omega) > 2 \epsilon$,
the latter conclusion agrees with the result following from Corollary 3.4 and Fig. 4 of \cite{AdamiJFA}. In both cases $N \geq 2$ and $N = 1$, we have $\mathcal{E}_{\mu} \sim -\frac{1}{12} \mu^3$ as $\mu \to \infty$,
which implies that the branch of $\Phi^{(K=1)}$ on the $(\mu,\eta)$ plane approaches the upper bound
of the interval (\ref{bounds-on-E}) from outside for $N \geq 2$ and from inside for $N = 1$,
in agreement with Figures \ref{fig-N-1} and \ref{fig-N-2}.


\begin{thebibliography}{100}

\bibitem{AdamiCV} R. Adami, E. Serra, and P. Tilli, ``NLS ground states on graphs",
Calc. Var. {\bf 54} (2015), 743--761.

\bibitem{AdamiJFA} R. Adami, E. Serra, and P. Tilli, ``Threshold phenomena and existence results for NLS ground states
on graphs", J. Funct. Anal. {\bf 271} (2016), 201--223.

\bibitem{AST17} R. Adami, E. Serra, P. Tilli, ``Negative energy ground states for the $L^2$-critical NLSE
on metric graphs", \newblock{\em Comm. Math. Phys.} {\bf 352} (2017), 387--406.

\bibitem{AST} R. Adami, E. Serra, and P. Tilli, ``Multiple positive bound states
for the subcritical NLS equation on metric graphs", Calc. Var. {\bf 58} (2019), 5 (16 pages).

\bibitem{BKKM} G. Berkolaiko, J.B. Kennedy, P. Kurasov and
D. Mugnolo, ``Surgery principles for the spectral analysis of quantum graphs", Trans. AMS {\bf 372} (2019), 5153--5197.

\bibitem{BK} G. Berkolaiko and P. Kuchment, ``Introduction to Quantum Graphs (Mathematical Surveys and Monographs, vol 186)", Providence, RI: American Mathematical Society (2013).

\bibitem{BMP} G. Berkolaiko, J. Marzuola and D.E. Pelinovsky,
``Edge-localized states on quantum graphs in the limit of large mass",
arXiv:1910.03449 (2019)

\bibitem{CFN} C.~Cacciapuoti, D.~Finco, and D.~Noja, Topology induced bifurcations
for the NLS on the tadpole graph, {\em Phys.Rev. E} {\bf 91} (2015), 013206.

\bibitem{Vill1} A. Garijo and J. Villadelprat, ``Algebraic and analytical tools for the study of the period function",
J. Diff. Eqs. {\bf 257} (2014), 2464--2484.

\bibitem{GP17} A. Geyer and D.E. Pelinovsky, ``Spectral stability of periodic waves in the generalized
reduced Ostrovsky equation",  Lett. Math. Phys. {\bf 107} (2017), 1293--1314.

\bibitem{Vill2} A. Geyer and J. Villadelprat, ``On the wave length of smooth periodic traveling waves
of the Camassa--Holm equation", J. Diff. Eqs. {\bf 259} (2015), 2317--2332.

\bibitem{Exner} P. Exner and H. Kovarik, {\em Quantum waveguides} (Springer, Cham--Heidelberg--New York--Dordrecht--London, 2015).

\bibitem{KP2} A. Kairzhan and D.E. Pelinovsky, ``Spectral stability of shifted states on star graphs",
J. Phys. A: Math. Theor. {\bf 51} (2018) 095203 (23 pages).

\bibitem{Noja} D.Noja, Nonlinear Schr\"odinger equation on graphs: recent results
and open problems, \newblock{\em Phil. Trans. R. Soc. A}, {\bf 372} (2014), 20130002 (20 pages).

\bibitem{MNP} D.Noja and D.E. Pelinovsky, ``Standing waves of the quintic NLS equation on the tadpole graph",
arXiv:2001.00881 (2020)

\bibitem{NPS} D. Noja, D. Pelinovsky, and G. Shaikhova, ``Bifurcations and stability of standing waves
in the nonlinear Schr\"{o}dinger equation on the tadpole graph",  Nonlinearity {\bf 28} (2015), 2343--2378.

\bibitem{Teschl} G. Teschl, {\em Ordinary Differential Equations and Dynamical Systems},
Graduate Studies in Mathematics {\bf 140} (AMS, Providence, RI, 2012).

\end{thebibliography}
\end{document}